\documentclass[12pt]{amsart}
\usepackage{amsmath} 
\usepackage{amsthm} 
\usepackage{amssymb} 
\usepackage{graphicx} 
\usepackage[dvips,letterpaper,margin=1in,bottom=0.7in]{geometry}
\usepackage{tensor}
\usepackage{upgreek}
\usepackage{tikz}
\usetikzlibrary{automata,positioning}
\usetikzlibrary{decorations.pathmorphing,arrows.meta}
\usetikzlibrary{intersections,calc}
\usepackage{tikz-cd}
\usepackage{bbm}
\usepackage{float}
\usepackage[backref=page]{hyperref}
\usepackage{cleveref}
\usepackage{quiver}
\usepackage{aliascnt}

\allowdisplaybreaks

\newcommand{\defi}[1]{\textsf{#1}} 

\DeclareMathOperator{\ord}{ord}

\DeclareMathOperator{\im}{Im}

\DeclareMathOperator{\Hom}{Hom}

\DeclareMathOperator{\Pic}{Pic}

\DeclareMathOperator{\rk}{rk}

\newcommand{\Mexact}{\mathcal{M}_{g,2}(=k)}

\newcommand{\Oc}{\mathcal{O}_C}

\newcommand{\abs}[1]{\left| #1 \right|}
\providecommand{\oo}{\overline}
\usepackage{enumitem}
 
\let\baraccent=\= 
\renewcommand{\=}[1]{\stackrel{#1}{=}} 

\providecommand{\xr}{\xrightarrow}

\providecommand{\ZZ}{\mathbb{Z}}

\providecommand{\PP}{\mathbb{P}}

\providecommand{\mL}{{\mathcal L}}
\providecommand{\mP}{{\mathcal P}}
\providecommand{\mM}{{\mathcal M}}
\providecommand{\mO}{{\mathcal O}}

\providecommand{\na}{\boldsymbol{a}}
\providecommand{\nb}{\boldsymbol{b}}
\providecommand{\mf}{\mathfrak{f}}
\providecommand{\B}{\big}
\providecommand{\BB}{\Big}

\DeclareMathOperator{\argmax}{argmax}

\DeclareMathOperator{\inv}{inv}
\DeclareMathOperator{\ASP}{ASP}
\DeclareMathOperator{\sgn}{sgn}
\DeclareMathOperator{\Ess}{Ess}
\DeclareMathOperator{\coker}{coker}

\providecommand{\wt}{\widetilde}

\setlist[enumerate]{noitemsep, topsep=0pt}

\newtheorem{thm}{Theorem}[section]
\crefname{thm}{Theorem}{Theorems}

\newaliascnt{lem}{thm}
\newtheorem{lem}[lem]{Lemma}
\aliascntresetthe{lem}
\crefname{lem}{Lemma}{Lemmas}

\newaliascnt{cor}{thm}
\newtheorem{cor}[cor]{Corollary}
\aliascntresetthe{cor}
\crefname{cor}{Corollary}{Corollaries}

\newaliascnt{prop}{thm}
\newtheorem{prop}[prop]{Proposition}
\aliascntresetthe{prop}
\crefname{prop}{Proposition}{Propositions}

\theoremstyle{definition}
\newaliascnt{dfn}{thm}
\newtheorem{dfn}[dfn]{Definition}
\aliascntresetthe{dfn}
\crefname{dfn}{Definition}{Definitions}

\newaliascnt{rmk}{thm}
\newtheorem{rmk}[rmk]{Remark}
\aliascntresetthe{rmk}
\crefname{rmk}{Remark}{Remarks}

\newaliascnt{example}{thm}
\newtheorem{example}[example]{Example}
\aliascntresetthe{example}
\crefname{example}{Example}{Examples}

\crefname{figure}{Figure}{Figures}
\crefname{section}{Section}{Sections}

\crefname{app}{Appendix}{Appendices}

\usepackage{mathtools}

\DeclarePairedDelimiter\floor{\lfloor}{\rfloor}
\usepackage{cancel}

\newcommand\numberthis{\addtocounter{equation}{1}\tag{\theequation}}

\title{Brill-Noether theory for totally ramified covers of the projective line}
\author[D. Aggarwal]{Daksh Aggarwal}
\email{daksh\_aggarwal@brown.edu}
\address{Department of Mathematics, Brown University}
\begin{document}
\begin{abstract}
 Given a curve $C$ that is a degree $k$ cover $C \to \PP^1$ totally ramified at two points $p$ and $q$, we can seek to understand the space of degree $d$ line bundles on $C$ with prescribed ramification at $p$ and $q$. The corresponding subschemes of $\Pic^d(C)$ are called \textit{transmission loci} and are parameterized via elements of the (extended) $k$-affine symmetric group $\wt{\Sigma}_k$. Transmission loci provide a refinement of the splitting loci that have recently been extensively studied for $k$-gonal curves. Pflueger has conjectured analogues of the classic Brill-Noether theorem should hold for transmission loci. In this paper we prove Pflueger's conjectures.
\end{abstract}

\maketitle

\section{Introduction}
A fundamental problem in algebraic geometry is to understand the space of maps from an algebraic curve $C$ to projective space and the study of this problem has led to the rich area of Brill-Noether theory. Given the correspondence between maps $C \to \PP^r$ of degree $d$ and degree $d$ line bundles $\mL \in \Pic^d(C)$ along with an $(r+1)$-dimensional basepoint-free subspace $V \subseteq H^0(C,\mL)$ of global sections, the main object of study in classical Brill-Noether theory is
\begin{equation*}
    W^r_d(C) = \{\mL \in \Pic^d(C): h^0(C, \mL) \geq r+1\}.
\end{equation*}
When $C$ is a \textit{general} curve of genus $g$, several papers from the 1980s give us insight into the core geometry of $W^r_d(C)$. Griffiths$-$Harris \cite{griffiths_variety_1980} proved that $W^r_d(C)$ has the expected dimension $\rho(g,r,d)\coloneqq g-(r+1)(g+r-d) $; Gieseker \cite{gieseker_stable_1982} proved that $W^r_d(C)$ is smooth away from $W^{r+1}_d(C)$; Fulton$-$Lazarsfeld \cite{fulton_connectedness_1981} proved that $W^r_d(C)$ is irreducible if $\rho(g,r,d) > 0$. These results together form the celebrated Brill-Noether theorem.

Curves arising in natural contexts are rarely general, often admitting explicit descriptions due to special maps to projective space. For example, when $C$ is a $k$-gonal curve, the covering map $f: C \to \PP^1$ is special when $k < \lfloor (g+3)/2 \rfloor$, as the general Brill–Noether theorem predicts $W^1_k(C)=\varnothing$. This motivates studying $W^r_d(C)$ for a \textit{general degree $k$ cover of $\PP^1$}. The analysis is challenging because $W^r_d(C)$ is typically reducible and its components need not be equidimensional \cite{coppens_varieties_2002, pflueger_brillnoether_2017}. These components can be classified by the isomorphism class of \(f_*\mL\), equivalently by its \textit{splitting type} \(\vec{e}\in \ZZ^k\) \cite{cook-powell_components_2022, larson_refined_2021}. The splitting type refines the data $(r,d)$ and defines the \textit{splitting locus}
\begin{equation*}
    W^{\vec{e}}(C) \coloneqq \{\mL\in \Pic^d(C): f_{*}\mL \cong \mO_{\PP^1}(\vec{e})\text{ or a specialization thereof}\}.
\end{equation*}
Recently, E.\ Larson$-$H.\ Larson$-$Vogt \cite{larson_global_2025} prove complete analogues of the general Brill-Noether theorem for $W^{\vec{e}}(C)$ for any general $k$-gonal curve.

In this paper, we study the Brill-Noether theory of covers of $\PP^1$ totally ramified at two points. The results about $W^{\vec{e}}(C)$ proved in \cite{larson_global_2025} extend to this setting, since their degeneration argument already uses such covers \cite[Remark 2]{larson_global_2025}. However, when the covering map $f:C\to \PP^1$ has two points $p$ and $q$ of total ramification, i.e., $\mO_C(kp) \cong f^{*}\mO_{\PP^1}(1)\cong \mO_{C}(kq)$, we can seek to understand the more refined locus of degree $d$ line bundles on $C$ with prescribed ramification at the special points $p$ and $q$, i.e., line bundles with specified values of $h^0(C,\mL(ap+bq))$ for all $a,b\in \ZZ$. By work of Pflueger \cite{pflueger_extended_2022}, we can prescribe ramification at $p$ and $q$ via a choice of a permutation $\tau$ in the \textit{(extended) $k$-affine symmetric group} $\wt{\Sigma}_k$, which means that $\tau$ is a bijection $\ZZ\to \ZZ$ such that $\tau(n+k) = \tau(n)+k$ for all $n$. Such a permutation $\tau$ defines a $\tau$-\textit{transmission locus}
\begin{equation*}
    W^{\tau}(C,p,q) = \{\mL \in \Pic^d(C): h^0(C,\mL(ap+bq)) \geq \#\{n \geq -b: \tau(n) \leq a\} \text{ for all } a,b\in \ZZ\}.
\end{equation*}
Every splitting locus $W^{\vec{e}}(C)$ is a transmission locus $W^{\tau(\vec{e})}$ for some permutation $\tau(\vec{e})$ (since a splitting type $\vec{e}$ can be prescribed using $k$ inequalities for $a\in\{-e_1k,\dots,-e_kk\}$ and $b=0$ above), but the converse is false. Indeed, up to a pre- and post-translation, $\tau(\vec{e})$ is the same as the permutation $w(\vec{e})$ constructed in \cite[Thm. 1.4]{larson_global_2025} (see \cref{rmk:llv} for more details). Studying $W^{\tau}(C,p,q)$ thus leverages the full combinatorics of $\wt{\Sigma}_k$.

Pflueger \cite{pflueger2026transmissionpermutationsdemazureproducts} has conjectured that analogues of Brill-Noether theorem should hold for the locus $W^{\tau}(C,p,q)$. The expected codimension of $W^{\tau}(C,p,q)$ in $\Pic^d(C)$ is $\inv_k(\tau)$, the number of $k$-inversions in $\tau$ modulo $k$ (\cref{dfn:k_inversions}). Pflueger has shown that the expected dimension serves as an upper bound for $\dim W^{\tau}(C,p,q)$ \cite[Thm.\ 1.7]{pflueger2026transmissionpermutationsdemazureproducts}.  We establish that the expected dimension also serves as a lower bound by proving a new regeneration theorem along with smoothness results, thus proving part of the conjectured analogues:

\begin{thm}\label{thm:main}
  Fix $\tau \in \wt{\Sigma}_{k}$ and let $C\to \PP^1$ be a general degree $k$, genus $g$ cover totally ramified at two points $p$ and $q$ such that $k$ is the smallest positive integer satisfying $\mO_C(k(p-q)) \cong \mO_C$.
    \begin{enumerate}
        \item The variety $W^{\tau}(C,p,q)$ is non-empty if and only if $g \geq \inv_k(\tau)$. If $W^{\tau}(C,p,q)$ is non-empty, then $\dim W^{\tau}(C,p,q) = g - \inv_k(\tau)$. 
        \item If $g=\inv_k(\tau)$, then the number of points in $W^{\tau}(C,p,q)$ is the number of reduced words for $\tau + d-g$ in $\Sigma_k$.
        \item $W^{\tau}(C,p,q)$ is normal and Cohen–Macaulay, and is smooth away from the union of transmission loci $W^{\tau'}(C,p,q) \subset W^{\tau}(C,p,q)$ having codimension 2 or more.
        \item When $g > \inv_{k}(\tau)$, $W^{\tau}(C,p,q)$ is irreducible.
        \item When the characteristic of the ground field does not divide $k$ and $g \geq \inv_{k}(\tau)$, the universal $\mathcal{W}^{\tau}$ has a unique component dominating a component of $\mathcal{H}_{k,g,2}$ that consists of twice-marked curves $(C,p,q)$ such that $k$ is the smallest positive integer satisfying $\mO_C(k(p-q)) \cong \mO_C$.
    \end{enumerate}
\end{thm}
\begin{rmk}\label{rmk:char_assumption}
    By a \textit{general} degree $k$, genus $g$ cover totally ramified at two points $p$ and $q$ we mean a general point of the component of the Hurwitz space $\mathcal{H}_{k,g,2}$ on which the degeneration family $\mathcal{X}$ we use exists (see \cref{sec:degeneration}). \textit{Assume now the characteristic of the ground field is 0 or greater than $k$}. Then, in \Cref{app:irreducibility} we show that if $g> 0 $, then $\mathcal{H}_{k,g,2}$ has exactly $\sigma_0(k)-1$ irreducible components. In particular, the locus of covers in $\mathcal{H}_{k,g,2}$ that do not factor through any nontrivial degree-$d$ cover of $\PP^1$ with $d \mid k$ is irreducible, and our theorem applies to a general point of this component. This is not a genuine limitation: if $(C,p,q)$ is a general point of another component of $\mathcal{H}_{k,g,2}$, corresponding to some proper divisor $d \mid k$, then \cref{thm:main} can be applied by letting $k = d$. 
\end{rmk}

A permutation $\tau \in \wt{\Sigma}_k$ is determined by the window $(\tau(0),\tau(1),\dots,\tau(k-1))$ and is called \defi{sorted} when the window is increasing. 
    Our \cref{thm:main} naturally also applies to general degree $k$ covers totally ramified at a single point $p$. This problem setting corresponds to studying the one-point transmission locus
   \begin{equation*}
    W^{\tau}(C,p) = \{\mL \in \Pic^d(C): h^0(C,\mL(ap)) \geq \#\{n \geq 0: \tau(n) \leq a\} \text{ for all } a\in \ZZ\}.
\end{equation*}
So, ramification at a single point can be prescribed using sorted permutations in $\wt{\Sigma}_k$ and they are in bijective correspondence with the set of $k$-cores. 

The following example illustrates \cref{thm:main}.

\begin{example} Suppose $C$ is a general smooth trigonal curve of genus $8$ totally ramified at two points $p$ and $q$. When $\tau = (0,8,10) \in \wt{\Sigma}_3$, $W^{\tau}(C,p,q) \subseteq \Pic^3(C)$ has codimension $5$ since $\tau$ has the following $3$-inversions: $\{(1,3),(1,6),(2,3),(2,6),(2,9)\}.$ Moreover, any line bundle $\mL\in W^{\tau}(C,p,q)$ satisfies 
\[h^0(C,\mL) \geq 1 \text{ and } h^0(C,\mL(9p)) \geq 5\]
and thus is also in the splitting locus $W^{\vec{e}}(C)$ for $\vec{e} = (-4,-3,0)$. Since the splitting locus also has codimension $h^1(\PP^1, \text{End}(\mO_{\PP^1}(\vec{e}))) = 5$, we see $W^\tau(C,p,q)$ is equal to the splitting locus.  The permutation $\tau = (0,8,10)$ therefore prescribes only generic ramification at $p$ and $q$ for line bundles of splitting type $(-4,-3,0)$. By shuffling the entries of $\tau$ we keep the one-point transmission loci at $p$ unchanged and hence the splitting type too, and these 5 specializations of $(0,8,10)$ and the corresponding locus dimensions are shown in \cref{fig:poset}; arrows indicate inclusion of the corresponding loci. Since there are four reduced words for $(10,8,0) - 5 = (5,3,-5)$ in $\Sigma_3$, we also see that $\# W^{(10,8,0)}(C,p,q) = 4$. There are 5 other sorted permutations in $\wt{\Sigma}_3$ [namely $(-2,8,12), (-2,9,11), (-1,7,12), (-1,9,10), \text{ and }  (0,7,11)$] that prescribe the same one-point loci at $q$ as $(0,8,10)$. Considering all shuffles of these 5 permutations along with the 6 already shown in \cref{fig:poset} results in a total of 36 permutations all of whose splitting types are $(-4,-3,0)$ but have various ramification at $p$ and $q$.

\begin{figure}[h!]
 \begin{tikzcd}
	& \tau && {\dim W^\tau(C,p,q)} \\
	& {(0,8,10)} && 3 \\
	{(8,0,10)} && {(0,10,8)} & 2 \\
	{(8,10,0)} && {(10,0,8)} & 1 \\
	& {(10,8,0)} && 0
	\arrow[from=3-1, to=2-2]
	\arrow[from=3-3, to=2-2]
	\arrow[from=4-1, to=3-1]
	\arrow[from=4-3, to=3-3]
	\arrow[from=5-2, to=4-1]
	\arrow[from=5-2, to=4-3]
\end{tikzcd}
\caption{Starting from a sorted $\tau$ and shuffling entries increases codimension.}
\label{fig:poset}
\end{figure}
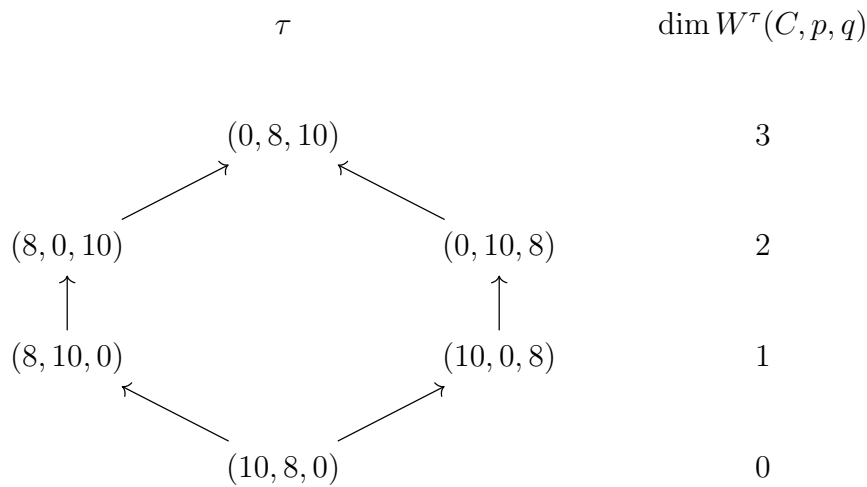

\end{example}

\subsection{Proof technique and outline of the paper}
We approach the problem using degeneration, wherein a general degree $k$ cover of $\PP^1$ degenerates to a singular curve of compact type, namely a chain of elliptic curves, on the special fiber. Pflueger's work \cite{pflueger2026transmissionpermutationsdemazureproducts} supplies a concrete combinatorial description of transmission loci on the special fiber of our family. Combined with upper-semicontinuity, this combinatorial description enables an upper bound on the dimension of transmission loci for the generic fiber of our family. The key technical result of this paper is a regeneration theorem (\cref{thm:regeneration}), which shows that limit line bundles in transmission loci on the central fiber smooth out to the generic fiber in at least the expected dimension. With regeneration in hand, the proof of enumeration is immediate and Cohen-Macaulayness (\cref{thm:cm}) and connectedness (\cref{thm:connected}) are shown by establishing them on the central fiber, since those properties are open in flat families. To prove smoothness of $W^{\tau}(C,p,q)$ on the open locus consisting of line bundles of type precisely $\tau$ (\cref{cor:smooth_open}), we define a pointed Petri map whose injectivity is equivalent to smoothness at a line bundle of type $\tau$, and show the map is injective on the central fiber of our family. We finally prove smoothness of $W^{\tau}(C,p,q)$ along transmission loci of codimension $1$ (\cref{thm:smooth}), by showing such loci are Cartier divisors of $W^{\tau}(C,p,q)$, which also yields normality and irreduciblity.

In \cref{sec:combinatorics}, we discuss the connection between prescribing ramification at $p$ and $q$ and permutations in $\wt{\Sigma}_k$ and reduce the problem of checking whether a line bundle $\mL$ lies in $W^{\tau}(C,p,q)$ to verifying for the existence of $k$ ``new" sections in $k$ distinct twists $\mL(a_i p+b_i q)$ for certain $a_i,b_i \in \ZZ$ (\cref{prop:non-colliding}). In \cref{sec:degeneration}, we discuss our degeneration along with the definition and combinatorial stratification of transmission loci on the special fiber. We then carry out set-theoretic regeneration in \cref{sec:regeneration} by carefully choosing non-colliding ramification indices for our so-called \textit{$\tau$-positive limit linear series} and showing they spread out in the expected dimension to the generic fiber of our family. Because the transmission loci on the special fiber are reduced, we obtain scheme-theoretic regeneration. In \cref{sec:smoothness}, we prove the smoothness of transmission loci away from codimension 2 transmission loci. Finally, in \cref{sec:monodromy}, we outline the proof of \cref{thm:main}(5), as we are able to re-use the arguments of \cite{larson_global_2025}.

\subsection{Conventions}
We work over an algebraically closed field. For $n\in \ZZ$, we will use $[n]$ to denote the set $\{0,1,\dots,\abs{n}-2,\abs{n}-1\}$. Throughout, \textit{curve} will mean a connected proper nodal curve and $C$ will refer to a smooth curve. We will use $g$ to refer to the genus of a curve and $d$ to refer to the degree of a line bundle.

\textbf{Acknowledgments.}  I am grateful to my advisor Isabel Vogt for her guidance and numerous insightful discussions, as well as for suggesting this problem; \Cref{app:irreducibility} is based on Vogt's notes. I also thank Hannah Larson, Feiyang Lin, and Emma Pickard for several helpful discussions.

\section{Combinatorics}\label{sec:combinatorics}
\subsection{Transmission loci}
Let $(C,p,q)$ be a twice-marked smooth curve of genus $g$.  In order to study the locus of line bundles in $\Pic^d(C)$ with prescribed ramification at $p$ and $q$, let us define the \defi{transmission function}
\begin{equation}\label{eq:transmission_fn}
    s_{\mL}^{p,q}(a,b) = h^0(C, \mL((a-1)p - bq)).
\end{equation}

\begin{rmk}
The shift by $1$ in \eqref{eq:transmission_fn} will result in cleaner formulae when we are dealing with the decomposition of permutations via the Demazure product on our degeneration.
\end{rmk}

Given an arbitary function $f:\ZZ^2 \to \ZZ_{\geq 0}$, it is natural to ask when $f$ might arise as the transmission function of some line bundle. In other words, what are the  characteristic properties of a transmission function? \cref{lem:h0_is_submodular} answers this by giving necessary conditions $f$ should possess, and as we will later see these are also sufficient when $C$ has sufficiently large genus.

\begin{lem}\label{lem:h0_is_submodular}
    Let $(C,p,q)$ be a twice-marked smooth curve of genus $g$ and $\mL$ a line bundle of degree $d$ on $C$. Let $M = \max\{d,2g\}$ and $\chi = d-g$. The following hold:
    \begin{enumerate}
        \item For all $a,b \in \ZZ$, if $a-b \leq -M$ then $s^{p,q}_{\mL}(a,b) = 0$, and if $a-b \geq M$ then $s^{p,q}_{\mL}(a,b) = \chi+a-b$.
        \item For all $a,b \in \ZZ$,
        \begin{equation*}
           s_{\mL}^{p,q}(a+1,b) - s_{\mL}^{p,q}(a,b) -s_{\mL}^{p,q}(a+1,b+1) + s_{\mL}^{p,q}(a,b+1) \geq 0.
        \end{equation*}
    \end{enumerate}
\end{lem}

\begin{proof}
 Note that (1) is simply a restatement of the fact that negative-degree line bundles have no global sections and the Riemann-Roch theorem. 
For (2), let $\mM = \mL(ap-bq)$. Then the image of the global sections of $\mM(-p)$ and $\mM(-q)$ in $H^0(\mM)$ are the subspace of sections vanishing at $p$ and $q$ respectively, and they intersect in the subspace of sections vanishing at both $p$ and $q$, which is the image of the global sections of $\mM(-p-q)$ in $H^0(\mM)$. Thus, 
\begin{equation*}
    h^0(\mM) - h^0(\mM(-p)) - h^0(\mM(-q)) + h^0(\mM(-p-q)) \geq 0,
\end{equation*}
which is the assertion of (2).
\end{proof}
A function $f:\ZZ^2 \to \ZZ$ for which there exist $M$ and $\chi$ such that (1) and (2) from \cref{lem:h0_is_submodular} hold is called a \defi{submodular slipface function}. Pflueger \cite{pflueger_extended_2022} shows that, in fact, such functions are quite special, since they all arise from \defi{almost-sign-preserving} permutations, which are permutations of $\ZZ$ that change the sign of only finitely many integers; we denote the group of such permutations as $\ASP$. The correspondence is as follows: given a permutation $\tau\in \ASP$, we define its \defi{slipface function} as 
\begin{equation*}
    s_{\tau}(a,b) = \#\{n \geq b: \tau(n) < a\}.
\end{equation*}
Then for any submodular slipface function $f$, there exists a unique $\tau \in \ASP$ such that $f = s_{\tau}$ \cite[Prop.\ 4.3]{pflueger_extended_2022}. Consequently, for a line bundle $\mL$ on $(C,p,q)$, there exists a unique $\tau = \tau^{p,q}_{\mL} \in \ASP$ such that $s_{\mL}^{p,q} = s_{\tau}$. 

The permutation $\tau_{\mL}^{p,q}$, through its combinatorics, packages geometric information about $\mL$ even beyond its vanishing behavior at $p$ and $q$. For $\alpha \in \ASP$, define its \defi{shift} as 
\begin{equation*}
    \chi_{\alpha}\coloneqq \#\{n \geq 0: \alpha(n) < 0\} - \#\{n < 0: \alpha(n) \geq 0\}.
\end{equation*} 
Then, the shift of $\tau = \tau^{p,q}_{\mL}$ encodes the degree of $\mL$:
\begin{equation}\label{eq:degree}
    \chi_\tau = d-g.
\end{equation}
Further, the inverse of $\tau$ corresponds to the twisted Serre dual of $\mL$:
\begin{equation}\label{eq:serre_dual}
    s_{\tau^{-1}}(b,a) = h^1(C,\mL((a-1)p - bq)) = s^{q,p}_{\omega_C(p+q)\otimes \mL^{\vee}}(-b+1,-a+1).
\end{equation}
If $\iota_{n}\in \ASP$ denotes the translation $\iota_{n}(m) = m-n$, then pre- and post-composition of $\iota_n$ with $\tau$ correspond to twisting by $q$ and $p$ respectively:
\begin{equation}\label{eq:twist}
    \tau\iota_{n} = \tau_{\mL(nq)}^{p,q} \text{ and } \iota_n\tau = \tau_{\mL(np)}^{p,q}.
\end{equation}

 Given the correspondence between the transmission functions of line bundles and slipface functions of $\ASP$ permutations, we can use $\ASP$ permutations to prescribe all possible ramification at $p$ and $q$ for line bundles in $\Pic^{\chi_{\tau}+g}(C)$.
\begin{dfn}
    For each $\tau \in \ASP$, the \defi{$\tau$-transmission locus} of $(C,p,q)$ set-theoretically is
\begin{equation}\label{eq:transmission_locus}
    W^{\tau}(C,p,q) = \{\mL\in \Pic^{\chi_{\tau}+g}(C): h^0(C,\mL(ap-bq)) \geq s_{\tau}(a+1,b)\text{ for all } a,b\in \ZZ\},
\end{equation}
and we view it as a subvariety of $\Pic^{\chi_{\tau}+g}(C)$ with its scheme structure coming from the intersection of the determinantal varieties determined by each of the $h^0$ conditions.
\end{dfn}
The slipface functions define the \defi{Bruhat order} on $\ASP$: we say $\alpha \leq \beta$ if $s_{\alpha}(a,b) \leq s_{\beta}(a,b)$ for all $a,b\in \ZZ$. The Bruhat order determines inclusion of the transmission loci: $W^{\beta}(C,p,q) \subseteq W^{\alpha}(C,p,q)$ if and only if $\chi_{\alpha} = \chi_{\beta}$ and $\alpha \leq \beta$.

A special subgroup of $\ASP$ consists of the permutations with the property that $\tau(n+k) = \tau(n) + k$ for all $n\in \ZZ$; this subgroup is called the \defi{extended $k$-affine symmetric group}, denoted $\wt{\Sigma}_k$. One can check that an ASP permutation $\tau$ is in $\wt{\Sigma}_k$ if and only if
\begin{equation}\label{eq:k_invariance}
    s_{\tau}(a+k,b+k) = s_{\tau}(a,b)
\end{equation}
for all $a,b\in \ZZ$. When $p$ and $q$ on $C$ are such that $kp\sim kq$, then $s_{\tau_{\mL}^{p,q}}$ does indeed satisfy \eqref{eq:k_invariance} and so $\tau_{\mL}^{p,q} \in \wt{\Sigma}_k$ in this case.

\begin{dfn}\label{dfn:k_inversions}
Given $\tau\in \ASP$, an \defi{inversion} of $\tau$ is a pair $(a,b)\in \ZZ^2$ such that $a < b$ but $\tau(a) > \tau(b)$. When $\tau \in \wt{\Sigma}_k$ and $(a,b)$ is an inversion of $\tau$, then so is $(a+kq,b+kq)$ for all $q\in \ZZ$. The \defi{$k$-inversions} of $\tau \in \wt{\Sigma}_k$ are the inversions of $\tau$ under the equivalence relation $(a_1,b_1) \sim_k (a_2,b_2)$ if $a_1 - a_2 = b_1-b_2 \equiv 0\pmod{k}$. The number of inversion classes under this relation $\sim_k$ is denoted \defi{$\inv_k(\tau)$}.
\end{dfn}
  Let $\mathcal{H}_{k,g,2}$ denote the stack of twice-marked smooth genus $g$ curves $(C,p,q)$ with $kp\sim kq$ or equivalently degree $k$ covers of $\PP^1$ totally ramified at $p$ and $q$. When $(C,p,q)$ is a point of $\mathcal{H}_{k,g,2}$, the \defi{expected codimension of $W^{\tau}(C,p,q)$} in $\Pic^{\chi_\tau+g}(C)$ is $\inv_k(\tau)$. 

 \begin{rmk}\label{rmk:llv}
In \cite{larson_global_2025}, where they study $W^{\vec{e}}(C)$ for a general $k$-gonal curve, a correspondence is exhibited between splitting types and certain special $k$-cores called $k$-staircase diagrams. A \textit{$k$-core} is a Young diagram with no hook-length divisible by $k$ and a \textit{$k$-staircase} is a $k$-core with a unique removable box. The set of $k$-cores is in bijection with elements of $\Sigma_k/S_k$, the \textit{$k$-affine symmetric group} modulo the action of the symmetric group. The subgroup $\Sigma_k \leq \wt{\Sigma}_k$ consists of permutations $\alpha$ satisfying the additional normalization constraint
    \begin{equation*}
        \sum_{i=0}^{k-1} \alpha(i) = \frac{k(k-1)}{2}.
    \end{equation*}
    Indeed, these are precisely the permutations $\alpha$ of $\wt\Sigma_{k}$ with $\chi_{\alpha}=0$. More generally, for an arbitrary element $\tau \in \wt{\Sigma}_k$, we have the relation
    \begin{equation}\label{eq:normalization}
        \sum_{i=0}^{k-1}\tau(i) = \frac{k(k-1)}{2} - k\chi_{\tau},
    \end{equation}
    or equivalently $\tau+\chi_{\tau}\in \Sigma_{k}$.

    If $w(\vec{e}) \in \Sigma_k$ is the permutation constructed in \cite[Thm.\ 1.4]{larson_global_2025} and $kp\sim kq$ on $C$, then $W^{\vec{e}}(C) = W^{\tau(\vec{e})}(C,p,q)$ where
    \begin{equation}\label{eq:translation}
        \tau(\vec{e})(n) \coloneqq w(\vec{e})(n+1) - \chi(\mO_{\PP^1}(\vec{e})),
    \end{equation}
    as shown in \cite[Prop.\ 6.12]{larson_global_2025}.
    \end{rmk}

 \begin{example}\label{ex:trigonal_3}
 When $C$ is a general smooth trigonal curve of genus $3$ and totally ramified at $p$ and $q$ then $\tau_{\mO_{C}}^{p,q} = (0,5,7)$. Indeed, we know $\chi_{\tau} =-3$ and from $s_{\tau}(0,0) = h^0(\mO_{C}(-p)) = 0$ we see all entries of $\tau$ must be non-negative and thus $\#\{n < 0: \tau(n) \geq 0\} = 3$. From $s_{\tau}(1,0)  = h^0(\mO_{C}) = 1$, we deduce that one of the entries of $\tau$ must be $0$. Furthermore, from $s_{\tau}(1,1) = h^0(\mO_{C}(-q)) = 0,$ we see that $0$ must be the first entry. From $s_{\tau}(4,0) = h^0(\mO_C(3p)) = 2$, we see that the other two entries of $\tau$ must be $\geq 4$, so they must either be $\{4,8\}$ or $\{5,7\}.$ Since $s_{\tau}(2,-1) = h^0(\mO_{C}(p+q)) = 1$, we conclude that the last entry of $\tau$ must be $> 4$, so that rules out $(0,8,4)$. Finally, from $s_{\tau}(2,-2) = h^0(\mO_{C}(p+2q)) = 1$ we rule out $(0,4,8)$ and from $s_{\tau}(3,-1) = h^0(\mO_{C}(2p+q)) = 1$ we rule out $(0,7,5)$. Thus, $\tau^{p,q}_{\mL} = (0,5,7)$. Observe that $\inv_{k}(\tau) = 3$ and there is a unique reduced word for $\tau+\chi_{\tau}$ in $\Sigma_k$, namely $\sigma^3_{0}\sigma^3_{1}\sigma^3_{2}$ (see \eqref{eq:transposition} for the definition of $\sigma^k_{i}$), and so $\dim W^{\tau}(C,p,q) = 0$ and $\# W^{\tau}(C,p,q) = 1$.
 Using \eqref{eq:twist}, it is now possible to compute $\tau^{p,q}_{\mO_{C}(ap+bq)}$ for all $a,b\in \ZZ$. From $\tau^{p,q}_{\mO_C}$ we can also compute the splitting type of $\mO_C$ is $\vec{e} = (-3,-2,0)$. 
 \end{example}
\subsection{Demazure product} Our key technique for studying $W^{\tau}(C,p,q)$ is degeneration and the Demazure product is the tool that enables a concrete combinatorial description of transmission loci on the special fiber.
Pflueger \cite[Theorem A]{pflueger_extended_2022} defines the \defi{Demazure product} $\star$ on $\ASP$ via the slipface functions: for $\alpha,\beta\in \ASP$, $\alpha\star \beta\in \ASP$ is the permutation whose slipface function is given by
\begin{equation}\label{eq:demazure}
    s_{\alpha \star \beta}(a,b) = \min_{\ell\in \ZZ}[s_{\alpha}(a,\ell) + s_{\beta}(\ell,b)].
\end{equation}
When $\alpha$ and $\beta^{-1}$ have no inversions in common then $\alpha \star \beta$ is  just the ordinary product $\alpha\beta$. 

Of note for our purposes will be the Demazure product of a permutation with a transposition since that will enable us to do induction on the number of $k$-inversions. For each $i\in [k]$, let $\sigma^k_i\in \wt{\Sigma}_k$ denote the following permutation:
\begin{equation}\label{eq:transposition}
    \sigma^k_i(n) = 
    \begin{cases}
        n+1, & \text{ if } n \equiv i \pmod{k}\\
        n-1, & \text{ if } n \equiv i+1 \pmod{k}\\
        n, & \text{ otherwise.}
    \end{cases}
\end{equation}
Also, for $i \in [k],$ we say $\tau$ is \defi{sorted at $i$} if $\tau(i) < \tau(i+1)$.

\begin{prop}\label{prop:swap}
If $\tau$ is sorted at $i$, then 
\[s_{\tau\sigma^k_i} = s_{\tau\star\sigma^k_i}(a,b) = 
\begin{cases}
    s_{\tau}(a,b)+1 & \text{if } b\equiv i+1\pmod{k} \text{ and } \tau(b-1) < a \leq \tau(b)\\
    s_{\tau}(a,b) & \text{otherwise}.
\end{cases}
\]
\end{prop}
\begin{proof}
First, note that
\[s_{\sigma^k_i}(a,b) = \max\{a-b,0\} + \delta(a=b\equiv i+1\pmod{k}).\]
Now, since $\tau$ is sorted at $i$, $\tau\sigma^k_i = \tau\star \sigma^k_i$, so
\[s_{\tau\sigma^k_i}(a,b) = s_{\tau\star \sigma^k_i}(a,b) = \min_{l\in \ZZ}[s_{\tau}(a,\ell)+s_{\sigma_{i}}(\ell,b)]. \]
Fix $(a,b)$. Denote $f(\ell) = s_{\tau}(a,\ell)+s_{\sigma_{i}}(\ell,b)$. Notice that $f(\ell)\geq s_{\tau}(a,b)$ whenever $\ell>b$ and $f(\ell) = s_{\tau}(a,\ell) \geq s_{\tau}(a,b)$ whenever $\ell < b$. Also,
\[f(b) = s_{\tau}(a,b)+\delta(b\equiv i+1 \pmod{k}).\]
Thus, $s_{\tau\sigma^k_i}(a,b)$ is always at least $ s_{\tau}(a,b)$.

So, suppose $b\not\equiv i+1 \pmod{k}$. Then $f(b) = s_{\tau}(a,b)$, so $s_{\tau\sigma^k_i}(a,b) = s_{\tau}(a,b)$. Next, if $a > \tau(b)$, then 
\[f(b+1) = (s_{\tau}(a,b)-1) + 1 = s_{\tau}(a,b),\]
so $s_{\tau\sigma^k_i}(a,b) = s_{\tau}(a,b)$. Similarly, if $a \leq \tau(b-1)$, then 
\[f(b-1) = s_{\tau}(a,b) + 0 =  s_{\tau}(a,b),\]
so $s_{\tau\sigma^k_i}(a,b) = s_{\tau}(a,b)$ in this case as well.

Finally, suppose $b\equiv i+1 \pmod{k}$ and $\tau(b-1) < a \leq \tau(b)$. In this case, note that \[f(b-1) = f(b) = f(b+1) = s_{\tau}(a,b)+1\]
and since $f$ is non-increasing for $\ell < b$ and non-decreasing for $\ell > b$, we see that $s_{\tau\sigma^k_i}(a,b) = s_{\tau}(a,b)+1$ in this case.
    \end{proof}

\subsection{Essential twists} While the definition of $W^{\tau}(C,p,q)$ involves checking one inequality for each pair $(a,b)\in \ZZ^2$, we can cut this down substantially using the \defi{essential set} of $\tau \in \ASP$:
\[\Ess(\tau) \coloneqq \{(a,b):\ZZ^2: \tau^{-1}(a-1)  \geq b > \tau^{-1}(a), \tau(b-1) \geq a > \tau(b)\}.\]
 A permutation $\tau$ is said to be of \defi{bounded difference} if $\abs{\tau(n)-n}$ is bounded; note that any $\tau \in \wt\Sigma_k$ has bounded difference. We use the following Lemma from \cite{pflueger_extended_2022} without proof.
\begin{lem}[Corollary 7.9 \cite{pflueger_extended_2022}]\label{lem:ess_set} 
    If $\alpha,\beta \in  \ASP$ have the same shift and $\alpha$ has bounded difference, then $\alpha \leq \beta$ if and only if $s_{\alpha}(a,b) \leq s_{\beta}(a,b)$ for all $(a,b) \in \Ess(\alpha)$.
\end{lem}
We call a permutation $\alpha \in \wt{\Sigma}_k$ \defi{sorted} if $\alpha(0) < \alpha(1) < \dots < \alpha(k-1)$. 
\begin{cor}\label{cor:sorted_twists}
Suppose $\alpha \in \wt{\Sigma}_k$ is sorted and $\beta \in \wt{\Sigma}_k$ has the same shift as $\alpha$. Then $\alpha \leq \beta$ if and only if $s_{\alpha}(a,0) \leq s_{\beta}(a,0)$ for all $a\in \ZZ$.
\end{cor}
\begin{proof}
    We see that $(a,b)\in \Ess(\alpha)$ implies $b \equiv 0 \pmod{k}$ since $\alpha(b-1) > \alpha(b)$ is only possible then. Since $s_{\alpha}(a,kq) = s_{\alpha}(a-kq,0)$ for all $a,q\in \ZZ$ and similarly for $\beta$, we conclude by \cref{lem:ess_set}.
\end{proof}
 Therefore, if $\tau$ is sorted, 
 \begin{equation*}
 W^{\tau}(C,p,q) = \{\mL\in\Pic^{\chi_\tau+g}(C): h^0(\mL(ap)) \geq s_{\tau}(a+1,0) \text{ for all } a \in \ZZ\}.
  \end{equation*}
  We now proceed to define certain essential twists corresponding to $\tau$ so that the presence of a ``new" section in each such twist for a line bundle verifies the inclusion of that line bundle in $W^{\tau}(C,p,q)$. Because of the property \eqref{eq:k_invariance}, it suffices to study the behavior of a line bundle at the twists in the set $T_k$ defined in \cref{dfn:order}. We also endow $T_k$ with a partial order to encapsulate the notion of which twists include into others via the natural maps of multiplying by sections pulled back from $\PP^1$ along the covering map, as well as powers of the constant sections of $\mO_{C}(p)$ and $\mO_{C}(q)$.

\begin{dfn}\label[dfn]{dfn:order}
Define the sets $T_k \coloneqq \ZZ \times \{-k+1,-k+2,\dots,0\}$. We endow $T_k$ with a partial order $\prec$ given by
\[(a_1,b_1) \prec (a_2,b_2) \text{ if }  (a_1 < a_2 \text{ and  } b_1 \leq b_2) \text{ or } a_1 + k < a_2.\]
Notice that if $t_1 = (a_1,b_1) \prec (a_2,b_2) = t_2$ then there exist unique integers $ 0\leq q_{t_1,t_2}, 0 \leq r_{t_1,t_2},s_{t_1,t_2} < k$ such that 
\begin{equation}\label{eq:twist_up}
    \mO_{C}(a_1p + b_1q) \otimes \mO_{C}((q_{t_1,t_2}k+r_{t_1,t_2})p + s_{t_1,t_2}q) \cong \mO_{C}(a_2p + b_2q)
\end{equation}
whenever $kp \sim kq$ on $C$.
\end{dfn}

The following key definition makes precise the notion of ``new" sections in the essential twists we shall define.
\begin{dfn}\label[dfn]{dfn:non-colliding}
    Let $\mL$ be a line bundle on $(C,p,q) \in \mathcal{H}_{k,g,2}$ along with two global sections $\mu_i \in H^0(C, \mL(a_ip+b_iq))$ for some $t_i = (a_ip+b_iq)\in T_k$ for $i=1,2$. Let the order of vanishing of $\mu_i$ at $q$ be $v_i$. We say $\mu_1$ and $\mu_2$ are \defi{non-colliding at $q$} if for each $t\in T_k$ such that $t_{1}\prec t$ and $t_{2} \prec t$, we have 
     \begin{equation}\label{eq:non-col}
        v_{1}+s_{t_{1},t}\not\equiv v_{2}+s_{t_{2},t}\pmod{k},
    \end{equation}
    where $s_{t_{1},t}$ and $s_{t_{2},t}$ are as in \cref{dfn:order}. Note that a single twist witnessing the non-colliding property \eqref{eq:non-col} verifies the property for all comparable twists. Further, a collection of $n$ global sections $\mu_i \in H^0(C, \mL(a_ip+b_iq))$ for some $t_i = (a_i,b_i)\in T_k$ for $i\in[n]$ is called \defi{non-colliding at $q$} if the sections are pairwise non-colliding at $q$.
\end{dfn}

\begin{dfn}\label[dfn]{dfn:ess_twists}
    For a permutation $\alpha \in \wt{\Sigma}_k$, define the function $e_\alpha: [k]\to \ZZ$ by
\[e_{\alpha}(i) \coloneqq \begin{cases}
        -i & \text{ if there exists } 0 \leq j < i \text{ such that } \alpha(j) > \alpha(i)\\
        0 &\text{ otherwise.}
    \end{cases}
\]
\end{dfn}

The $k$ twists $\{(\alpha(i),e_{\alpha}(i))\}_{i\in [k]}$ will serve as the \defi{essential twists} for $\alpha$ and the following proposition explains the name.

 \begin{prop}\label{prop:non-colliding}
Let $\alpha \in \wt{\Sigma}_k$. Suppose $\mL \in \Pic^{\chi_{\alpha}+g}(C)$ is a line bundle for which there exist sections $\mu_i \in H^0(\mL(\alpha(i)p+e_{\alpha}(i)q))$ for each $i\in [k]$, that are non-colliding at $q$. Then $\tau^{p,q}_{\mL} \geq\alpha$. 
\end{prop}
\begin{proof}
    We prove the claim directly when $\alpha$ is sorted and then proceed by induction on the number of $k$-inversions in $\alpha$ due its ordering. 
    
\textbf{Step 1}: Suppose $\alpha$ is sorted; in this case $e_{\alpha}(i) = 0$ for all $i$. Let $\eta_{p}$ and  $\eta_{q}$ be the constant sections of $\mO_C(p)$ and $\mO_C(q)$ respectively.  Let $a\in \ZZ$. Note that 
\begin{equation*}
    s_{\alpha}(a+1,0) = \#\{n\geq 0: \alpha(n) \leq a\} = \sum_{i=0}^{k-1}\max\Big\{\Big\lfloor{\frac{a-\alpha(i)}{k}\Big\rfloor}+1, 0\Big\}.
\end{equation*}
Now, for each $i\in [k]$ such that $\alpha(i) \leq a$, there is a natural multiplication map with sections pulled back from $\PP^1$:
\begin{equation*}
    \langle \mu_i \rangle\otimes H^0(\PP^1, \mO_\PP^1(q_{i,a})) \hookrightarrow H^0(\mL(\alpha(i)\cdot p)) \otimes H^0(\PP^1, \mO_\PP^1(q_{i,a})) \xr{\cdot \eta_p^{r_{i,a}}} H^0(\mL(ap))
\end{equation*}
where $q_{i,a} = \lfloor\frac{a-\alpha(i)}{k}\rfloor$ and $r_{i,a} = a-\alpha(i)-q_{i,a}k \geq 0$. We thus, get a map
\begin{equation*}
    \bigoplus_{i\in[k]: \alpha(i) \leq a}\langle \mu_i \rangle\otimes H^0(\PP^1, \mO_\PP^1(q_{i,a})) \to H^0(\mL(ap)),
\end{equation*}
which is injective by comparing orders of vanishing at $q$. So, since $h^0(\PP^1,\mO_{\PP^1}(q_{i,a})) = q_{i,a}+1$, we find that
\[h^0(\mL(a p)) \geq \sum_{i\in[k]: \alpha(i) \leq a}(q_{i,a}+1) = s_{\alpha}(a+1,0)\]
for all $a\in \ZZ$.
By \cref{cor:sorted_twists}, it follows that $\tau^{p,q}_{\mL} \geq \alpha$.

\textbf{Step 2}: Now, let $\alpha$ be unsorted and let $\mu_i \in H^0(\mL(\alpha(i)p+e_{\alpha}(i)q))$ be non-colliding sections as in the statement of the proposition. Take 
\[j_0 = \argmax\{\alpha(j'): \alpha(j') > \alpha(j'+1), j'\in [k-1]\}.\]
Let $\alpha' = \alpha \sigma^k_{j_0}$. So, $\alpha'$ is sorted at $j_0$, $\alpha = \alpha' \sigma^k_{j_0}$, and $\inv_k(\alpha') = \inv_k(\alpha)-1$. Suppose the proposed claim is true for $\alpha'$. We will show that the non-colliding sections $\{\mu_i\}$ give rise to non-colliding sections $\{\mu'_i\}$ in the essential twists of $\alpha'$. For $i \not \in \{j_0,j_0+1\}$, it is clear that $(\alpha'(i), e_{\alpha'}(i)) = (\alpha(i), e_\alpha(i))$; thus, set $\mu'_i = \mu_{i}$ for $i\not\in\{j_0,,j_0+1\}$. Next, by choice of $j_0$, note that 
\[e_{\alpha}(j_0) = 0 = e_{\alpha'}(j_0+1),\]
and $\alpha(j_0) = \alpha'(j_0+1)$, so we can set $\mu'_{j_0+1} = \mu_{j_0}$. Lastly, $e_{\alpha}(j_0+1) = -(j_0+1)$, so \[\mu_{j_0+1} \in H^0(\mL(\alpha(j_0+1)p-(j_0+1)q)) = H^0(\mL(\alpha'(j_0)p-(j_0+1)q)),\] which maps to $H^0(\mL(\alpha'(j_0)p+e_{\alpha'}(j_0)q))$ via multiplication with $\eta_q^{e_{\alpha'}(j_0)+(j_0+1)}$ (note $e_{\alpha'}(j_0) \geq -j_0$). Therefore, taking $\mu'_{j_0} = \mu_{j_0+1}\cdot \eta_q^{e_{\alpha'}(j_0)+(j_0+1)}$, we get non-colliding sections $\{\mu'_{i}\}$ in the essential twists of $\alpha'$. Thus, by the inductive hypothesis, $\tau^{p,q}_{\mL} \geq \alpha'$. 

\textbf{Step 3}: Now, by Proposition \ref{prop:swap}, to show $\tau^{p,q}_{\mL} \geq \alpha$, it suffices to show that 
\[h^0(\mL((\alpha'(j_0)+\ell) p - (j_0+1)q)) \geq s_{\alpha'}(\alpha'(j_0)+\ell+ 1, j_0+1) + 1\]
for each $0 \leq\ell< \alpha'(j_0+1)-\alpha'(j_0)$. For convenience, denote $\mM_{\ell} = \mL((\alpha'(j_0)+\ell) p - (j_0+1)q)$.

We have 
\begin{align*}
    s_{\alpha'}(\alpha'(j_0)+ \ell + 1, j_0+1) = &\sum_{i=0}^{j_0} \max
\Big\{\Big\lfloor{\frac{\alpha'(j_0)+ \ell - \alpha'(i)}{k}\Big\rfloor},0\Big\}\\
&+ \sum_{i=j_0+1}^{k-1} \max
\Big\{\Big\lfloor{\frac{\alpha'(j_0)+ \ell - \alpha'(i)}{k}\Big\rfloor}+1,0\Big\}.\numberthis \label{eq:h0_intermediate}
\end{align*}

First, if $0 \leq i < j_0$ and $\Big\lfloor{\frac{\alpha'(j_0)+ \ell - \alpha'(i)}{k}\Big\rfloor} > 0$, then \[t_i \coloneqq (\alpha'(i),e_{\alpha'}(i)) \prec (\alpha'(j_0)+\ell,-(j_0+1)) =: w_{\ell},\]
so if $q_{t_i,w_{\ell}}, s_{t_i,w_{\ell}}, r_{t_i,w_{\ell}}$ are as in \cref{dfn:order}, then $\langle\mu_{i}\rangle\otimes H^0(\mO_{\PP^1}(q_{t_i,w_{\ell}}))$ maps injectively into $H^0(\mM_{\ell})$ via multiplication with $\eta_p^{r_{t_i,w_{\ell}}}\cdot \eta_q^{s_{t_i,w_{\ell}}}$. Note that
\begin{equation}\label{eq:le_j}
    q_{t_i,w_{\ell}} = \Big\lfloor{\frac{\alpha'(j_0)+ \ell - \alpha'(i)}{k}\Big\rfloor} -1 \;\; (0 \leq i < j_0).
\end{equation}

Second, if $j_0 <i < k$ and $\Big\lfloor{\frac{\alpha'(j_0)+ \ell - \alpha'(i)}{k}\Big\rfloor} \geq 0$, then \[\alpha'(i) \leq \alpha'(j_0)+\ell < \alpha'(j_0+1)\] so, $e_{\alpha'}(i) = -i < -(j_0+1)$. So, again
\[t_i = (\alpha'(i),-i) \prec w_{\ell},\]
and $\langle\mu_{i}\rangle\otimes H^0(\mO_{\PP^1}(q_{w_{\ell,t_i}}))$ maps injectively into $H^0(\mM_\ell)$ via multiplication with $\eta_p^{r_{t_i,w_{\ell}}}\cdot \eta_q^{s_{t_i,w_{\ell}}}$. In this case,
\begin{equation}\label{eq:ge_j}
    q_{t_i,w_{\ell}} = \Big\lfloor{\frac{\alpha'(j_0)+ \ell - \alpha'(i)}{k}\Big\rfloor} \;\; (j_0 < i < k).
\end{equation}

Finally, for $i=j_0$, $\mu_{j_0+1} \in  H^0(\mL(\alpha'(j_0) p - (j_0+1)q)$, so since \[t \coloneqq (\alpha'(j_0), - (j_0+1)) \prec w_{\ell},\] $\langle\mu_{j_0+1}\rangle\otimes H^0(\mO_{\PP^1}(q_{t,w_{\ell}}))$ maps injectively into $H^0(\mM_\ell)$ via multiplication with $\eta_p^{r_{t,w_{\ell}}}$ (note $s_{t,w_{\ell}} = 0$); 
\begin{equation}\label{eq:additional_section}
    q_{t,w_{\ell}} = \Big\lfloor{\frac{\alpha'(j)+ \ell-\alpha'(j)}{k}\Big\rfloor} = \Big\lfloor{\frac{\ell}{k}\Big\rfloor} .
\end{equation}

Thus, we see from comparing \eqref{eq:le_j}, \eqref{eq:ge_j}, and \eqref{eq:additional_section} with \eqref{eq:h0_intermediate} that
\begin{align*}
    h^0(\mM_\ell) \geq (q_{t,w_{\ell}}+1) + \sum_{\substack{i=0\\i\neq j_0}}^{k-1} (q_{t_i,w_\ell}+1) =  s_{\alpha'}(\alpha'(j_0)+ \ell + 1, j_0+1)+1.
\end{align*}
This completes the proof.
\end{proof}

 \begin{rmk}
     The constraints 
     \begin{equation*}
         h^0(\mL(\alpha(i)p+e_{\alpha}(i)q)) \geq s_{\alpha}(\alpha(i)+1, -e_{\alpha}(i))
     \end{equation*} for all $i\in [k]$ are insufficient to ensure that $\tau_{\mL}^{p,q} \geq \alpha$. For example, if $\alpha = (1,4,2,3)$ and $\tau_{\mL}^{p,q} = (3,2,4,1)$ in $\wt{\Sigma}_{4}$, then although $ s_{\tau_{\mL}^{p,q}}(\alpha(i)+1, -e_{\alpha}(i)) \geq s_{\alpha}(\alpha(i)+1, -e_{\alpha}(i))$ for each $i\in [4]$, we have $s_{\alpha}(4,2) = 2$ but $s_{\tau_{\mL}^{p,q}}(4,2) = 1$.
 \end{rmk}

The following lemma explains why in the $j$-th essential twist there is exactly one ``new'' section when the combined inclusion of all $j'$-th twists with $j' \prec j$ into the $j$-th twist is injective.

\begin{lem}\label{lem:unique_new_section}
    Let $\alpha \in \wt{\Sigma}_k$ and fix $j\in [k]$. Consider $[k]$ to be a poset under the order $\prec$ inherited from the set of essential twists $\{(\alpha(i), e_{\alpha}(i))\}_{i \in [k]}$ and for each $j' \prec j$, let $q_{j',j}\in \ZZ_{\geq 0}$ be as in \eqref{eq:twist_up}. Then
    \begin{equation*}
        \sum_{j' \prec j} (q_{j',j} + 1) = s_{\alpha}(\alpha(j)+1, -e_{\alpha}(j))-1.
    \end{equation*}
\end{lem}

\begin{proof}

First, suppose $e_{\alpha}(j) = 0$. Then
\begin{align*}
     s_{\alpha}(\alpha(j)+1,0) &= \#\{n \geq 0: \alpha(n) \leq \alpha(j)\}\\
     &=  \sum_{j'=0}^{k-1} \max\Big\{\Big\lfloor{\frac{\alpha(j)-\alpha(j')}{k}\Big\rfloor}+1,0\Big\}.
\end{align*}
Notice that a $j' \in [k]$ has a positive contribution to the sum above if and only if $j' \prec j$ or $j' = j$. Further, for a $j' \prec j$, the two twists under consideration are $(\alpha(j'),e_{\alpha}(j'))$ and $(\alpha(j),0)$ and so since $e_{\alpha}(j') \leq 0$, we see that $q_{j',j} = \Big\lfloor{\frac{\alpha(j)-\alpha(j')}{k}\Big\rfloor}$. As $j$ itself contributes exactly $1$ to the sum, we see that
\begin{equation}\label{eq:sum_of_qs}
     1+  \sum_{j' \prec j} (q_{j',j} + 1) =  s_{\alpha}(\alpha(j)+1,-e_{\alpha}(j)).
\end{equation}

Next, suppose $e_{\alpha}(j) = -j$. Then,
\begin{align*}
     s_{\alpha}(\alpha(j)+1,j) &= \#\{n \geq j: \alpha(n) \leq \alpha(j).\}\\
     &=\sum_{j'=0}^{j-1} \max\Big\{\Big\lfloor{\frac{\alpha(j)-\alpha(j')}{k}\Big\rfloor},0\Big\} +  \sum_{j'=j}^{k-1} \max\Big\{\Big\lfloor{\frac{\alpha(j)-\alpha(j')}{k}\Big\rfloor}+1,0\Big\}.
\end{align*}
Again, a $j' \in [k]$ has a possible contribution to the sums above only if $j' \prec j$ or $j' = j$. For a $j' \prec j$, the two twists under consideration are $(\alpha(j'),e_{\alpha}(j'))$ and $(\alpha(j),-j)$. If $j' < j$ and $j' \prec j$, then since $e_{\alpha}(j') \geq -j' > -j$, we see that $q_{j',j} = \Big\lfloor{\frac{\alpha(j)-\alpha(j')}{k}\Big\rfloor} - 1$ in this case. Next, if $j' > j$ and $j' \prec j$, then necessarily  $e_{\alpha}(j') = -j'$ which is less than $e_{\alpha}(j) = -j$, so $q_{j',j} = \Big\lfloor{\frac{\alpha(j)-\alpha(j')}{k}\Big\rfloor}$ in this case. Thus, \eqref{eq:sum_of_qs} holds again.
\end{proof}

For $\tau \in \wt{\Sigma}_k$, define
\begin{equation}\label{eq:h_tau}
    h_{\tau}(a,b) \coloneqq s_{\tau}(a,b) - (a - b +\chi_{\tau}) = \#\{n < b: \tau(n) \geq a\}. 
\end{equation}

Observe that $h_{\tau_{\mL}}(a+1,b) = h^1(C,\mL(ap-bq))$. 
We have the following combinatorial identity connecting the sum of $h_{\tau}$ over the essential twists of $\tau$ with the expected codimension $\inv_{k}(\tau)$, which will be useful for regeneration (\cref{thm:regeneration}).

\begin{lem}\label{lem:sum_h1_equals_inv}
Let $\tau \in \wt{\Sigma}_k$. Then
\begin{equation}\label{eq:combin_identity}
   \sum_{j=0}^{k-1} h_{\tau}(\tau(j)+1,-e_{\tau}(j)) = \inv_k(\tau).
\end{equation}

\end{lem}

\begin{proof}
Using the definition \eqref{eq:h_tau}  along with \eqref{eq:normalization}, we can rewrite the claimed equality as
\begin{equation}\label{eq:rewrite}
    \sum_{j=0}^{k-1} s_{\tau}(\tau(j)+1, -e_{\tau}(j))= \inv_{k}(\tau) + \frac{k(k+1)}{2} + \sum_{j=0}^{k-1} e_{\tau}(j).
\end{equation}

  We prove \eqref{eq:rewrite} directly when $\tau$ is sorted and then proceed by induction on the number of $k$-inversions in $\tau$ due its ordering.

\textbf{Step 1}: Suppose $\tau$ is sorted. We would like to prove
\[\sum_{i=0}^{k-1}s_{\tau}(\tau(i)+1,0) = \inv_{k}(\tau)+ \frac{k(k+1)}{2}. \]
Note that since $\tau$ is sorted, for each $i \in [k]$,

\begin{equation*}\label{eq:sorted_h0}
    s_{\tau}(\tau(i)+1,0) = \sum_{j=0}^{k-1} \max\left\{\floor[\Big]{\frac{\tau(i)-\tau(j)}{k}}+1,0\right\} = i+1 + \sum_{j=0}^{i} \floor[\Big]{\frac{\tau(i)-\tau(j)}{k}}.
\end{equation*}
Thus, 
\[ \sum_{i=0}^{k-1} s_{\tau}(\tau(i)+1,0) = \frac{k(k+1)}{2} + \sum_{i=0}^{k-1}\sum_{j=0}^{i} \floor[\Big]{\frac{\tau(i)-\tau(j)}{k} } = \frac{k(k+1)}{2} + \inv_{k}(\tau),\]
where in the last step we again used that $\tau$ is sorted.
  
 \textbf{Step 2}:  Now, suppose $\tau$ is unsorted. Take 
\[j_0 = \argmax\{\tau(j'): \tau(j') > \tau(j'+1), j'\in [k-1]\}.\]
Let $\tau' = \tau \sigma^k_{j_0}$. So, $\tau'$ is sorted at $j_0$, $\tau = \tau' \sigma^k_{j_0}$, and $\inv_k(\tau') = \inv_k(\tau)-1.$ Let $L(\tau)$ and $R(\tau)$ denote the left-hand and right-hand side of the claimed equality \eqref{eq:rewrite} respectively. We analyze how $L(\tau')$ and $R(\tau')$ change when we transition from $\tau'$ to $\tau$ in a case-by-case fashion depending on the value of $e_{\tau'}(j_0)$. Notice that $e_{\tau'}(j_0+1) = 0 = e_{\tau}(j_0)$ and $e_{\tau}(j_0+1) = -(j_0+1)$ always.

 \textbf{Case 1}: First, suppose $e_{\tau'}(j_0) = 0$. 

    Since $e_{\tau'}(j_0) = 0$, we know $\tau'(j) < \tau'(j_0)$ for all $j\in [j_0]$. So, 
    \begin{equation*}
        L(\tau)-L(\tau') = s_{\tau}\B(\tau'(j_0)+1,j_0+1\B) - s_{\tau'}\B(\tau'(j_0)+1, 0\B) = -j_0.
    \end{equation*}
    And,
    \begin{equation*}
        R(\tau)-R(\tau') = \inv_{k}(\tau)- \inv_{k}(\tau') - (j_0+1) = 1 - (j_0+1) = -j_0.
    \end{equation*}

\textbf{Case 2}: Next, suppose $e_{\tau'}(j_0) = -j_0$.

In this case,
    \begin{equation*}
        L(\tau)-L(\tau') = s_{\tau}\B(\tau'(j_0)+1,j_0+1\B) - s_{\tau'}\B(\tau'(j_0)+1, j_0\B) = 0.
    \end{equation*}
And,
\begin{equation*}
    R(\tau)-R(\tau') = \inv_{k}(\tau)- \inv_{k}(\tau') + (-(j_0+1) + j_0) = 1 - 1 = 0.
\end{equation*}

    Thus, in both cases
    \[R(\tau)-R(\tau') = L(\tau)-L(\tau')\]
    and the claim now follows by the inductive hypothesis $L(\tau') = R(\tau')$ for $\tau'$.
\end{proof}

\section{Degeneration}\label{sec:degeneration}
As in \cite{larson_global_2025}, we degenerate to a chain $X = E^1 \cup_{p^1} E^2 \cup_{p^2} \dots \cup_{p^{g-1}} E^g$ of $g$ elliptic curves glued together at nodes such that $\mO_{E^i}(kp^{i-1}) \cong \mO_{E^i}(kp^i)$ (\cref{fig:chain}).
\begin{figure}
    \centering
    \begin{tikzpicture}
\draw (0, 0) .. controls (1, -1) and (2, -1) .. (3, 0);
\draw (2, 0) .. controls (3, -1) and (4, -1) .. (5, 0);
\draw (4, 0) .. controls (5, -1) and (6, -1) .. (7, 0);
\filldraw (7.8, -0.5) circle[radius=0.02];
\filldraw (7.5, -0.5) circle[radius=0.02];
\filldraw (7.2, -0.5) circle[radius=0.02];
\draw (8, 0) .. controls (9, -1) and (10, -1) .. (11, 0);
\draw (10, 0) .. controls (11, -1) and (12, -1) .. (13, 0);
\filldraw (0.5, -0.42) circle[radius=0.03];
\filldraw (2.5, -0.42) circle[radius=0.03];
\filldraw (4.5, -0.42) circle[radius=0.03];
\filldraw (10.5, -0.42) circle[radius=0.03];
\filldraw (12.5, -0.42) circle[radius=0.03];
\draw (0.5, -0.75) node{$p^0$};
\draw (2.5, -0.75) node{$p^1$};
\draw (4.5, -0.75) node{$p^2$};
\draw (10.5, -0.8) node{$p^{g - 1}$};
\draw (12.5, -0.8) node{$p^g$};
\draw (1.5, -0.95) node{$E^1$};
\draw (3.5, -0.95) node{$E^2$};
\draw (5.5, -0.95) node{$E^3$};
\draw (9.5, -0.95) node{$E^{g - 1}$};
\draw (11.5, -0.95) node{$E^g$};
\end{tikzpicture}
    \caption{Degeneration $X$}
    \label{fig:chain}
\end{figure}
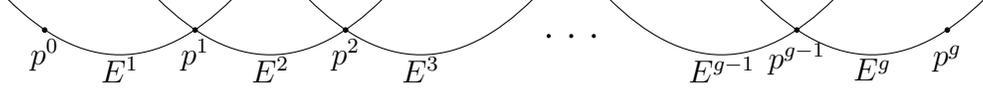
For each $i$ take a degree $k$ map $E^{i} \to \PP^1$ that is totally ramified at $p^{i-1}$ and $p^i$. These maps paste together give us a map $f: X\to P$, where $P$ is a chain of $g$ rational curves glued together at the images of the $p^i$'s. The theory of admissible covers then implies that there exists a map $\mathfrak{f}: \mathcal{X} \to \mathcal{P}$ over $B = \text{Spec } K[[t]],$ where $\mathcal{X}$ is a family of genus $g$ curves whose generic fiber is a smooth $k$-gonal curve totally ramified at two points over the corresponding fiber $\mathcal{P}|_{K((t))} \cong \PP^1_{K((t))}$ and the special fiber of $\mathfrak{f}$ is $f:X\to \mathcal{P}|_{0}\cong P$. We can further assume the total space $\mathcal{X}$ is smooth. Let $\mathfrak{f}_{q}$ and $\mathfrak{f}_{p}$ denote the sections of $\mathcal{X}$ whose generic fibers are the two points of total ramification and their special fibers are $p^0$ and $p^g$ respectively (\cref{fig:family}). For more details about this family of curves, see \cite[Section 2]{larson_global_2025}.

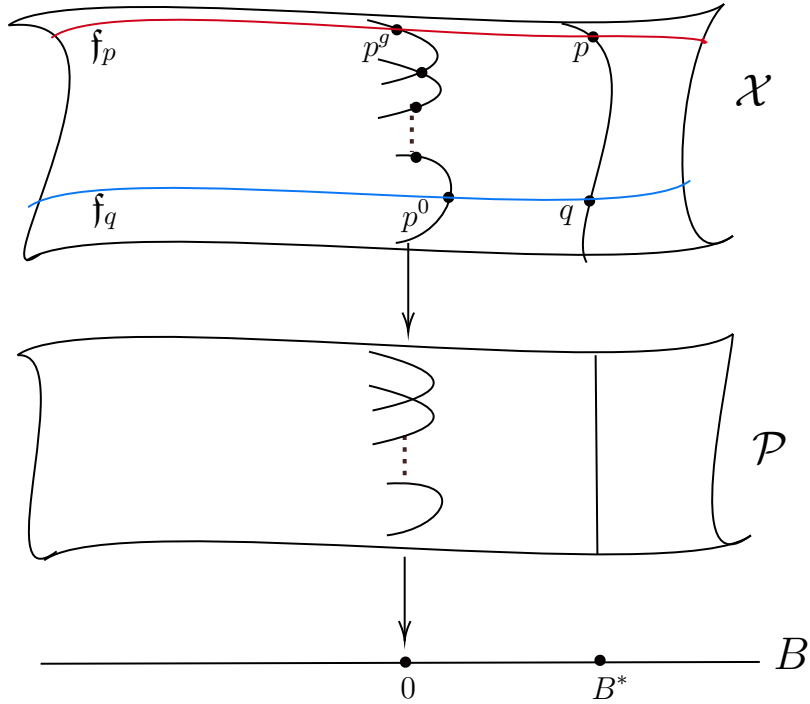
\begin{figure}
    \centering
\tikzset{every picture/.style={line width=0.75pt}} 

\begin{tikzpicture}[x=0.75pt,y=0.75pt,yscale=-0.9,xscale=0.9]

\draw    (112,36) .. controls (152,6) and (473,54) .. (513,24) ;
\draw    (122,168) .. controls (162,138) and (478,184) .. (518,154) ;
\draw    (112,36) .. controls (198,27) and (90,194) .. (130,164) ;
\draw    (513,24) .. controls (484,53) and (478,184) .. (518,154) ;
\draw    (117,221) .. controls (157,191) and (478,239) .. (518,209) ;
\draw    (132,336) .. controls (172,306) and (488,352) .. (528,322) ;
\draw    (117,221) .. controls (153,218) and (99,361) .. (139,331) ;
\draw    (518,209) .. controls (515,240) and (488,352) .. (528,322) ;
\draw    (130,394) -- (533,393) ;
\draw    (336,158) -- (336,206) ;
\draw [shift={(336,208)}, rotate = 270] [color={rgb, 255:red, 0; green, 0; blue, 0 }  ][line width=0.75]    (10.93,-3.29) .. controls (6.95,-1.4) and (3.31,-0.3) .. (0,0) .. controls (3.31,0.3) and (6.95,1.4) .. (10.93,3.29)   ;
\draw    (334,334) -- (334,379) ;
\draw [shift={(334,381)}, rotate = 270] [color={rgb, 255:red, 0; green, 0; blue, 0 }  ][line width=0.75]    (10.93,-3.29) .. controls (6.95,-1.4) and (3.31,-0.3) .. (0,0) .. controls (3.31,0.3) and (6.95,1.4) .. (10.93,3.29)   ;
\draw    (313,33) .. controls (361,45) and (367,59) .. (321,69) ;
\draw    (329,109) .. controls (378,106) and (362,154) .. (329,158) ;
\draw [color={rgb, 255:red, 68; green, 34; blue, 34 }  ,draw opacity=1 ][line width=1.5]  [dash pattern={on 1.69pt off 2.76pt}]  (338,80) -- (338,107) ;
\draw  [fill={rgb, 255:red, 8; green, 0; blue, 0 }  ,fill opacity=1 ] (330.58,36.05) .. controls (331.91,36.59) and (332.56,38.1) .. (332.02,39.43) .. controls (331.49,40.76) and (329.97,41.4) .. (328.64,40.87) .. controls (327.31,40.33) and (326.67,38.82) .. (327.2,37.49) .. controls (327.74,36.16) and (329.25,35.52) .. (330.58,36.05) -- cycle ;
\draw    (319,55) .. controls (367,67) and (365,78) .. (319,88) ;
\draw  [fill={rgb, 255:red, 8; green, 0; blue, 0 }  ,fill opacity=1 ] (344.58,60.05) .. controls (345.91,60.59) and (346.56,62.1) .. (346.02,63.43) .. controls (345.49,64.76) and (343.97,65.4) .. (342.64,64.87) .. controls (341.31,64.33) and (340.67,62.82) .. (341.2,61.49) .. controls (341.74,60.16) and (343.25,59.52) .. (344.58,60.05) -- cycle ;
\draw  [fill={rgb, 255:red, 8; green, 0; blue, 0 }  ,fill opacity=1 ] (359.58,130.05) .. controls (360.91,130.59) and (361.56,132.1) .. (361.02,133.43) .. controls (360.49,134.76) and (358.97,135.4) .. (357.64,134.87) .. controls (356.31,134.33) and (355.67,132.82) .. (356.2,131.49) .. controls (356.74,130.16) and (358.25,129.52) .. (359.58,130.05) -- cycle ;
\draw  [fill={rgb, 255:red, 8; green, 0; blue, 0 }  ,fill opacity=1 ] (341.38,107.56) .. controls (342.71,108.1) and (343.35,109.61) .. (342.82,110.94) .. controls (342.28,112.27) and (340.77,112.91) .. (339.44,112.38) .. controls (338.11,111.84) and (337.46,110.33) .. (338,109) .. controls (338.54,107.67) and (340.05,107.03) .. (341.38,107.56) -- cycle ;
\draw  [fill={rgb, 255:red, 8; green, 0; blue, 0 }  ,fill opacity=1 ] (341.38,79.56) .. controls (342.71,80.1) and (343.35,81.61) .. (342.82,82.94) .. controls (342.28,84.27) and (340.77,84.91) .. (339.44,84.38) .. controls (338.11,83.84) and (337.46,82.33) .. (338,81) .. controls (338.54,79.67) and (340.05,79.03) .. (341.38,79.56) -- cycle ;
\draw    (314,219) .. controls (362,231) and (362,242) .. (316,252) ;
\draw    (324,293) .. controls (373,290) and (357,318) .. (324,322) ;
\draw [color={rgb, 255:red, 68; green, 34; blue, 34 }  ,draw opacity=1 ][line width=1.5]  [dash pattern={on 1.69pt off 2.76pt}]  (334,266) -- (334,293) ;
\draw    (314,238) .. controls (362,250) and (362,261) .. (316,271) ;
\draw    (422,35) .. controls (484,47) and (422,146) .. (436,169) ;
\draw  [fill={rgb, 255:red, 8; green, 0; blue, 0 }  ,fill opacity=1 ] (440.58,40.05) .. controls (441.91,40.59) and (442.56,42.1) .. (442.02,43.43) .. controls (441.49,44.76) and (439.97,45.4) .. (438.64,44.87) .. controls (437.31,44.33) and (436.67,42.82) .. (437.2,41.49) .. controls (437.74,40.16) and (439.25,39.52) .. (440.58,40.05) -- cycle ;
\draw  [fill={rgb, 255:red, 8; green, 0; blue, 0 }  ,fill opacity=1 ] (438.58,132.05) .. controls (439.91,132.59) and (440.56,134.1) .. (440.02,135.43) .. controls (439.49,136.76) and (437.97,137.4) .. (436.64,136.87) .. controls (435.31,136.33) and (434.67,134.82) .. (435.2,133.49) .. controls (435.74,132.16) and (437.25,131.52) .. (438.58,132.05) -- cycle ;
\draw    (441,221) -- (442,333) ;
\draw [color={rgb, 255:red, 208; green, 2; blue, 27 }  ,draw opacity=1 ]   (136,43) .. controls (166.42,20.18) and (338,43.07) .. (436.01,42.13) .. controls (534.02,41.2) and (492.42,50.18) .. (502,43) ;
\draw [color={rgb, 255:red, 10; green, 119; blue, 247 }  ,draw opacity=1 ]   (123,138) .. controls (163,108) and (454,153) .. (494,123) ;
\draw  [fill={rgb, 255:red, 8; green, 0; blue, 0 }  ,fill opacity=1 ] (335.44,390.68) .. controls (336.77,391.22) and (337.41,392.73) .. (336.88,394.06) .. controls (336.34,395.39) and (334.83,396.04) .. (333.5,395.5) .. controls (332.17,394.96) and (331.53,393.45) .. (332.06,392.12) .. controls (332.6,390.79) and (334.11,390.15) .. (335.44,390.68) -- cycle ;
\draw  [fill={rgb, 255:red, 8; green, 0; blue, 0 }  ,fill opacity=1 ] (444.44,389.68) .. controls (445.77,390.22) and (446.41,391.73) .. (445.88,393.06) .. controls (445.34,394.39) and (443.83,395.04) .. (442.5,394.5) .. controls (441.17,393.96) and (440.53,392.45) .. (441.06,391.12) .. controls (441.6,389.79) and (443.11,389.15) .. (444.44,389.68) -- cycle ;

\draw (517,62) node [anchor=north west][inner sep=0.75pt]  [font=\Large] [align=left] {$\displaystyle \mathcal{X}$};
\draw (527,262) node [anchor=north west][inner sep=0.75pt]  [font=\Large] [align=left] {$\displaystyle \mathcal{P}$};
\draw (539,378) node [anchor=north west][inner sep=0.75pt]  [font=\Large] [align=left] {$\displaystyle B$};
\draw (309,39) node [anchor=north west][inner sep=0.75pt]   [align=left] {$\displaystyle p^{g}$};
\draw (331,132) node [anchor=north west][inner sep=0.75pt]   [align=left] {$\displaystyle p^{0}$};
\draw (427,44) node [anchor=north west][inner sep=0.75pt]   [align=left] {$\displaystyle p$};
\draw (419,135) node [anchor=north west][inner sep=0.75pt]   [align=left] {$\displaystyle q$};
\draw (330,400) node [anchor=north west][inner sep=0.75pt]   [align=left] {$\displaystyle 0$};
\draw (437,399) node [anchor=north west][inner sep=0.75pt]   [align=left] {$\displaystyle B^{*}$};
\draw (156,36) node [anchor=north west][inner sep=0.75pt]  [font=\large] [align=left] {$\displaystyle \mathfrak{f}_{p}$};
\draw (157,128) node [anchor=north west][inner sep=0.75pt]  [font=\large] [align=left] {$\displaystyle \mathfrak{f}_{q}$};
\end{tikzpicture}

    \caption{The family $\mathcal{X} \to \mathcal{P} \to B$.}
    \label{fig:family}
\end{figure}

\subsection{Classification of limiting bundles}
When $\mL^*$ is a line bundle of degree $d$ on the generic fiber $\mathcal{X}^*$ of $\mathcal{X}$, we can extend $\mL^*$ to a line bundle on the total space, however the extension is not unique. Instead, we have a unique extension $\mL_{\vec{d}}$ for each degree distribution $\vec{d}=(d^1,\dots,d^g)$ with $\sum d^i = d$. Restricting $\mL_{\vec{d}}$ to the special fiber $X$, we thus obtain a limit $L_{\vec{d}} = \mL_{\vec{d}}|_{X}$ for each degree distribution $\vec{d}$. Moreover, all these limits can be recovered by any one limit via the chip-firing operation, as explained in \cite[Section 3]{larson_global_2025}. We can thus define 
\begin{equation*}
    \Pic^d(X) \coloneqq \left(\bigsqcup_{\vec{d}: \sum d^i = d} \Pic^{\vec{d}}(X)\right)/\sim,
\end{equation*}
where $\sim$ is the equivalence relation generated by the chip-firing relation. An element of $\Pic^d(X)$ is called a \defi{limit line bundle}. Given limits $\{L_{\vec{d}}\}_{\vec{d}:\sum{d^i}=d}$, we refer to the corresponding limit line bundle $L$ by specifying its aspects
\begin{equation*}
    L^i = L_{(0,\dots,\underbrace{\scriptstyle d}_{i\text{th slot}},\dots,0)}|_{E^i}.
\end{equation*}
for each $i = 1,\dots,g$.

\subsection{Transmission loci on the special fiber} 

Now, suppose $\mL^{*} \in W^{\tau}(\mathcal{X}^*, \mathfrak{f}_p|_{\mathcal{X}^*},\mathfrak{f}_{q}|_{\mathcal{X}^*})$. What can we infer about the corresponding limit line bundle $L\in\Pic^d(X)$? In addition to requiring $d = \chi_{\tau} + g$, by upper-semicontinuity, we see that 
\begin{equation}
    h^0(X, L(ap^g-bp^{0})_{\vec{d}}) \geq h^0(\mathcal{X}^*, \mL^*(a\cdot \mathfrak{f}_p|_{\mathcal{X}^*}-b\cdot \mathfrak{f}_q|_{\mathcal{X}^*})) \geq s_{\tau}(a+1, b),
\end{equation}
must hold for all $a,b\in \ZZ$ and all degree distributions $\vec{d}$ such that $\sum d^i = d+a-b$. We call such limit line bundles \defi{$\tau$-positive} and denote the locus of $\tau$-positive bundles in $\Pic^{\chi_{\tau}+g}(X)$ by $W^{\tau}(X, p^g, p^0)$.

Pflueger used the Demazure product to completely describe $W^{\tau}(X, p^g, p^0)$ \cite[Prop.\ 3.9]{pflueger2026transmissionpermutationsdemazureproducts}. In particular, it is shown that $W^{\tau}(X,p^g,p^0)$ can be completely stratified using all \textit{reduced words} for $\tau + \chi_{\tau}$ in $\Sigma_{k}$, which are all the ways to write \[\tau + \chi_{\tau} = \sigma_{m_\ell}^{k}\sigma_{m_{\ell-1}}^{k}\cdots \sigma_{m_1}^{k},\]
with $\ell = \inv_k(\tau)$ and each $\sigma_{m_i}^k$ is one of the transpositions as in \eqref{eq:transposition}.
\begin{dfn}
Let $T = \sigma^k_{m_\ell} \dots \sigma^k_{m_1}$ be a reduced word for $\tau+\chi_\tau$. For every subset $S \subseteq \{1,\dots,g\}$ such that $\#S = \inv_k(\tau)$, let $\phi_S$ be the unique increasing bijection $S \to \{1,\dots,\inv_{k}(\tau)\}.$ For each subset $S$, define the following reduced subscheme of $\Pic^d(X)$:
\begin{equation*}
    W^{T,S}(X,p^g,p^0)  = \{L \in \Pic^d(X): L^i \cong \mO_{E^i}((m_{\phi_S(i)}+i)p^{i-1} +(d-m_{\phi_S(i)}-i)p^{i}) \text{ for all } i \in S\}.
\end{equation*}
Also, define 
\begin{equation*}
        W^{T}(X,p^g,p^0) = \bigcup_{\substack{S \subseteq \{1,\dots,g\}\\ \#S = \inv_k(\tau)}} W^{T,S}(X,p^g,p^0).
\end{equation*}
\end{dfn}
Then \cite[Prop.\ 3.9 \& Lemma\ 4.2]{pflueger2026transmissionpermutationsdemazureproducts} implies that
\begin{equation}\label{eq:stratification}
    W^{\tau}(X,p^g,p^0) = \bigcup_{T: \text{ reduced word for } \tau+\chi_\tau} W^{T}(X,p^g,p^0).
\end{equation}

\begin{rmk}
    The reader familiar with \cite{larson_global_2025} will observe the similarity with the description of splitting loci $W^{\vec{e}}(X)$ on the special fiber:
    \[W^{\vec{e}}(X) = \bigcup_{T: \text{ efficient $k$-regular filling of } \Gamma(\vec{e})} W^T(X).\]
\end{rmk}

\begin{example}
    Consider $\tau = (2,0,1)$ and $g = 3$; notice $\chi_{\tau} = 0$. Since $\inv_{k}(\tau) = 2$, any reduced word for $\tau$ has length 2. In fact, $T = \sigma^3_1 \sigma^3_0$ is the only reduced word for $\tau$. For each subset $S = \{a < b\} \subset \{0,1,2\}$, we have 
    \[W^{T,S}(X,p^3,p^0) = \{L \in \Pic^3(X): L^a \cong \mO_{E^a}(ap^{a-1} + (3-a)p^a) \text{ and }  L^b \cong \mO_{E^b}((1+b)p^{b-1} + (2-b)p^b) \}.\]
\end{example}

In light of $\eqref{eq:stratification}$ and the fact that each reduced word for $\tau$ has length $\inv_{k}(\tau)$, we conclude that $\dim W^{\tau}(X,p^g,p^0) = g-\inv_k(\tau)$, since for any $L\in W^T(X,p^g,p^0)$ exactly $g-\inv_{k}(\tau)$ aspects are unspecified and  can be any general line bundle in $\Pic^d(E^i)$. This observation combined with the fact that $\dim W^{\tau}(C,p,q)$ is an upper-semicontinuous function on the family $\mathcal{X}$ \cite[Thm.\ 5.1]{pflueger2026transmissionpermutationsdemazureproducts}, we see that
\begin{equation}\label{eq:upper_bound}
   \dim W^{\tau}(\mathcal{X}^*, \mathfrak{f}_p|_{\mathcal{X}^*},\mathfrak{f}_{q}|_{\mathcal{X}^*}) \leq g-\inv_{k}(\tau).
\end{equation}

\section{Regeneration}\label{sec:regeneration}
In this section, we prove that every $\tau$-positive limit line bundle on the central fiber of our family $\mathcal{X}$ arises as the limit of a line bundle in $W^{\tau}(\mathcal{X}^*,\mathfrak{f}_p|_{X^*},\mathfrak{f}_{q}|_{X^*})$. We do this by proving a regeneration theorem for non-colliding sections in the essential twists associated to a permutation $\tau \in \wt{\Sigma}_k$. We call a collection of non-colliding sections \[\mu_{j} \in H^0(\mL(\tau(j)p+e_{\tau(j)}q)),\] where $\mL$ is a line bundle of degree $\chi_\tau + g$ on a smooth curve, a \defi{$\tau$-positive linear series}. By \cref{prop:non-colliding}, we know the existence of a $\tau$-positive linear series in $\mL$ implies $\mL \in W^{\tau}(C,p,q)$. 

Now, suppose $L$ is a $\tau$-positive limit line bundle. Over the course of regeneration, we will construct non-colliding sections $\mu^{i}_{j} \in H^0(E^i, L^i(e_{\tau}(j)p^{i-1}+\tau(j)p^i))$ for each $1 \leq i \leq g$ and $j\in [k]$; such a collection of sections is called a \defi{$\tau$-positive limit linear series}. The collection of sections $\{\mu^i_j\}_{j\in [k]}$ for a fixed $i$ is called the \defi{$E^i$-aspect} of the limit linear series. As in the classical case of limit linear series, we will require that the constructed limit linear series is \defi{refined}, i.e., the relation
\begin{equation*}
    \ord_{p^{i}}(\mu^i_j) + \ord_{p^{i}}(\mu^{i+1}_j) = d+e_{\tau}(j) + \tau(j),
\end{equation*}
holds for all $1\leq i < g$ and $j\in [k]$, which will help us prove that the limit linear series smooths out to the general fiber. To define the desired ramification indices for the $\tau$-positive limit linear series, we require truncations, which we next discuss.

\subsection{Truncations}

Given a reduced word $T =\sigma^k_{m_\ell} \dots \sigma^k_{m_1}$ for $\tau + \chi_{\tau}$, for each $t\in [\ell+1]$ the \defi{$t$-th truncation} is the permutation $T^{\leq t} = \sigma^k_{m_t} \dots \sigma^k_{m_1}$, with the convention that $T^{\leq 0}$ is simply the identity permutation. From the definition, we see that
\begin{equation}\label{eq:truncations}
    T^{\leq t}(j) = \sigma_{m_{t}}^{k}(T^{\leq t-1}(j))=
    \begin{cases}
        T^{\leq t-1}(j) + 1 & \text{if } T^{\leq t-1}(j)\equiv m_t \pmod{k}\\
         T^{\leq t-1}(j) - 1 & \text{if } T^{\leq t-1}(j) \equiv m_t +1 \pmod{k}\\
          T^{\leq t-1}(j) & \text{else}.
    \end{cases}
\end{equation}
Since $T^{\leq t-1}$ is also a permutation, there exist unique indices $j_{+}(t),j_{-}(t) \in [k]$ with the property $ T^{\leq t}(j_{\pm}(t)) =  T^{\leq t-1}(j_{\pm}(t)) \pm 1$ and we call $j_{+}(t)$ the \defi{increasing} index and $j_{-}(t)$ the \defi{decreasing} index for $T^{\leq t}$. We will later require the following properties of $j_{\pm}(t)$.

\begin{prop}\label{prop:increasing_decreasing}
Let $T = \sigma^k_{m_\ell} \dots \sigma^k_{m_1}$ be a reduced word for $\tau + \chi_{\tau}$. Let $j_{\pm}$ be the increasing and decreasing indices  for $T^{\leq t_0}$.  Then:
\begin{enumerate}
    \item We have $T^{\leq t_0}(j_{+}) \equiv T^{\leq t_0}(j_{-}) + 1 \pmod{k}$ \text{ and } $T^{\leq t_0}(j_-) < T^{\leq t_0}(j_+)$. 
    \item Further, $T^{\leq t_0}(j_+) = T^{\leq t_0}(j_-)+1$ if and only if $j_{+} < j_{-}$.
    \item Lastly, 
    \begin{equation*}
        T^{\leq t}(j_+) - T^{\leq t}(j_-) \geq T^{\leq t_0}(j_+) - T^{\leq t_0}(j_-)
    \end{equation*}
    for $t_0 \leq t \leq \ell$.
\end{enumerate}
\end{prop}

\begin{proof}
The first part of the claim (1) follows from
\begin{equation*}
    T^{\leq t_0}(j_+) = T^{\leq t_0-1}(j_+) + 1 \equiv m_{t_0} + 1 \equiv T^{\leq t_0-1}(j_{-}) = T^{\leq t_0}(j_{-})+1 \pmod{k}.  
\end{equation*}
Next, we prove the second part of (1). Given a permutation $\alpha \in \Sigma_k$, for $j_1\neq j_2 \in [k]$ define 
\begin{equation}\label{eq:inv}
    \inv_k(\alpha, j_1,j_2) \coloneqq \BB(\floor[\Big]{\frac{\alpha(j_1) - \alpha(j_2)}{k}} + \mathbbm{1}_{j_1 < j_2}\BB) \cdot \mathbbm{1}_{\alpha(j_1) > \alpha(j_2)} + \BB(\floor[\Big]{\frac{\alpha(j_2) - \alpha(j_1)}{k}} + \mathbbm{1}_{j_1 > j_2}\BB) \cdot \mathbbm{1}_{\alpha(j_1) < \alpha(j_2)}.
\end{equation}
The number of $k$-inversion classes $(a,b)$ in $\tau$ such that $\{a,b\} \equiv \{j_1,j_2\} \pmod{k}$ is given by $\inv_{k}(\alpha, j_1,j_2) = \inv_{k}(\alpha, j_2,j_1)$. Now, define
\begin{equation*}
   \delta_t(j_1,j_2)\coloneqq  \inv_k(T^{\leq t},j_1,j_2) - \inv_k(T^{\leq t-1}, j_1,j_2).
\end{equation*}
The quantity $\delta_t(j_1,j_2)$ measures the change in number of $k$-inversions between the pair of residue classes $\{j_1,j_2\}$ when we go from the $T^{\leq t-1}$ to $T^{\leq t}$. Since 
\begin{equation*}
    \abs{(T^{\leq t_0}(j_1) - T^{\leq t_0}(j_2))- (T^{\leq t_0-1}(j_1) - T^{\leq t_0-1}(j_2))} = \begin{cases}
0 & \text{if } j_1 \not \in \{j_{\pm}\} \text{ and }  j_2 \not \in \{j_{\pm}\}\\
1 & \text{exactly one of } j_1 \text{ or } j_2 \in \{j_{\pm}\}\\
2 & \text{ else},
\end{cases}
\end{equation*} 
we have $\delta_{t_0}(j_1,j_2) = 0$ if $j_1,j_2 \not \in \{j_\pm\}$. Moreover, since $T^{\leq t}(j_1) - T^{\leq t}(j_2) \not\equiv 0 \pmod{k}$, we see that $\delta_{t_0}(j_1,j_2) = 0$ even if exactly one of $j_1$ or $j_2$ are in $\{j_{\pm}\}$. Thus, the number of $k$-inversions between every pair of indices $j_1,j_2 \in [k]$ such that $\{j_1,j_2\} \neq \{j_{\pm}\}$ remains the same in going from the $(t-1)$-st truncation to the $t$-th truncation and the number of $k$-inversions for the pair $\{j_{\pm}\}$ can only increase by at most $1$. Therefore, 
\begin{equation*}
    \inv_k(T^{\leq t}) - \inv_k(T^{\leq t-1}) = \sum_{j_1 < j_2 \in [k]} \delta_t(j_1, j_2) = \delta_t(j_{-}(t), j_+(t)) \leq 1,
\end{equation*}
and since the length of $T$ is exactly $\ell = \inv_{k}(\tau)$, it must be that $\delta_{t_0}(j_{-},j_+) = 1$.

Now, suppose (1) fails. Write $T^{\leq t_0}(j_-) = T^{\leq t_0}(j_+)+kq-1$ for some $q > 0$. Then,
\begin{equation*}
    T^{\leq t_{0}-1}(j_-) =  T^{\leq t_0}(j_-) + 1 = T^{\leq t_0}(j_+)+kq =  T^{\leq t_0-1}(j_+)+1+kq.
\end{equation*}
So,
\begin{equation*}
    \delta_{t_0}(j_-,j_+) = (q-1 + \mathbbm{1}_{j_- < j_+})-(q+\mathbbm{1}_{j_- < j_+}) = -1,
\end{equation*}
contradicting $\delta_{t_0}(j_-,j_+) = 1$. Therefore, $T^{\leq t_0}(j_-) < T^{\leq t_0}(j_+)$.

We now prove the forward direction of (2). Using (1), we can now write $T^{\leq t_0}(j_+) = T^{\leq t_0}(j_-)+kq+1$ for some $q \geq 0$. Suppose $q = 0$ but $j_- < j_+$. Then, since $T^{\leq t_0 - 1}(j_+) = T^{\leq t_0-1}(j_-) - 1$, we obtain
\begin{equation*}
     \delta_{t_0}(j_-,j_+) = (0+\mathbbm{1}_{j_- > j_+}) - (0 + \mathbbm{1}_{j_- < j_+}) = -1,
\end{equation*}
a contradiction. Thus, $j_+ < j_-$ if $q = 0$. Conversely, if $j_+ < j_{-}$, then $\delta_{t_0}(j_-, j_+) = 1$ forces $q=0$. This verifies (2).

Finally, for (3), suppose 
\begin{equation*}
    T^{\leq t}(j_+) - T^{\leq t}(j_-) < T^{\leq t_0}(j_+) - T^{\leq t_0}(j_-)
\end{equation*}
for some $t > t_0$ and let $t$ be the smallest such step with this property. The assumption on $t$ implies that $T^{\leq t-1}(j_+) \equiv T^{\leq t-1}(j_-)+1 \pmod{k}$ and since $T^{\leq t}(j_-) \not \equiv  T^{\leq t}(j_+)\pmod{k}$, it must be the case that $j_-(t) = j_+$ and $j_+(t) = j_-$, and
\begin{equation*}
     T^{\leq t}(j_+) - T^{\leq t}(j_-)  = kq-1.
\end{equation*}
We then see that $\delta_t(j_-(t),j_+(t)) = -1$, again a contradiction, and so (3) must be true. 
\end{proof}

\subsection{Ramification indices}

If we denote by $\sigma^k_{\bullet}$ the identity permutation, then every element of $W^{\tau}(X,p^g,p^0)$ is associated to a word $T=\sigma^k_{m_g}\dots \sigma^k_{m_1}$, $m_i \in [k] \cup \{\bullet\}.$ Given such a word $T$ for $\tau + \chi_\tau$ with degree $d = \chi_{\tau}+g$, define 
\begin{equation}
    \na^{i}_{j} = T^{\leq i}(j) + e_{\tau}(j) + i
\end{equation}
for each $i \in [g+1]$ and $j\in [k]$.
And define
\begin{equation}
    \nb^{i}_{j} = d + \tau(j) + e_{\tau}(j) - \na^{i}_{j}
\end{equation}
 for each $1 \leq i \leq g$ and $j\in [k]$.

We will use $\na^{i}_{j}$ to prescribe the order of vanishing of the section $\mu^{i+1}_j$ at $p^{i}$ on $E^{i+1}$ and $\nb^{i}_{j}$ to prescribe the order of vanishing of the section $\mu^{i}_j$ at $p^{i}$ on $E^{i}$.

Since
\begin{equation*}
   T^{\leq i}(j) + e_{\tau}(j) + i \geq (j-i) + (-j) + i = 0,
\end{equation*}
and
\begin{align*}
      d+\tau(j) - T^{\leq i}(j) - i &=  d+(T^{\leq g}(j) + g-d)- T^{\leq i}(j) - i \\
      & = g + (T^{\leq g}(j)-  T^{\leq i}(j)) - i\\
      & \geq g - (g-i) - i = 0,\\
\end{align*}
we know $\na^i_j, \nb^i_j \geq 0$ for all $i,j$.
Critically, this choice of ramification ensures the constructed sections are refined at all nodes and non-colliding at $p^{i-1}$ on $E^{i}$ (\cref{dfn:non-colliding}). We now verify the latter property.

\begin{lem}\label{lem:verify_noncolliding}
If there exists sections $\mu^{i}_{j} \in H^0(E^i, L^i(e_{\tau}(j)p^{i-1} + \tau(j)p^i)$ with order of vanishing exactly $\na^i_{j}$ at $p^{i-1}$, then they are non-colliding at $p^{i-1}$.
\end{lem} 

\begin{proof}
   Fix $j_1\neq j_2 \in [k]$. Suppose $t = (a,b) \in T_k$ is a pair such that
   \begin{equation*}
       t_1 \coloneqq (\tau(j_1),e_{\tau}(j_1)) \prec t \text{ and }   t_2 \coloneqq (\tau(j_2),e_{\tau}(j_2)) \prec t.
   \end{equation*}
   Let $s_{t_\ell,t}$ for $\ell = 1,2$ be as in \cref{dfn:order}. Then, since $s_{t_\ell,t} \equiv b-e_{\tau}(j_\ell)\pmod{k}$, we have
   \begin{equation*}
       \na^{i}_{j_{\ell}} + s_{t_\ell,t} \equiv T^{\leq i}(j_\ell) + i + b  \pmod{k},
   \end{equation*}
   and since $T^{\leq i}(j_1) \not \equiv T^{\leq i}(j_2) \pmod{k}$, we see that 
    \begin{equation*}
       \na^{i}_{j_{1}} + s_{t_1,t} \not\equiv \na^{i}_{j_{2}} + s_{t_2,t}  \pmod{k},
   \end{equation*}
   verifying the twists are non-colliding.
\end{proof}

The following lemma proves that we can rigidify all general bundles in $W^{T}(X,p^g,p^0)$ via $\tau$-positive limit linear series with the desired ramification.

\begin{lem}\label{lem:existence_of_lls}
Let $\tau \in \wt{\Sigma}_k$. Let $T = \sigma^k_{m_g}\dots \sigma^k_{m_1}$ be a reduced word for $\tau+\chi_{\tau}$ and let $\na^{i}_{j}, \nb^{i}_j$ be as defined above. A general limit line bundle in $W^{T}(X,p^g,p^0)$ possesses a unique $\tau$-positive limit linear series whose ramification indices at $p^i$ are $\na^{i}_j$ on $E^{i+1}$ and $\nb^{i}_j$ on $E^i$. This $\tau$-positive limit linear series is refined at all nodes and non-colliding at $p^{i-1}$ for the $E^i$-aspect.
\end{lem}

\begin{proof}
If such a limit linear series exists, it is refined at each node since
\begin{equation*}
    \na^{i}_{j} + \nb^{i}_{j} = d + \tau(j) + e_{\tau}(j),
\end{equation*}
and non-colliding at $p^{i-1}$ for the $E^{i}$-aspect by \cref{lem:verify_noncolliding}.

Let $D^i_j \coloneqq e_{\tau}(j)p^{i-1} + \tau(j)p^i$ be the essential twists corresponding to $\tau$ on $E^i$ and let
\begin{equation*}
    M^i_{j} \coloneqq  L^{i}(D^i_j)(-\na^{i-1}_j p^{i-1} -\nb^{i}_j p^i).
\end{equation*}
\textbf{Step 1}: Fix $i$. If $m_i = \bullet$, then $L^i$ is a general line bundle of degree $d$ on $E^i$. We have $T^{\leq i}(j) = T^{\leq i-1}(j)$ for each $j$, so
\begin{equation}\label{eq:genl_bundle}
    \na^{i-1}_j + \nb^{i}_j = T^{\leq i-1}(j)+e_{\tau}(j)+i-1 + (d+\tau(j)  -T^{\leq i}(j)-i)=d+e_{\tau}(j)+\tau(j)-1.
\end{equation}
Hence, $M^i_j$ is a general degree 1 line bundle on $E^i$, and so there exists a unique section (up to multiplication by constants) $\mu^i_{j}\in H^0(E^i, L^i(D^i_j))$ that vanishes exactly to orders $\na^{i-1}_j$ at $p^{i-1}$ and $\nb^{i}_j$ at $p^i$, for each $j \in [k]$.

\textbf{Step 2}: Now, suppose $m_i \neq \bullet$. Since $m_i \equiv T^{\leq i}(j_-) \pmod{k}$ (see \eqref{eq:truncations}), we know by definition of $W^T(X,p^g,p^0)$ that
\begin{align*}
    L^i &\cong \mO_{E^i}\BB(\B(T^{\leq i}(j_-) + i\B)p^{i-1} + \B(d-T^{\leq i}(j_-)-i\B)p^i\BB)\\
    &= \mO_{E^i}\BB(\B(\na^{i}_{j_-} - e_{\tau}(j_{-})\B)p^{i-1} + \B(\nb^{i}_{j_-}-\tau(j_{-})\B)p^i\BB).\numberthis \label{eq:L_i}
\end{align*}
Also, observe that 
\begin{equation*}
    \na^{i}_{j_-} = T^{\leq i}(j_-)+e_{\tau}(j_{-})+i = T^{\leq i-1}(j_-)-1+e_{\tau}(j_{-})+i = \na^{i-1}_{j_-},
\end{equation*}
so $\nb^{i}_{j_-} = \nb^{i-1}_{j_-}$ as well.

\textbf{Case 1}: Suppose $j\not\in \{j_\pm\}$. As in \eqref{eq:genl_bundle}, we see $M^i_j$ 
is a degree 1 line bundle on $E^i$; suppose it is isomorphic to $\mO_{E^i}(p^i)$. Then we must have 
\begin{equation*}
    \na^{i-1}_{j_-} - e_{\tau}(j_-) \equiv \na^{i-1}_j - e_{\tau}(j)  \pmod{k},
\end{equation*}
which is equivalent to $T^{\leq i-1}(j_-) \equiv T^{\leq i-1}(j) \pmod{k},$ and is not possible. By a similar argument, we see $M^i_j$ cannot be isomorphic to $\mO_{E^{i}}(p^{i-1})$. Thus, there is a unique section $\mu^i_j \in H^0(E^i, L^i(D^i_j))$ with the desired vanishing.

\textbf{Case 2}: Suppose $j=j_-$. Since \eqref{eq:L_i} implies $M^i_{j_{-}} \cong \mO_{E^i}$, we see there is a unique section $\mu^i_{j_{-}}\in H^0(E^i, L^i(D^i_{j_-}))$ with the desired vanishing. 

\textbf{Case 3}: Finally, suppose $j = j_+$. In this case, 
\begin{align*}
    \na^{i-1}_{j_+} + \nb^{i}_{j_+} & = (T^{\leq i-1}(j_+)+ e_{\tau}(j_+)+i-1) + (d+\tau(j_+)-T^{\leq i}(j_+)-i)\\
    & = \B((T^{\leq i}(j_+)-1)+e_{\tau}(j_+)+i-1\B) + (d + \tau(j_+) - T^{\leq i}(j_+)-i) \\
    &= d+\tau(j_+)+e_{\tau}(j_+)-2.
\end{align*}
So, since
\begin{equation*}
    \B(\na^{i-1}_{j_-} - e_{\tau}(j_-)\B)+ \B(e_{\tau}(j_+) - \na^{i-1}_{j_+}\B) = T^{\leq i-1}(j_-)-T^{\leq i-1}(j_+) \equiv 1 \pmod{k},
\end{equation*}
 we see $M^i_{j_+} \cong  \mO_{E^i}(p^{i-1} + p^{i})$. Thus, there does exist a section $\mu^i_{j_+} \in H^0(E^i, L^i(D^i_{j_+}))$ with at least the desired vanishing. To show that there exists a unique section with the exact vanishing, it suffices to show that the 1-dimensional space of sections
 \begin{equation*}
     H^0(\mO_{E^i}) \subset H^0(\mO_{E^i}(p^{i-1} + p^{i}))= H^0(E^i, M^i_{j_+}) \subset H^0(E^i,L^i(D^i_{j_+}))
 \end{equation*}
vanishing to order $\na^{i-1}_{j_+}+1$ at $p^{i-1}$ and $\nb^{i}_{j_+}$+1 at $p^i$ is contained in the image of the multiplication map
 \begin{equation*}
     H^0(E^i, L^i(D^i_{j_-}))\otimes H^0(E^i, \mO_{E^i}(D^i_{j_+}-D^i_{j_-})) \to H^0(E^i, L^i(D^i_{j_+})).
 \end{equation*}
 Indeed, since $T^{\leq i}(j_+) = T^{\leq i}(j_-) + qk + 1$ for some $q \geq 0$ by \cref{prop:increasing_decreasing}(1), we can write
\begin{align*}
    \nb^{ i}_{j_+} + 1 & = d + \tau(j_+) - T^{\leq i}(j_+)-i+1\\
    & = d + \tau(j_+) - (T^{\leq i}(j_-) + qk + 1) - i + 1\\
    & = \nb^{i}_{j_-} + \tau(j_+) - \tau(j_-) - qk,
\end{align*}
and by \cref{prop:increasing_decreasing}(3),
\begin{equation}\label{eq:ge}
    \tau(j_+) - \tau(j_-) -qk\geq 1.
\end{equation}

Also, 
\begin{equation*}
     \na^{i-1}_{j_+} + 1 =  \na^{i-1}_{j_-}+ e_{\tau}(j_+) - e_{\tau}(j_-) + qk.
\end{equation*}
If $q = 0$, by \cref{prop:increasing_decreasing}(2) we have $j_+ < j_-$ and so \eqref{eq:ge} implies that $e_{\tau}(j_-) = -j_-$, ensuring that
\begin{equation}\label{eq:ge_2}
    e_{\tau}(j_+) - e_{\tau}(j_-) > 0.
\end{equation}
Together \eqref{eq:ge} and \eqref{eq:ge_2} imply that 
\begin{equation*}
    D \coloneqq (e_{\tau}(j_+) - e_{\tau}(j_-)+qk)p^{i-1} + (\tau(j_+)-\tau(j_-) - qk)p^{i}
\end{equation*}
is an effective divisor and so the constant global section of $\mO_{E_i}(D) \cong \mO_{E_i}(D^{i}_{j_+}-D^{i}_{j_-})$ multiplied with $\mu^i_{j_-}$ gives us the section of $H^0(L^i(D^i_{j_{+}}))$ vanishing to order $\na^{i-1}_{j_+}+1$ at $p^{i-1}$ and $\nb^{i}_{j_+}$+1 at $p^i$.
\end{proof}

\subsection{Set-theoretic regeneration}

We are now ready to prove the regeneration theorem for refined and non-colliding $\tau$-positive limit linear series.

\begin{thm}\label{thm:regeneration}
  Let $\tau \in \wt{\Sigma}_{k}$. Suppose $\na^i_j$ and $\nb^i_j$ are the ramification indices at the $p^i$'s for a refined $\tau$-positive limit linear series that is non-colliding at $p^{i-1}$ for the $E^i$-aspect. Let $\mathcal{X}\to \mathcal{P} \to B$ be the family of curves as in \cref{sec:degeneration}. There is a quasi-projective scheme $\wt{W}$ over $B$ whose generic fiber is contained in the space of $\tau$-positive linear series on $\mathcal{X}^*$ and whose special fiber is the space of $\tau$-positive limit linear series on $X$ with ramification specified by $\na^i_j$ and $\nb^i_j$. Every component of $\wt{W}$ has dimension $\geq \dim \Pic(\mathcal{X}/B) - \inv_k(\tau)$.
\end{thm}

\begin{proof}

Let $\Pic = \Pic^d(\mathcal{X}/B)$ and let $\mL$ be the universal limit line bundle on $\mathcal{X} \times_{B} \Pic \xr{\pi} \Pic$. We have the following diagram:

\[\begin{tikzcd}
	{\mathcal{X}\times\text{Pic}} && {\mathcal{X}} \\
	&& {\mathcal{P}} \\
	{\text{Pic}} && B
	\arrow["{\text{pr}_1}", from=1-1, to=1-3]
	\arrow["{\pi}", from=1-1, to=3-1]
	\arrow["{\mathfrak{f}}", from=1-3, to=2-3]
	\arrow["\alpha", from=2-3, to=3-3]
	\arrow["\theta"', from=3-1, to=3-3]
\end{tikzcd}\]

Let $A$ be an effective divisor of relative degree $N$ on $\mathcal{X} \to B$ meeting each component of the special fiber in sufficiently high degree.  Let $\mf_q$ and $\mf_p$ be the divisors on $\mathcal{X}$ consisting of the points of total ramification on the generic fiber and specializing to $p^0$ and $p^g$ respectively on the special fiber (\cref{fig:family}). Define the divisors $D_j\coloneqq \tau(j) \cdot \mf_p+ e_{\tau}(j)\cdot \mf_q$ for $j\in [k]$, which correspond to the essential twists for $\tau$ on the generic fiber. For convenience of notation, we will also use $A$ and $D_j$ to denote their respective pullbacks along $\text{pr}_1$. We use the order $\prec$ on $[k]$ naturally inherited from the set of twists $\{(\tau(j),e_{\tau}(j))\}_{j\in [k]}$ (see \cref{dfn:order}).

\textbf{Setup.}
For each component $E^i$ and $j\in [k]$, we will inductively define a vector bundle $\mathcal{Q}^i_j$ and its projectivization will be denoted $P^i_j\coloneqq\PP \mathcal{Q}^i_j$. For $j \in [k]$, we define
\begin{align*}
    N^{< j} \coloneqq \{j'\in [k]: j' \prec j \text{ and there exists no } j'' \in [k] \text{ such that } j' \prec j'' \prec j \}.
\end{align*}
For a subset $J\subseteq [k]$, we define $N^{< J} = \bigcup_{j\in J} N^{< j}$. The bundle $\mathcal{Q}^{i}_j$ will be defined over an open 
\begin{equation*}
    U^i_{j} \subseteq  H^i_{j}, 
\end{equation*}
where $H^i_j$ will be defined below using the bundles $P^i_{j'}$ for $j'\in N^{< j}$.

First, for each minimal element $j\in [k]$, define $\mathcal{Q}^i_{j} \coloneqq \pi_* \mL(D_j)^i(A)$, which, by cohomology and base change, is a vector bundle of rank $N+\tau(j)+e_{\tau}(j)+\chi_{\tau}+1$ over 
\begin{equation*}
U^i_j = H^i_j = \Pic^{d+\tau(j)+e_{\tau}(j)+N}(\mathcal{X}/B),
\end{equation*} identified with $\Pic$ by tensoring the latter with $\mO_{\mathcal{X}}(D_j)^i(A)$. The projective bundle $P^i_{j}$ comes equipped with a map to $\Pic$ which we denote by $\psi^i_{j}$. 

 In order to define the bundle $\mathcal{Q}^{i}_j$ for non-minimal $j$, we also define a fiber product $P^i(J)$ associated to any subset $J \subseteq [k]$. We will simultaneously define $P^{i}(J)$ and $\mathcal{Q}^{i}_j$ by inducting on the maximum height of any element in $J$, which we denote by $\text{ht}(J)$, and the height of $j$ (we consider the height of a minimal element in $[k]$ to be 0). Define $P^i(\varnothing) = \Pic$ and for a subset $J$ with $\text{ht}(J) = 0$, define
\begin{equation*}
    P^i(J) = \prod^{\Pic}_{j \in J} P^i_{j},
\end{equation*}
which denotes a fiber product over $\Pic$.
Suppose we have defined $P^i(J)$ for all subsets $J$ with $\text{ht}(J) \leq n$ and $P^{i}_{j} = \PP \mathcal{Q}^i_j$ for all elements with height $\leq n$. Now, we define the bundle $\mathcal{Q}^i_j$ for $j$ with height $n+1$. Let $H^i_j = P^i(N^{< j})$. For each $j' \prec j$ there exists $0 \leq q_{j',j}, 0 \leq r_{j',j},s_{j',j} < k$ such that
\begin{equation}\label{eq:twist_up2}
    \mO_{\mathcal{X}}(D_{j'})^i\otimes \mO_{\mathcal{X}}\B((q_{j',j}k + r_{j',j})\mf_p+s_{j',j}\mf_q\B)^i \cong \mO_{\mathcal{X}}(D_j)^i.
\end{equation}

Let $\eta_i$ be the constant section of $\mO_{\mathcal{X}}(\mf_p)^i$ and let $\beta_i$ be the constant section of $\mO_{\mathcal{X}}(\mf_q)^i$; note that $\mO_{\mathcal{X}}(\mf_{p})^{i}|_{E^i} \cong \mO_{E^i}(p^i)$ and $\mO_{\mathcal{X}}(\mf_q)^i|_{E^i} \cong \mO_{E^i}(p^{i-1})$. Let $\mathcal{S}^i_{j'}$  be the the tautological subbundle on $P^{i}_{j'}$; $\mathcal{S}^i_{j'}$ corresponds to the new section in the $j$-th twist. Denote by $\pi^i_{j',j}: H^i_j \to P^i_{j'}$ the projection map and denote by $\varphi^i_{j}: H^i_{j}\to \Pic$ the map induced by composition of (any) projection maps down to $\Pic$. Also, define
\begin{equation*}
    V^i_{j'} \coloneqq \pi^{i*}_{j',j}S^{i}_{j'}\otimes \varphi^{i*}_j\B(\theta^*\alpha_{*}\mathcal{O}_{\mathcal{P}}(q_{j',j})^i\B).
\end{equation*}
Then we have a natural multiplication map
\begin{equation*}
    V^{i}_{j'} \xr{\cdot \eta_i^{r_{j',j}} \cdot \beta_{i}^{s_{j',j}}} \varphi^{i*}_j \B(\pi_* \mL(D_j)^i(A)\B)
\end{equation*}
 giving us a map of vector bundles over $H^i_j$ 
 \begin{equation*}
    \Pi^i_{j}: \prod_{j' \prec j} V^i_{j'} \to \varphi^{i*}_j \B(\pi_* \mL(D_j)^i(A)\B),
 \end{equation*}
 that has rank at most $s_{\tau}\B(\tau(j)+1,-e_{\tau}(j)\B)-1$ by \cref{lem:unique_new_section}. Let $U^i_j \subseteq H^i_j$ be the open over which $\Pi^i_j$ has rank exactly $s_{\tau}\B(\tau(j)+1,-e_{\tau}(j)\B)-1$; $U^{i}_{j} \neq \varnothing$ since it is nonempty over the special fiber. Let $\mathcal{Q}^i_j$ be the cokernel of $\Pi^i_j$ over $U^i_j$ and let $\psi^i_j$ be the natural map $P^{i}_j = \PP\mathcal{Q}^i_{j} \to H^i_{j}$. 
 
 Next, we define $P^i(J)$ for all subsets $J$ with $\text{ht}(J) \leq n+1$. For any element $j \in [k]$ with height $n+1$, define $P^i(\{j\}) = P^i_{j}$. Further, given a $J$ with $\text{ht}(J) \leq n+1$ such that $P^i(J)$ is defined and any $j \in [k]$ with height $\leq n+1$ define
\[P^i(J \cup \{j\}) = P^i(J) \times_{P^i(N^{< J} \cap N^{< j})} P^i_{j}.\]
This defines $P^i(J)$ for all subsets $J$ with height at most $n+1$. We have thus inductively defined the bundles $P^{i}_{j}$ for all $j\in [k]$.

The rank $q_j\coloneqq \rk \mathcal{Q}^i_j$ is
  \begin{align*}
    q_j &= \rk\BB( \varphi^{i*}_j \pi_* \B(\mL(D_j)^i(J)\B)\BB)- \Big(s_{\tau}\big(\tau(j)+1,-e_{\tau}(j)\big)-1\Big)\\
        & = \BB(N+ d-g+1 + \tau(j)+e_{\tau}(j)\BB) - \Big(s_{\tau}\B(\tau(j)+1,-e_{\tau}(j)\B)-1\Big)\\
        & = N - h_{\tau}\B(\tau(j)+1,-e_{\tau}(j)\B)+1,\numberthis \label{eq:q_j}
  \end{align*}
where $h_{\tau}$ is as defined in \eqref{eq:h_tau}.
Let $\wt{P}^i_j \to P^{i}_j$ be the space of lifts of $1$-dimensional subspaces of $\mathcal{Q}^i_j$ to $1$-dimensional subspaces of $\varphi^{i*}_j \pi_* \mL(D_j)^i(A)$. See \cref{fig:tower} for a schematic diagram showing the various spaces and maps associated to $j\in [k]$.

\begin{figure}
    \centering
\begin{tikzcd}
	{P^i_j = \mathbb{P}\mathcal{Q}^i_j} & {\mathbb{P}(\varphi^{i*}_{j}\pi_*\mathcal{L}(D_j)^i(A))} & {\widetilde{P}^i_j} \\
	{U^i_j} & {H^i_j} && {} \\
	{P^i_{j'}} & \dots \\
	& \vdots \\
	{P^{i}_{j''}} & \dots \\
	& {\text{Pic}}
	\arrow["{\psi^i_j}"', from=1-1, to=2-1]
	\arrow[from=1-2, to=1-1]
	\arrow[from=1-2, to=2-2]
	\arrow["{\text{open}}"', hook', from=1-3, to=1-2]
	\arrow[from=1-3, to=2-2]
	\arrow["{\text{open}}", hook, from=2-1, to=2-2]
	\arrow[shift left=2, from=2-2, to=3-1]
	\arrow[dashed, from=2-2, to=3-2]
	\arrow["{\exists! \hspace{1mm} \varphi^i_{j}}", shift left=2, curve={height=-18pt}, from=2-2, to=6-2]
	\arrow[from=3-1, to=4-2]
	\arrow[from=4-2, to=5-1]
	\arrow[from=5-1, to=6-2]
	\arrow[dashed, from=5-2, to=6-2]
\end{tikzcd}
    \caption{The tower of bundles $P^i_j$ coming from the poset $[k]$.}
    \label{fig:tower}
\end{figure}
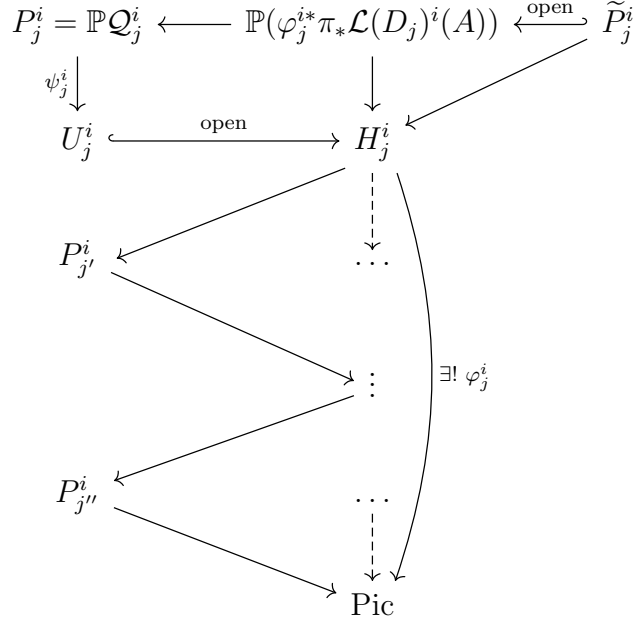

\textbf{Construction of the frame space.} We next define a frame space $F^{i,p}$ that will keep track of the $E^i$-aspect of the $\tau$-positive limit linear series for each node $p$ on $E^i$. We build $F^{i,p}$ iteratively based on the height of an element. Let $j_{\max}$ be the unique maximal element of $[k]$.
Define
\begin{equation*}
    F^{i,p}_{0}= \prod^{\Pic}_{j: \text{ht}(j)=0} \wt{P}^i_{j}.
\end{equation*}
Let $j_{\ell, 1},\dots,j_{\ell, m_{\ell}}$ be all the distinct elements of $[k]$ having height $1 \leq \ell \leq \text{ht}(j_{\max})$. Then define
\begin{equation*}
    F^{i,p}_{\ell} = F^{i,p}_{\ell-1} \times_{H^{i}_{j_{\ell, 1}}} \wt{P}^{i}_{j_{\ell, 1}}\times_{H^{i}_{j_{\ell, 2}}}  \wt{P}^{i}_{j_{\ell, 2}} \times_{H^{i}_{j_{\ell, 3}}} \dots \times_{H^{i}_{j_{\ell,m_{\ell}}}}  \wt{P}^{i}_{j_{\ell, m_{\ell}}},
\end{equation*}
 and let $F^{i,p} = F^{i,p}_{\text{ht}(j_{\max})}$. Then define $F$ to be the `master frame space' that collects all the sections on all the components into one space: 
  \begin{equation*}
      F \coloneqq F^{1, p^1} \times F^{2,p^1} \times F^{2, p^2} \times F^{3, p^2}\times F^{3, p^3} \times \cdots \times F^{g-1, p^{g-2}}\times F^{g-1, p^{g-1}} \times F^{g,p^{g-1}},
  \end{equation*}
  all products being over $\Pic$. Denoting $P^i \coloneqq P^{i}_{j_{\max}}$, we see that $F$ maps to 
  \begin{equation*}
      P\coloneqq P^{1}\times P^{2} \times P^{2} \times P^{3} \times P^{3}\times \dots \times P^{g-1}\times P^{g-1}\times P^{g},
  \end{equation*}
the products again being over $\Pic$. We now compute the dimension of $F$. The relative dimension of $\wt{P}^{i}_{j}$ over $P^i_{j}$ is $
    \rk\BB(\varphi^{i*}_j\pi_* \B(\mL(D_{k-1})^i(A)\B)\BB) - q_j$
and the relative dimension of $P^i_{j}$ over $H^{i}_{j}$ is $q_j-1$, so the relative dimension of $\wt{P}^{i}_{j}$ over $H^i_{j}$ is
\begin{equation*}
     \rk\BB(\varphi^{i*}_j\pi_* \B(\mL(D_{k-1})^i(A)\B)\BB) -1 = N + \tau(j) + e_{\tau}(j) + \chi_{\tau}.
\end{equation*}
 So, since $\wt{P}^{i}_{j}$ for each $j\in [k]$ appears precisely once in $F^{i,p}$ and is fibered over $H^{i}_{j}$, we have 
 \begin{equation*}
     \dim F^{i,p} = \dim \Pic + \sum_{j = 0}^{k-1} \dim_{H^{i}_{j}}  \wt{P}^{i}_{j} = \dim \Pic + \sum_{j=0}^{k-1}
      \B(N+\tau(j)+e_{\tau}(j)+\chi_{\tau}\B). 
 \end{equation*}
 Thus, 
  \begin{align*}
    \dim F &= \dim \Pic + (2g-2)\sum_{j=0}^{k-1}
      \B(N+\tau(j)+e_{\tau}(j)+\chi_{\tau}\B). \numberthis \label{eq:dim_F}
  \end{align*}

  We now construct a subvariety $\wt{W}^{\text{fr}} \subset F$ that will parameterize compatible projective frames for $\tau$-positive limit linear series by imposing three sets of conditions, namely:
\begin{enumerate}
    \item Sections on the same component $E^i$ are equal;
    \item Sections vanish along the divisor $A$;
    \item Sections above the special fiber have the desired ramification $\na^i_j, \nb^i_j$ imposed by $\tau$.
\end{enumerate}
    \textbf{Compatibility on same component}: We restrict to the diagonal in $P$ along the inner components of the product in order to ensure sections in $F$ on the same component are equal. This compatibility requirement imposes
  \begin{align*}
      \dim_{\Pic} P^{2} + \dots \dim_{\Pic} P^{g-1} & = (g-2)\sum_{j=0}^{k-1}(q_j-1)\\
      & = (g-2)\sum_{j=0}^{k-1}\BB(N-h_{\tau}\B(\tau(j)+1,-e_{\tau}(j)\B)\BB)\numberthis \label{eq:same_comp}
  \end{align*}
  conditions.

    \textbf{Vanishing along $A$}: We impose vanishing of sections along the divisor $A$ that has relative degree $N$. After \'{e}tale base change, we may assume that each component of $A$ meets only one $E^i$. Let $A^i$ be the union of components of $A$ meeting $E^i$. Let $Z^i_j \subset P^i_j$ be the locus of sections that vanishes along $A^i$. We compute an upper bound for the codimension of $Z^i_j$: it is cut out by the vanishing of the composition
    \begin{equation}
        S^i_j \to {\psi^{i}_{j}}^* \mathcal{Q}^i_j \to (\varphi^{i}_j\circ \psi^i_j)^* \B(\pi_* \mL(D_j)^i(A)\otimes \mO_{A^i}\B),
    \end{equation}
    which represents $\deg A^i$ equations. Summing over $j \in [k]$ followed by summing over $1\leq i \leq g$, we see that the vanishing of sections in $F$ along $A$ can be imposed using 
    \begin{equation}\label{eq:vanishing}
     \leq k\sum_{i=1}^{g}\deg A^i = kN
    \end{equation} equations.

   \textbf{Compatibility at nodes}: Next, we impose the desired ramification at $p^i$ on the special fiber by describing suitable equations on $F^{i,p^i}\times F^{i+1,p^i}$. Let $\sigma^{i}_{j}$ and $\lambda^{i}_{j}$ be coordinates on $\wt{P}^i_j$ and $\wt{P}^{i+1}_j$ respectively. Denote by $X^{\leq i}$ the chain $E^1 \cup \dots \cup E^i$ and by $X^{> i}$ the chain $E^{i+1} \cup \dots \cup E^g$.  Let $\zeta^i$ the constant section of $\mO_{\mathcal{X}}(X^{\leq i})$, vanishing to the left of $p^i$, and $\gamma^i$ the constant section of $\mO_{\mathcal{X}}(X^{> i})$, vanishing to the right of $p^i$. Define a closed subvariety $Y^i$ of $F^{i,p^i}\times F^{i+1,p^i}$ by the equations 
   \begin{equation}\label{eq:ramification}
       \sigma^{i}_{j} \otimes (\zeta^i)^{\na^i_j} =  \lambda^{i}_{j} \otimes (\gamma^i)^{\nb^i_j},
   \end{equation}
   for $j \in [k]$. The sections appearing in \eqref{eq:ramification} are viewed elements of the projectivization of
   \begin{equation*}
       \pi_* \mL(D_j)^i(A+\na^i_j X^{\leq i}) \cong  \pi_* \mL(D_j)^{i+1}(A+\nb^i_j X^{> i}),
   \end{equation*}
   where the isomorphism of bundles is a result of the isomorphisms
   \begin{equation*}
       \mL(D_j)^i \cong \mL(D_j)^{i+1}\BB(-\B(d+\tau(j)+e_{\tau}(j)\B) X^{\leq i}\BB),
   \end{equation*}
$\mO_{\mathcal{X}}(X^{\leq i}) \cong \mO_{\mathcal{X}}(-X^{> i})$, and $\na^{i}_{j} + \nb^{i}_{j} = d + \tau(j) + e_{\tau}(j)$.

Away from the special fiber, the sections $\zeta^i$ and $\gamma^i$ are nonzero, so \eqref{eq:ramification} simply asserts that $\sigma^i_j$ and $\lambda^i_j$ are equal up to scaling by a constant.

On the special fiber, \eqref{eq:ramification} asserts that $\sigma^i_j$ and $\lambda^i_j$ are determined by their restrictions to $X^{\leq i}$ and $X^{> i}$ respectively which are subject to the following conditions:
\begin{equation}
    \sigma^i_j \mid_{X^{\leq i}} \text{ vanishes to order } \geq \nb^i_j \text{ at } p^i,
\end{equation}
\begin{equation}
    \lambda^i_j \mid_{X^{> i}} \text{ vanishes to order } \geq \na^i_j \text{ at } p^i,
\end{equation}
and, given these conditions we can thus write
\begin{equation}\label{eq:relations}
    \lambda^i_j \mid_{X^{\leq i}}  =  \sigma^i_j\mid_{X^{\leq i}} \cdot \frac{(\zeta^i)^{\na^i_j}\mid_{X^{\leq i}}}{(\gamma^i)^{\nb^i_j}\mid_{X^{\leq i}}} \text{ and }  \sigma^i_j \mid_{X^{> i}}  =  \lambda^i_j\mid_{X^{> i}} \cdot \frac{(\gamma^i)^{\nb^i_j}\mid_{X^{> i}}}{(\zeta^i)^{\na^i_j}\mid_{X^{> i}}}.
\end{equation}
 Whenever $\na^i_j$ and $\nb^i_j$ are nonzero, \eqref{eq:relations} simplifies to $\lambda^i_j \mid_{X^{\leq i}} = 0$ and $\sigma^i_j \mid_{X^{> i}} = 0$.

 For each $j$, \eqref{eq:ramification} thus leads to
 \begin{equation*}
     \rk \pi_* \mL(D_j)^i(A+\na^i_j X^{\leq i}) - 1 = N + \tau(j)+e_{\tau}(j)+\chi_{\tau}
 \end{equation*}
equations, and so  
\begin{equation*}
\text{codim}(Y^i \subset  F^{i,p^i}\times  F^{i+1,p^i})\leq \sum_{j=0}^{k-1}\B(N + \tau(j)+e_{\tau}(j)+\chi_{\tau}\B). 
\end{equation*}
Altogether, 
\begin{equation}\label{eq:nodes}
    \text{codim}(Y^1\times Y^2 \times \cdots \times Y^{g-1} \subset  F) \leq (g-1)\sum_{j=0}^{k-1}\B(N + \tau(j)+e_{\tau}(j)+\chi_{\tau}\B).
\end{equation}
We also remove the closed locus in $Y^1\times Y^2 \times \cdots \times Y^{g-1}$ where the order of vanishing along the special fiber is greater than the desired ramification.

Imposing the three sets of conditions \eqref{eq:same_comp}, \eqref{eq:vanishing}, and \eqref{eq:nodes} results in a closed subvariety $\wt{W}^{\text{fr}} \subset F$ and denote its image in $P$ by $\wt{W}$. 

   \textbf{Fiber dimensions}: In order to use the codimension bound for $\wt{W}^{\text{fr}}$ to obtain a codimension bound for $\wt{W}$ we now determine the dimension of fibers $\wt{W}^{\text{fr}} \to W$.

On the generic fiber our conditions just impose equality of $(g-1)$ pairs of sections across the same nodes, and since the relative dimension of each $\wt{P}^{i}_{j}$ over $P^{i}_{j}$ is $s_{\tau}\B(\tau(j)+1,-e_{\tau}(j)\B)-1$,  the dimension of the generic fiber $\wt{W}^{\text{fr}}\to W$ is 
\begin{equation}
    (g-1)\sum_{j=0}^{k-1}\BB(s_{\tau}\B(\tau(j)+1,-e_{\tau}(j)\B) - 1\BB).
\end{equation}
On the special fiber, given a point $\{\sigma^i_j, \lambda^i_j\}_{j\in[k]}$ of $\wt{W}^{\text{fr}}$, all other points in the same fiber over $P$ have the following form: to $\sigma^i_j$ we may add the image of any section $\sigma^i_{j'}$ with $j' \prec j$ having higher vanishing order at $p^i$ than $\sigma^i_j$. Explicitly, let $0\leq q_{j',j}$ and $0\leq r_{j',j},s_{j',j} < k$ be integers as in \eqref{eq:twist_up}. Then to $\sigma^i_j$ we can add any section of the form $\sigma^i_{j'}\beta_i^{\delta \cdot k+ s_{j',j}}\eta_i^{(q_{j',j}-\delta)\cdot k + r_{j',j}}$ such that $0 \leq \delta \leq q_{j',j}$ and
\begin{equation}\label{eq:bij}
    \nb^i_{j'} + (q_{j',j}-\delta) k + r_{j',j}  > \nb^i_j.
\end{equation}
Similarly, to $\lambda^i_j$ we may add any section of the  form $\lambda^i_{j'}\beta_i^{\delta \cdot k+ s_{j',j}}\eta_i^{(q_{j',j}-\delta)\cdot k + r_{j',j}}$ such that $0 \leq \delta \leq q_{j',j}$ and 
\begin{equation}\label{eq:aij}
    \na^i_{j'} + \delta k + s_{j',j} > \na^i_j.
\end{equation}
Thus, we see that the contribution to the special fiber for each $j$ comes from counting for each $j' \prec j$ the number of $\delta$'s satisfying \eqref{eq:bij} plus those $\delta$'s satisfying \eqref{eq:aij}. 

Note that a $\delta$ satisfies \eqref{eq:bij} if and only if 
\begin{equation*}
    \na^i_{j'} + \delta k + s_{j',j} < \na^i_j
\end{equation*}
 since $\na^i_\ell + \nb^i_{\ell} = d+\tau(\ell)+e_{\tau}(\ell)$ for each $\ell$ and 
\begin{equation*}
    \B(\tau(j)+e_{\tau}(j)\B)-\B(\tau(j')+e_{\tau}(j')\B) = q_{j',j}k+r_{j',j}+s_{j',j}.
\end{equation*}

Thus, a $\delta$ can only satisfy one of \eqref{eq:bij} or \eqref{eq:aij}, and it always does satisfy one because the $a^i_{\ell}$'s are non-colliding. So, each $0 \leq \delta \leq q_{j',j}$ satisfies exactly one of the inequalities. Thus, summing over the $(g-1)$ nodes and using \cref{lem:unique_new_section}, the dimension of the special fiber of $\wt{W}^{\text{fr}}$ over $W$ is 
\begin{equation}\label{eq:fiber_dim}
    (g-1)\sum_{j=0}^{k-1} \sum_{j' \prec j} (q_{j',j}+1) =  (g-1)\sum_{j=0}^{k-1}  \BB(s_{\tau}\B(\tau(j)+1,-e_{\tau}(j)\B)-1\BB).
\end{equation}

   \textbf{Final codimension bound}: Putting together the above dimension counts, every component $\wt{W}' \subset \wt{W}$ has dimension
   \begin{align*}
       \dim \wt{W}' \geq & \dim F - (\text{number of defining equations}) - (\text{fiber dim})\\
     =& \dim \Pic + (2g-2)\sum_{j=0}^{k-1} 
      \B(N+\tau(j)+e_{\tau}(j)+\chi_{\tau}\B)  & \hspace{-1in}(\dim F,  \eqref{eq:dim_F})\\
      &-  (g-2)\sum_{j=0}^{k-1}\BB(N-h_{\tau}\B(\tau(j)+1,-e_{\tau}(j)\B)\BB) & \hspace{-1in}(\text{diagonal condition},  \eqref{eq:same_comp})\\
      &- kN &\hspace{-1in}(\text{vanishing on } D,  \eqref{eq:vanishing}) \\
      &-(g-1)\sum_{j=0}^{k-1}\BB(N + \tau(j)+e_{\tau}(j)+\chi_{\tau}\BB) &\hspace{-1in}(\text{ramification},  \eqref{eq:nodes})\\ 
      &-(g-1) \sum_{j=0}^{k-1}\BB(s_{\tau}\B(\tau(j)+1,-e_\tau(j)\B) - 1\BB) &\hspace{-1in}(\text{fiber dim},  \eqref{eq:fiber_dim})\\
       =& \dim \Pic + (g-2)\sum_{j=0}^{k-1}h_{\tau}\B(\tau(j)+1,-e_{\tau}(j)\B) \\
      &-(g-1)\sum_{j=0}^{k-1}\BB(s_{\tau}\B(\tau(j)+1,-e_\tau(j)\B) - \B(\tau(j)+1+e_{\tau}(j)+\chi_{\tau}\B)\BB)\\
      =& \dim \Pic -\sum_{j=0}^{k-1}h_{\tau}\B(\tau(j)+1,-e_{\tau}(j)\B) &\hspace{-1in}(\text{definition of } h_{\tau}, \eqref{eq:h_tau})\\
      =& \;\dim \Pic  - \inv_{k}(\tau).  &\hspace{-1in}(\text{Lemma }\ref{lem:sum_h1_equals_inv})\\
   \end{align*}
\end{proof}

We can now deduce that every $\tau$-positive limit line bundle on the special fiber is the limit of some line bundle in $W^{\tau}(\mathcal{X}^*, \mathfrak{f}_p|_{\mathcal{X}^*},\mathfrak{f}_{q}|_{\mathcal{X}^*})$, completing the proof of set-theoretic regeneration.

\begin{cor}
    Let $T$ be a reduced word for $\tau \in \wt{\Sigma}_k$. Then $W^{T}(X,p^g,p^0)$ is contained in the closure of $W^{\tau}(\mathcal{X}^*, \mathfrak{f}_p|_{\mathcal{X}^*},\mathfrak{f}_{q}|_{\mathcal{X}^*})$.
\end{cor}
\begin{proof}
    By \cref{lem:existence_of_lls}, we know that a generic line bundle in $W^{T}(X,p^g,p^0)$ is in the image of $\wt{W} \to \Pic$, where $\wt{W}$ is the quasiprojective scheme we constructed in \cref{thm:regeneration}. Moreover, \cref{lem:existence_of_lls} also ensures that every generic line bundle in $W^{T}(X,p^g,p^0)$ has a unique preimage in $\wt{W}$ and since $\dim W^{T}(X,p^g,p^0) = g-\inv_{k}(\tau)$, we infer there exists an irreducible component $Y$ of the special fiber of $\wt{W}$ dominating $W^{T}(X,p^g,p^0)$ and  $\dim Y = g-\inv_k(\tau)$. By \cref{thm:regeneration}, the dimension of any irreducible component of $\wt{W}$ is  $> \dim Y$, so let $\wt{Y}$ be an irreducible component of $\wt{W}$ that contains $Y$ and it necessarily dominates $B$. Then the closure of the image of $\wt{Y}$ in $\Pic$ contains $W^{T}(X,p^g,p^0)$, and is itself contained in the closure of $W^{\tau}(\mathcal{X}^*, \mathfrak{f}_p|_{\mathcal{X}^*},\mathfrak{f}_{q}|_{\mathcal{X}^*})$ by \cref{prop:non-colliding}.
\end{proof}

\subsection{Reducedness and Cohen-Macaulayness}

We now deduce geometric properties of transmission loci using our regeneration theorem. First, we prove reducedness, which will also establish scheme-theoretic regeneration.
\begin{prop}
     Let $\tau\in \wt{\Sigma}_k$. The variety $W^{\tau}(X,p^g,p^0)$ is reduced and hence $W^{\tau}(C,p,q)$ is reduced for a general degree $k$, genus $g$ cover $C$ totally ramified at two points $p$ and $q$ such that $k$ is the smallest positive integer satisfying $\mO_C(k(p-q)) \cong \mO_C$.
\end{prop}
\begin{proof}
    The locus $W^{\tau}(X,p^g,p^0)$ is the intersection of determinantal loci of the form 
    \begin{equation*}
        W^{r}_{\vec{d},a,b}(X) = \{L\in \Pic^{\vec{d}}(X): h^0(L(ap^g-bp^0)) \geq r+1\}
    \end{equation*}
    for various $r,a,b,$ and degree distributions $\vec{d}$. As shown in \cite[Thm.\ 7.5]{larson_global_2025}, such loci are reduced and their intersections are also reduced, proving that $W^{\tau}(X,p^g,p^0)$ is reduced. Since geometric reducedness is open in flat families, we see that $W^{\tau}(C,p,q)$ is reduced for a general degree $k$, genus $g$ cover totally ramified at two points $p$ and $q$.
\end{proof}

We have thus shown that
\begin{equation*}
    \oo{W^{\tau}(\mathcal{X}^*,\mathfrak{f}_p|_{X^*},\mathfrak{f}_{q}|_{X^*})}|_{0} \subseteq W^{\tau}(X,p^g,p^0) = W^{\tau}(X,p^g,p^0)_{\text{red}} \subseteq  \oo{W^{\tau}(\mathcal{X}^*,\mathfrak{f}_p|_{X^*},\mathfrak{f}_{q}|_{X^*})}|_{0},
\end{equation*}
implying that $ \oo{W^{\tau}(\mathcal{X}^*,\mathfrak{f}_p|_{X^*},\mathfrak{f}_{q}|_{X^*})}|_{0} = W^{\tau}(X,p^g,p^0)$ as schemes.

We can also easily deduce Cohen-Macaulayness of transmission loci on the general fiber by doing so on the special fiber. 
\begin{thm}\label{thm:cm}
    Let $\tau\in \wt{\Sigma}_k$. The variety $W^{\tau}(X,p^g,p^0)$ is Cohen-Macaulay and hence $W^{\tau}(C,p,q)$ is Cohen-Macaulay for a general degree $k$, genus $g$ cover $C\to \PP^1$ totally ramified at two points $p$ and $q$ such that $k$ is the smallest positive integer satisfying $\mO_C(k(p-q)) \cong \mO_C$.
\end{thm}
\begin{proof}
    The proof is identical to that of \cite[Thm.\ 7.8]{larson_global_2025} \textit{mutatis mutandis}.
\end{proof}

\subsection{Connectedness}\label{sub:connected}

When $g > \inv_{k}(\tau)$, we can also prove connectedness of the $\tau$-transmission locus on the general fiber by doing so on the special fiber. To do so, we utilize \defi{braid moves} in $\Sigma_k$: given a reduced word $T = \sigma^k_{m_\ell}\dots \sigma^{k}_{m_1}$ for $\tau$ there are two types of moves we can perform under certain conditions to obtains other reduced words for $\tau$. The first move called a \defi{flip}, denoted $F^i$, can be applied whenever $m_i - m_{i+1} \neq \pm 1 \pmod{k}$ and is defined as
\[F^i T = \sigma^k_{m_\ell}\dots \sigma^k_{m_{i+2}}\sigma^k_{m_{i}}\sigma^k_{m_{i+1}}\sigma^k_{m_{i-1}}\dots\sigma^{k}_{m_1}.\]
The second move called a \defi{shuffle}, denoted $S^i$, is applicable whenever $m_i \equiv m_{i+2} \pmod{k}$ and $m_{i} \equiv m_{i+1} \pm 1 \pmod{k}$, and is defined as
\[S^i T = \sigma^k_{m_\ell}\dots \sigma^k_{m_{i+3}}\sigma^k_{m_{i+1}}\sigma^k_{m_{i}}\sigma^k_{m_{i+1}}\sigma^k_{m_{i-1}}\dots\sigma^{k}_{m_1}.\]
Since we can transform a reduced word for $\tau$ to any other one through a sequence of braid moves \cite[Theorem 3.3.1(ii)]{bjorner2006combinatorics}, to show $W^{\tau}(X,p^g,p^0)$ is connected, it suffices to show that $W^{T}(X,p^g,p^0)$ is connected with $W^{F^i T}(X,p^g,p^0)$ and $W^{S^i T}(X,p^g,p^0)$ whenever those braid moves are applicable and $g > \inv_{k}(\tau)$. The proof of both these facts is the same as that of \cite[Lemmas 8.4 \& 8.5]{larson_global_2025} \textit{mutatis mutandis}.

\begin{thm}\label{thm:connected}
   Let $\tau \in \wt{\Sigma}_k$ and let $C$ be a general degree $k$, genus $g$ cover of $\PP^1$ totally ramified at two points $p$ and $q$ such that $k$ is the smallest positive integer satisfying $\mO_C(k(p-q)) \cong \mO_C$. If $g > \inv_{k}(\tau)$, then $W^{\tau}(C,p,q)$ is connected.
\end{thm}

\section{Smoothness} \label{sec:smoothness}

In this section, we prove the smoothness of $W^{\tau}(C,p,q)$ away from all codimension 2 transmission loci, and as a consequence deduce normality and irreduciblity as well. We begin by constructing a Petri map for our setting, whose injectivity will yield smoothness along the open locus of type precisely $\tau$.

Let ${W^{\tau}}(C,p,q)^{\circ} \subset W^{\tau}(C,p,q)$ denote the open locus of line bundles $\mL$ satisfying $\tau_{\mL} = \tau$. For each $j$, let $D_j := \tau(j)\cdot p-j\cdot q$ and
\begin{equation*}
    S_j \coloneqq H^0(\mL(D_j))/H^0(\mL(D_j-p)).
\end{equation*}
Let $\mL \in \Pic^{\chi_{\tau}+g}(C)$. If $\tau_{\mL} = \tau$, then it is clear that $\dim S_{j} = 1$ for every $j\in [k]$. On the other hand, $\dim S_j = 1$ for all $j\in [k]$ implies that $\tau_{\mL}^{-1}(\tau(j)) \geq j$ for all $j\in [k]$, which along with $\chi_{\tau_{\mL}} = \chi_{\tau}$ forces $\tau_{\mL} = \tau$. Thus, $\mL \in W^{\tau}(C,p,q)^{\circ}$ if and only if $\dim S_j = 1$ for every $j\in [k]$.

Now let $\mL \in \Pic^d(C)$ satisfy $\tau_{\mL} = \tau$, and let $\xi \in T_{\mL}\Pic^d(C) \cong \text{Ext}^{1}(\mL,\mL) \cong H^1(C, \mO_C)$ be a tangent vector. Let $\mL_{\xi}$ denote the corresponding first-order deformation of $\mL$. For each $j$, tensoring the extension of $\mL$ corresponding to $\xi$ by $\mO_C(D_j)$ yields a short exact sequence
\begin{equation*}
    0 \to \mL(D_j) \to \mL_{\xi}(D_j) \to \mL(D_j) \to 0.
\end{equation*}
The associated long exact sequence in cohomology contains $H^0(C,\mL(D_j)) \xrightarrow{\partial_{\xi,j}} H^1(C,\mL(D_j))$,
where $\partial_{\xi,j}$ is the connecting homomorphism. A nonzero class in $S_j$ lifts to first order if and only if its image under the induced map $\bar{\partial}_{\xi,j}: S_j \to H^1(C,\mL(D_j))$
vanishes. Thus, if we define
\begin{equation*}
    \Phi_{\mL}: H^1(\mO_C) \to \bigoplus_{j=0}^{k-1}\Hom(S_j, H^1(\mL(D_j)))
\end{equation*}
by $\xi \mapsto (\bar{\partial}_{\xi,j})_j$, then
\begin{equation*}
    T_{\mL} W^{\tau}(C,p,q) = \ker \Phi_{\mL}.
\end{equation*}
Dualizing, we obtain a map
\begin{equation*}
    \Phi_{\mL}^{\vee}: \bigoplus_{j=0}^{k-1} S_j\otimes H^1(\mL(D_j))^{\vee} \to H^1(\mO_C)^{\vee},
\end{equation*}
which by Serre duality becomes
\begin{equation*}
    \Phi_{\mL}^{\vee}: \bigoplus_{j=0}^{k-1} S_j\otimes H^0(\mL^{\vee}(-D_j)\otimes \omega_{C}) \to H^0(\omega_C).
\end{equation*}

To compute the dimension of the domain of $\Phi^{\vee}_{\mL}$, we use the following lemma.

\begin{lem}\label{lem:h1_equal}
    For $\tau \in \wt{\Sigma}_k$, we have $h_{\tau}(\tau(j)+1, j) = h_{\tau}(\tau(j)+1, -e_{\tau}(j))$ for all $j\in[k]$.
\end{lem}
\begin{proof}
    If $e_{\tau}(j) = -j$, there is nothing to show. Now, suppose $e_{\tau}(j) = 0$, which implies that $\tau(j') < \tau(j)$ for all $0 \leq j' < j$. So, as in \eqref{eq:h_tau}, 
    \[h_{\tau}(\tau(j)+1,j) = \#\{n < j: \tau(n) > \tau(j)\} = \#\{n < 0: \tau(n)> \tau(j)\}= h_{\tau}(\tau(j)+1,0).\qedhere\]\end{proof}
Thus, by \cref{lem:h1_equal,lem:sum_h1_equals_inv}, the domain of $\Phi_\mL^\vee$ has dimension $\inv_k(\tau) = \text{codim }W^{\tau}(C,p,q)$. Therefore, to prove smoothness of $W^{\tau}(C,p,q)$ at $\mL$, it suffices to show that $\Phi_\mL^{\vee}$ is injective. We prove this by passing to the central fiber $X$ of our family $\mathcal{X}$.

Suppose $\mL^{*}$ is a line bundle in $W^{\tau}(\mathcal{X}^*, \mathfrak{f}_p|_{\mathcal{X}^*},\mathfrak{f}_{q}|_{\mathcal{X}^*})^{\circ}$. Let the limit of $\mL^*$ on the central fiber be associated to a reduced word $T = \sigma_{m_g}^k\dots \sigma_{m_1}^k$. Let $L(j)_{\vec{d}}$ denote the limit with degree distribution $\vec{d}$ on the central fiber of $\mL^{*}(\tau(j)\cdot\mathfrak{f}_p|_{\mathcal{X}^*} - j\cdot \mathfrak{f}_q|_{\mathcal{X}^*})$. For each $j\in[k]$, define a degree distribution $\vec{d}_j$ associated to $T$ by
\begin{equation}\label{eq:distribution}
\vec{d_j}(i) = \begin{cases}
0 &\text{ if }  j_{-}(i) = j\\
2 & \text{ if } j_+(i) = j\\
1 & \text{ else}.
\end{cases} \end{equation}
We show several important properties of $L(j)_{\vec{d}_{j}}$ in \cref{prop:distribution_properties}, including that $\vec{d}_j$ is indeed a degree distribution summing to $\deg(L(j))$.

\begin{prop}\label{prop:distribution_properties}
    The following statements are true for the degree distribution $\vec{d}_j$ associated to $T = \sigma^k_{m_g}\dots \sigma^k_{m_1}$:
    \begin{enumerate}
    \item We have $\#\{1 \leq i \leq g : j_{-}(i) = j\} = h_{\tau}(\tau(j)+1,j).$
    \item For each $1 \leq i \leq g$, we have $\sum_{i' < i} \vec{d}_{j}(i') = T^{\leq i-1}(j) -j + i-1$ and $\sum_{i' > i} \vec{d}_{j}(i') = d+\tau(j)-T^{\leq i}(j)-i$. In particular, $\vec{d}_j$ sums to $d + \tau(j) - j$.
    \item If $j  = j_{-}(i)$, then $L(j)_{\vec{d}_j}|_{E^i} \cong \mO_{E^i}$.
    \item If $j = j_{+}(i)$, then $L(j)_{\vec{d}_j}|_{E^i} \cong \mO_{E^i}(p^{i-1} + p^i)$.
    \item If $j \not \in \{j_{-}(i),j_{+}(i)\}$, then $L(j)_{\vec{d}_j}|_{E^i}$ is not isomorphic to either $\mO_{E^i}(p^{i-1})$ or $\mO_{E^i}(p^{i})$.
    \item We have $h^0(L(j)_{\vec{d}_j}^{\vee} \otimes \omega_X) = h_{\tau}(\tau(j)+1, j)$. 
    \end{enumerate}
\end{prop}

\begin{proof}
    We verify each statement in turn.

    \begin{enumerate}
        \item Note that 
        \[h_{\tau}(\tau(j)+1,j) = \#\{n < j: \tau(n) > \tau(j)\}\] counts the $k$-inversions $(m,n)$ such that $n \equiv j\pmod{k}$. As in the proof of \cref{prop:increasing_decreasing}, this is equal to $\sum_{j < j' \in [k]} \inv_k(\tau, j, j')$, which in turn is the number of indices $i$ such that $j_{-}(i) = j$.
        \item To prove the claim about the sum of $\vec{d}_j$ to the left of $i$ we argue by induction on $i$. When $i = 1$, the degree distribution to the left of $E^1$ is $0$, as claimed. Assume the formula holds for $i$. There are then three cases for $i+1$. If $j=j_{-}(i)$, the left-hand side remains unchanged, while the right-hand side increases by
        \begin{equation*}
        (T^{\leq i}(j)-T^{\leq i-1}(j)) + 1 = -1 + 1 = 0.
        \end{equation*}
        If $j= j_{+}(i)$, both sides increase by $2$. If $j \not \in \{j_{-}(i),j_{+}(i)\}$, both sides increase by $1$. The claim about the sum of $\vec{d}_{j}$ to the right of $i$ is proved similarly, by induction starting from $E^g$.
        \item Let $j = j_-(i)$. The claimed statement follows from Case 2 of \cref{lem:existence_of_lls}, but we indicate the argument again. As in \eqref{eq:L_i},
\begin{equation*}
L^i \cong \mO_{E^i}((T^{\leq i-1}(j) +i-1)p^{i-1} + (d-T^{\leq i}(j) - i)p^{i}),
\end{equation*}
and so
\begin{equation*}
L(j)^i \cong \mO_{E^i}\left(\sum_{i' < i} \vec{d}_j(i')p^{i-1} + \sum_{i' > i} \vec{d}_j(i')p^{i}\right),
\end{equation*}
and the claim follows.
        \item Similarly, this is shown in Case 3 of \cref{lem:existence_of_lls}.
        \item And, this is shown in Case 1 of \cref{lem:existence_of_lls}.
        \item Recall the dualizing sheaf $\omega_X$ of $X$ satisfies $\omega_X|_{E^1} \cong \mO_{E^1}(p^1)$, $\omega_X|_{E^g} \cong \mO_{E^g}(p^{g-1})$, and $\omega_X|_{E^i} \cong \mO_{E^i}(p^{i-1}+p^i)$ for $1<i<g$. We first show that every global section of $L(j)_{\vec{d}_j}^{\vee} \otimes \omega_X$ vanishes at all nodes. We prove this by induction, starting at $p^1$. If $j = j_{-}(1)$, then by (3) the restriction of the bundle to $E^1$ is $\mO_{E^{1}}(p^1)$, which has a basepoint at $p^1$. If $j \neq j_{-}(1)$, then by (4) and (5), $L(j)_{\vec{d}_j}^{\vee}\otimes \omega_X|_{E^1}$ has non-positive degree and is not isomorphic to $\mO_{E^1}$, so it has no nonzero sections. Now suppose the section vanishes at $p^m$ for all $m < i < g$. We show that it also vanishes at $p^i$. If $j = j_-(i)$, then the restriction of the section to $E^i$ is a section of $\mO_{E^i}(p^i) \subset \mO_{E^i}(p^{i-1}+p^i)$ (since it vanishes at $p^{i-1}$), which has a basepoint at $p^i$. If $j \neq j_{-}(i)$, then vanishing at $p^{i-1}$ forces the section to vanish identically on $E^i$. This completes the induction. Since $L(j)_{\vec{d}_j}^{\vee}\otimes \omega_X|_{E^i}$ has a basepoint at all nodes on $E^i$ if and only if $j = j_{-}(i)$, any element of $H^0(L(j)_{\vec{d}_j}^{\vee}\otimes \omega_X)$ vanishes identically on any $E^i$ with $j \neq j_{-}(i)$. This fact along with (1) proves the claim.
        \end{enumerate}
\end{proof}

\begin{prop}\label{prop:injective_on_X}
Let $\vec{d}_{j}$ be the degree distribution associated to $T = \sigma^k_{m_g}\dots \sigma^k_{m_1}$ defined in \eqref{eq:distribution} and let 
\begin{equation*}
    S_j^X := H^0(L(j)_{\vec{d}_j})/H^0(L(j)_{\vec{d}_j}(-p^g)).
\end{equation*}
Fix a $j\in [k]$ and suppose $i_1,\dots,i_{n_j}$ are precisely the indices $i$ such that $j_{-}(i) = j$. Then for each $1 \leq m \leq n_j$, there exists a section $\lambda^m_j$ in the image of
    \begin{equation*}
    S_j^X \otimes H^0(L(j)_{\vec{d}_{j}}^{\vee} \otimes \omega_X) \to H^0(\omega_X)
    \end{equation*}
    such that $\lambda^m_j$ is supported on $E^{i_m}$ and vanishes identically on every other component.
    Thus, the following map is injective:
\begin{equation*}\Phi^{\vee}_{L(j)_{\vec{d}_j}}:\bigoplus_{j=0}^{k-1} S_j^X \otimes H^0(L(j)_{\vec{d}_j}^{\vee}\otimes \omega_X) \to H^0(\omega_X).\end{equation*}
\end{prop}

\begin{proof}
   We first construct a nonzero element of $S_{j}^X$. By \cref{prop:distribution_properties}, the restriction of $L(j)_{\vec{d}_j}$ to each component is either $\mO_{E^i}$, $\mO_{E^i}(p^{i-1}+p^i)$, or a degree $1$ bundle with no basepoint at $p^{i-1}$ or $p^i$. Thus, let $s_j \in H^0(L(j)_{\vec{d}_j})$ be a section such that its restriction to each component does not vanish at any node, nor at $p^0$ or $p^g$. In particular, $s_j \notin H^0(L(j)_{\vec{d}_j}(-p^{g}))$, so $s_j$ represents a nonzero class in $S_j^X$. Next, for each $1 \leq m \leq n$, let $\theta^{i_m}_{j} \in H^0(L(j)_{\vec{d}_j}^{\vee} \otimes \omega_{X})$ be a section that restricts on $E^{i_m}$ to a nonzero section vanishing at all nodes on $E^{i_m}$ and is identically zero on every other component; this is possible since by \cref{prop:distribution_properties},
\begin{equation*}
L(j)_{\vec{d}_j}^{\vee} \otimes \omega_{X}|_{E^{i_m}} \cong \mO_{E^{i_m}}(p^{i_m-1} + p^{i_m}), 
\end{equation*}
when $1 < i_m < g$, and is isomorphic to $\mO_{E^{1}}(p^1)$ or $\mO_{E^{g}}(p^{g-1})$, when $i_m = 1 \text{ or } g$.
Then $\lambda^m_j := s_j \theta^{i_m}_j$
is in the image of 
\[ S_j^X \otimes H^0(L(j)_{\vec{d}_{j}}^{\vee} \otimes \omega_X) \to H^0(\omega_X)\]
and is nonzero precisely on $E^{i_m}$.

 Now, for $j \neq j' \in [k]$, since $\{i : j_{-}(i) = j\} \cap \{i : j_{-}(i) = j'\} = \varnothing$, we see the elements \[\{\lambda^m_{j}\}_{\substack{1\leq m \leq n_j\\ j\in [k]}} \subseteq \im \Phi^{\vee}_{L(j)_{\vec{d}_j}}\] are linearly independent. So, \cref{prop:distribution_properties}(1, 4), along with \cref{lem:h1_equal,lem:sum_h1_equals_inv}, shows $\Phi^{\vee}_{L(j)_{\vec{d}_j}}$ has rank equal to the dimension of the domain, and thus the result follows by rank-nullity.
\end{proof}

\begin{cor}\label{cor:smooth_open}
   Let $\tau \in \wt{\Sigma}_k$ and let $\mL\in W^{\tau}(C,p,q)^{\circ}$. Then for a general degree $k$, genus $g$ cover $C \to \PP^1$ totally ramified at $p$ and $q$ such that $k$ is the smallest positive integer satisfying $\mO_C(k(p-q)) \cong \mO_C$, the map 
\begin{equation*}\Phi^{\vee}_{\mL}:\bigoplus_{j=0}^{k-1} S_j \otimes H^0(\mL^{\vee}(-D_j)\otimes \omega_C) \to H^0(\omega_C)\end{equation*}
 is injective. Thus, $W^{\tau}(C,p,q)$ is smooth along the open locus $W^{\tau}(C,p,q)^{\circ}$.
\end{cor}

\begin{proof}
   For ease of notation, let $(C,p,q)$ denote the generic fiber of our family $\mathcal{X}$. Let $L(j)_{\vec{d}_j}$ denote the limit of $L(D_j)$ on $X$ with the degree distribution $\vec{d}_j$ defined in \eqref{eq:distribution}. Then, we have the following commuting square, where the vertical arrows are induced by the usual (injective) operation of limiting sections to the central fiber:
\[\begin{tikzcd}
	{\bigoplus_{j=0}^{k-1} S_j \otimes H^0(\mL^{\vee}(-D_j)\otimes \omega_C)} & {H^0(\omega_C)} \\
	{ \bigoplus_{j=0}^{k-1} S_j^X \otimes H^0(L(j)_{\vec{d}_j}^{\vee}\otimes \omega_X)} & {H^0(\omega_X)}
	\arrow["{\Phi^{\vee}_{\mL}}", from=1-1, to=1-2]
	\arrow[hook, from=1-1, to=2-1]
	\arrow[hook, from=1-2, to=2-2]
	\arrow["{\Phi^{\vee}_{L(j)_{\vec{d}_j}}}", from=2-1, to=2-2]
\end{tikzcd}\]
    By \cref{prop:injective_on_X}, $\Phi^{\vee}_{L(j)_{\vec{d}_j}}$ is injective, and so $\Phi^{\vee}_{\mL}$ must be injective too.
\end{proof}

 We next show that $W^{\tau}(C,p,q)$ is also smooth along codimension 1 loci. In order to do this, we characterize codimension 1 permutations of a given $\tau$.

\begin{lem}\label{lem:codim1}
     Let $\tau, \tau' \in \wt{\Sigma}_k$. Suppose $\tau' > \tau, \chi_{\tau'} = \chi_{\tau}$, and $\inv_{k}(\tau') = \inv_{k}(\tau)+1$. Then, there exist $j_{-}, j_{+} \in [k]$ and $\Delta > 0 \in \ZZ$ such that 
     \begin{itemize}    
         \item  $\tau'(j) = \tau(j)$ for all $j  \in [k]\setminus\{j_{\pm}\}$,
         \item $\tau'(j_{\pm}) = \tau(j_{\pm}) \pm \Delta$,  
         \item $\Delta \equiv \tau(j_{-}) - \tau(j_{+}) \pmod{k}$.
     \end{itemize}
     If $j_{-} < j_{+}$, then $\tau(j_{-}) < \tau(j_{+})$ and $\Delta < k$. If $j_{-} > j_{+}$, then $\Delta = \tau(j_{-}) - \tau(j_{+})$. 
     
     Further, we have the following equalities:
\begin{equation}\label{eq:constant}
s_{\tau}(\tau(j)+1,j)=s_{\tau'}(\tau(j)+1,j) \text{ for  every $j \in [k]$},
\end{equation} 
\begin{equation}\label{eq:increase}
    s_{\tau'}(\tau(j)+1, j) = \begin{cases}
        s_{\tau'}(\tau(j)+1, j+1)+1 & \text{ if } j_{-} < j_{+} \text{ and } j \in [k] \setminus\{j_{+}\}\\
        s_{\tau'}(\tau(j), j)+1 & \text{ if } j_{-} > j_{+} \text{ and } j \in [k] \setminus\{j_{-}\},
    \end{cases}
\end{equation}
and 
\begin{equation}\label{eq:constant2}
    s_{\tau}(\tau'(j_+)+1, j_{+}) =  s_{\tau'}(\tau'(j_+)+1, j_{+}).
\end{equation}
 \end{lem}

\begin{proof}
  Let $T' = \sigma^k_{m_\ell}\dots \sigma^{k}_{m_1}$ be a reduced word for $\tau' + \chi_{\tau'}$. Then, there exists $1 \leq n \leq \ell$ such that $T \coloneqq \sigma^k_{m_{\ell}} \dots \sigma^k_{m_{n+1}}\sigma^k_{m_{n-1}}\dots \sigma^k_{m_1}$ is a reduced word for $\tau + \chi_{\tau}$ \cite[Cor.\ 2.2.3]{bjorner2006combinatorics}. Let $j_{-} = j_{-}(n)$ and $j_{+} = j_{+}(n)$. Since $T^{\leq n}(j) = T^{\leq n-1}(j)$ for all $j \in [k]\setminus\{j_{-}, j_{+}\}$, it follows that $\tau'(j) = \tau(j)$ for all $j  \in [k] \setminus\{j_{-},j_{+}\}$. Further, because $\sigma_{m_n}^k$ exchanges the residue classes at $j_{-}$ and $j_{+}$ and $\chi_{\tau} = \chi_{\tau'}$ (equivalent to preserving the sum of $\tau$ over the window $0,\dots,k-1$), the existence of the claimed $\Delta \in \ZZ$ follows. We will show that in fact $\Delta > 0$ below.

    Now, suppose $j_{-} < j_{+}$. Then it must be that $\tau(j_{-}) < \tau(j_{+})$, because by \cref{prop:increasing_decreasing}(1,2), $T^{\leq n-1}(j_{+})-T^{\leq n-1}(j_{-}) \geq k-1 > 0$, and \cref{prop:increasing_decreasing}(3) implies that  $\sigma_{m_{n'}}^k$ for $n' > n$ can only increase that difference. We next see that the addition of $\sigma_n^k$ to $T$ simply exchanges all the inversions $j_{-}$ and $j_{+}$ are involved in after the $n$-th step, and so in particular, the number of inversions between the indices $j_{-}$ and $j_{+}$ increases exactly by 1 when we go from $\tau$ to $\tau'$. Thus, $0 < 2\Delta < 2k$ or $0 < \Delta < k$. 

    Lastly, by \cref{prop:increasing_decreasing}, if $j_{+} < j_{-}$, then
     \begin{equation*}
         T^{\leq n}(j_{+}) = T^{\leq n}(j_{-}) + 1 = T^{\leq n-1}(j_{-})
     \end{equation*}
     and similarly, $T^{\leq n}(j_{-}) = T^{\leq n-1}(j_{+})$. In other words, the effect of $\sigma_{m_n}^{k}$ is to swap the $j_{-}$ and $j_{+}$-th entries of $\tau$, corresponding to $\Delta = \tau(j_{-}) - \tau(j_{+})$. Moreover, by \cref{prop:increasing_decreasing}(3), we must have $\tau'(j_{+}) > \tau'(j_{-})$, which is equivalent to $\Delta > 0$.

We now prove the claim \eqref{eq:constant}. For $j_1,j_2\in [k]$, define
\begin{equation*}
    i_{k}(\alpha, j_1, j_2) \coloneqq  \max
\Big\{\Big\lfloor{\frac{\alpha(j_1)-\alpha(j_2)}{k}\Big\rfloor},0\Big\} + \mathbbm{1}_{j_1 < j_2}\cdot \mathbbm{1}_{\alpha(j_1) > \alpha(j_2)},
\end{equation*}
which counts the number of $k$-inversion pairs congruent to $(j_1,j_2)$ modulo $k$.
Fix $j \in [k]\setminus\{j_{-},j_{+}\}$. Then, 
\begin{equation*}
    s_{\tau}(\tau(j)+1,j) =
1+\sum_{j'\in [k]\setminus\{j\}} i_k(\tau,j,j').
\end{equation*}
and similarly
\begin{equation*}
    s_{\tau'}(\tau(j)+1,j)=s_{\tau'}(\tau'(j)+1,j)
=
1+\sum_{j'\in [k]\setminus\{j\}} i_k(\tau',j,j').
\end{equation*}

Now for every $j'\notin\{j_{-},j_{+}\}$, we have $\tau'(j')=\tau(j')$ and so $
i_k(\tau,j,j')=i_k(\tau',j,j')$. Therefore, it suffices to see that 
\begin{equation*}
    i_k(\tau,j,j_{-})+i_k(\tau,j,j_{+}) = i_k(\tau',j,j_{-})+i_k(\tau',j,j_{+}),
\end{equation*}
but this follows from the fact that the addition of $\sigma_{m_n}^{k}$ to $T$ exchanges all the inversions $j_{-}$ and $j_{+}$ are involved in after the $n$-th step, and so keeping the combined number of $k$-inversions preserved.

To prove \eqref{eq:constant} for $j \in \{j_{-},j_{+}\}$, it suffices to show that 
\begin{align*}
1+\max\Big\{\Big\lfloor{\frac{\tau(j_{\pm})-\tau(j_{\mp})}{k}\Big\rfloor}+\mathbbm{1}_{j_{\mp} > j_{\pm}},0\Big\} = &  \max  \Big\{\Big\lfloor{\frac{\tau(j_{\pm})-\tau'(j_{\pm})}{k}\Big\rfloor}+1,0\Big\}\\
&+\max\Big\{\Big\lfloor{\frac{\tau(j_{\pm})-\tau'(j_{\mp})}{k}\Big\rfloor}+\mathbbm{1}_{j_{\mp} > j_{\pm}},0\Big\},
\end{align*}
which can be verified directly on a case-by-case basis.

For \eqref{eq:increase}, first suppose that $j_{-} < j_{+}$. Then, the claimed equality is equivalent to $\tau'(j) \leq \tau(j)$ for $j \in [k]\setminus \{j_+\}$, which is true since $\tau'(j_{-}) = \tau(j_{-})- \Delta < \tau(j_-)$ and $\tau'(j) = \tau(j)$ for all $j \in [k]\setminus\{j_\pm\}$. Else if $j_{+} < j_{-}$, then the claimed equality is equivalent to $\tau'^{-1}(\tau(j)) \geq j$, which is true since $\tau'^{-1}(\tau(j_{+})) = j_{-} \geq j_{+}$.

Finally, for \eqref{eq:constant2}, it suffices to check
\begin{align*}
1 + \max\Big\{\Big\lfloor{\frac{\tau'(j_{+})-\tau'(j_{-})}{k}\Big\rfloor}+\mathbbm{1}_{j_{-} > j_{+}},0\Big\}  = & \max\Big\{\Big\lfloor{\frac{\tau'(j_{+})-\tau(j_{+})}{k}\Big\rfloor}+1,0\Big\} \\
&+ \max\Big\{\Big\lfloor{\frac{\tau'(j_{+})-\tau(j_{-})}{k}\Big\rfloor}+\mathbbm{1}_{j_{-} > j_{+}},0\Big\}. 
\end{align*}
which is again easily verified.
\end{proof}

\begin{thm}\label{thm:smooth}
Let $\tau \in \wt{\Sigma}_k$ and let $C$ be a general degree $k$, genus $g$ cover of $\PP^1$ totally ramified at two points $p$ and $q$ such that $k$ is the smallest positive integer satisfying $\mO_C(k(p-q)) \cong \mO_C$. Then $W^{\tau}(C,p,q)$ is smooth along transmission loci of codimension $1$ inside $W^{\tau}(C,p,q)$.
\end{thm}
\begin{proof}
    Let $\tau'\in \wt{\Sigma}_k$ be such that $\tau' > \tau$ with $\inv_{k}(\tau') = \inv_{k}(\tau)+1$ and $\chi_{\tau'} = \chi_{\tau}$. We need to show that $W^{\tau}(C,p,q)$ is smooth along $W^{\tau'}(C,p,q)^{\circ}$. Let 
    \begin{equation*}
        Z\coloneqq \bigcup_{\substack{\tau'' > \tau\\
        \tau'' \neq \tau'\\ \chi_{\tau''} = \chi_{\tau}}} W^{\tau''}(C,p,q).
    \end{equation*}
    It suffices to show that $Y \coloneqq W^{\tau'}(C,p,q)^{\circ}$ is a Cartier divisor in $U \coloneqq {W^{\tau}}(C,p,q)\setminus Z$. Indeed, if $W^{\tau}(C,p,q)$ was singular at some $L \in Y$, then $\dim T_{L} W^{\tau}(C,p,q) > \dim  W^{\tau}(C,p,q) = \dim Y + 1$. Assuming $Y$ is Cartier in $U$, we would have \begin{equation*}
        \dim T_{L} Y \geq \dim T_{L} W^{\tau}(C,p,q)-1> \dim Y,
    \end{equation*}
contradicting that $Y$ is smooth, as we showed in \cref{cor:smooth_open}. 

To show $Y \subset U$ is Cartier, we will construct a map of equal rank vector bundles that drops rank exactly over $Y$, thus letting us cut out $Y$ by the vanishing of a single section of a line bundle. In order to assert that various pushforwards are vector bundles, we will appeal to \cref{lem:codim1} and Grauert's Theorem on several occasions. The central idea of the following argument can also be gleaned from \cref{ex:codim1}. 

From \cref{lem:codim1}, we know there exist indices $j_{-}$ and $j_{+}$ such that the windows for $\tau$ and $\tau'$ only differ at $j_{\pm}$.  Let $\mathcal{P}$ be a Poincar\'e bundle on $C\times U$ and let $\pi: C\times U \to U$ be the projection. Denote by $D_j$ the pullback of $\tau(j)\cdot p - j\cdot q$ to $C\times U$; let $t$ be the pullback of $q$ if $j_{-} < j_{+}$ and let it be the pullback of $p$ otherwise. Let $j_{\epsilon} = j_{+}$ if $j_{-} < j_{+}$ and let it be $j_{-}$ otherwise. 
Then, by \cref{lem:codim1}(\eqref{eq:constant},\eqref{eq:increase}), we see for $j \in [k]\setminus\{j_{\epsilon}\}$ that $\pi_*\mP(D_j)$ and $\pi_*\mP(D_j-t)$ are both vector bundles, and there is a natural inclusion $\pi_*\mP(D_j-t) \hookrightarrow \pi_*\mP(D_j)$ on each fiber, so $\mathcal{S}_j \coloneqq \pi_{*}\mP(D_j)/\pi_*\mP(D_j-t)$ is a line bundle over $U$. To define $\mathcal{S}_{j_{\epsilon}}$, let 
\[I_{\epsilon} \coloneqq \{j\in [k]\setminus\{j_{\epsilon}\}: (\tau(j), -j) \prec (\tau(j_{\epsilon}),-j_{\epsilon})\},\]
and let $q_{j,j_{\epsilon}}, r_{j,j_{\epsilon}}, \text{ and } s_{j,j_{\epsilon}}$ be nonnegative integers as in \cref{dfn:ess_twists}. Let $\eta$ be the constant section of $\mO_{C}(p)$ and $\beta$ the constant section of $\mO_{C}(q)$. Then, we have a map of vector bundles 
\begin{equation*}
  \zeta:  \bigoplus_{j\in I_{\epsilon}} \mathcal{S}_{j}\otimes H^0(\PP^1, \mO_{\PP^1}(q_{j,j_{\epsilon}})) \xr{\oplus_{j\in I_{\epsilon}}\cdot \eta^{r_{j,j_{\epsilon}}}\cdot \beta^{s_{j,j_{\epsilon}}}} \pi_{*} \mP(D_{j_{\epsilon}}),
\end{equation*}
which is injective by comparing orders of vanishing modulo $k$ at $t$ since both $\tau(j) \not \equiv \tau(j')$ and $j \not \equiv j'$ modulo $k$ for $j\neq j' \in I_{\epsilon}$. Thus $\zeta$ has rank 
\begin{equation*}
   \sum_{j \in I_{\epsilon}}(q_{j,j_{\epsilon}}+1) = s_{\tau}(\tau(j_{\epsilon})+1,j_{\epsilon})-1= \rk \pi_*\mP(D_{j_{\epsilon}}) - 1,
\end{equation*}
and so $\mathcal{S}_{j_{\epsilon}}\coloneqq \coker \zeta$ is a line bundle over $U$. Away from $Y$, $\mathcal{S}_{j_{\epsilon}} = \pi_{*}\mP(D_{j_{\epsilon}})/\pi_{*}\mP(D_{j_{\epsilon}}-t)$, so the section corresponding to $\mathcal{S}_{j_{\epsilon}}$ does not vanish at $t$ away from $Y$, while it does vanish at $t$ over $Y$ since $\pi_{*}\mP(D_{j_{\epsilon}})|_{Y}\cong\pi_{*}\mP(D_{j_{\epsilon}}-t)|_Y$.

Let $D_{+}'$ be the pullback of $\tau'(j_{+})\cdot p - j_{+} \cdot q$ and let
\[I_{+'}  \coloneqq \{j\in [k]: (\tau(j), -j) \prec (\tau'(j_{+}),-j_{+})\}.\]
Notice that $j_{-}, j_{+} \in I_{+'}$. For each $j \in I_{+'}$, let $q_{j,j_{+}'}, r_{j,j_{+}'}, \text{ and } s_{j,j_{+}'}$ be nonnegative integers as in \cref{dfn:ess_twists}. Then, we have a map 
\begin{equation*}
    \psi: \mathcal{V} \coloneqq \bigoplus_{j \in I_{+'}} \mathcal{S}_j \otimes H^0(\PP^1, \mO_{\PP^1}(q_{j,j_+'})) \xr{\oplus_{j\in I_{+'}}\cdot \eta^{r_{j,j_{+}'}}\cdot \beta^{s_{j,j_{+}'
    }}}  \pi_* \mP(D_+'). 
\end{equation*}
between vector bundles of equal rank $r = s_{\tau}(\tau'(j_+)+1,j_{+})$ by \cref{lem:codim1}\eqref{eq:constant2}. Away from $Y$, $\psi$ is injective by comparing orders of vanishing modulo $k$ at $t$, while that argument fails over $Y$ because the image of the section under $\psi$ coming from $\mathcal{S}_{j_{\epsilon}}$ vanishes at $t$. Since
\[s_{\tau'}(\tau'(j_+)+1, j_+) = s_{\tau'}(\tau'(j_+), j_+) + 1 = s_{\tau'}(\tau'(j_+)+1, j_{+} +1)+ 1, \]
we know $\pi_{*}\mP(D'_{+})$ contains a section not vanishing at $t$, but every section in the image of $\psi$ over $Y$ vanishes at $t$. Thus, $\psi$ fails to be surjective exactly over $Y$ and so $Y$ is cut out by the vanishing of $\det \psi$, which is a section of the line bundle
\begin{equation*}
    \bigwedge^{r} \mathcal{V}^{\vee} \otimes  \bigwedge^{r} \pi_* \mP(D'_{+}),
\end{equation*}
thus proving $Y \subset U$ is Cartier.
\end{proof}

We illustrate the proof of \cref{thm:smooth} in the following example.
\begin{example}\label{ex:codim1}
    Consider $\tau = (0,7,11)$ and $\tau' = (0,5,13)$ in $\wt{\Sigma}_3$. Note that $\chi_{\tau'} = \chi_{\tau}$ and $\inv_{k}(\tau') = 7 = \inv_{k}(\tau) + 1$. For this example, $j_{-} = 1$, $j_{+} = 2$, $\Delta = 2$, $t = q$, and $j_\epsilon = 2 $. The fibers of $\mathcal{S}_{0}$ and $ \mathcal{S}_{1}$ correspond to the new section not vanishing at $q$ in the $(0,0)$ and $(7,-1)$ twists. The twist $(11,-2)$ corresponds to $D_{j_{\epsilon}}$. We have the following injective map:
    \begin{equation*}
        \mathcal{S}_0 \otimes H^0(\PP^1, \mO_{\PP^1}(2)) \oplus \mathcal{S}_1 \xr{\eta^2\cdot \beta^1 \oplus \eta^1 \cdot \beta^2} \pi_{*}\mP(D_{j_{\epsilon}}),
    \end{equation*}
    which has rank $4$, and thus the cokernel $\mathcal{S}_2$ has rank $1$, since
    \[s_{\tau}(\tau(2)+1,2) = 5 = s_{\tau'}(\tau(2)+1,2).\]
   Moreover, since
   \[s_{\tau'}(\tau(2)+1,2) = s_{\tau'}(\tau(2)+1,3) \text{ and } s_{\tau}(\tau(2)+1,2) = s_{\tau}(\tau(2)+1,3)+1,\]
   the section corresponding to $\mathcal{S}_2$ vanishes at $q$ exactly over $Y$.
Finally, we have 
    \[s_{\tau}(\tau'(2)+1, 2) = 7 = s_{\tau'}(\tau'(2)+1,2),\]
    so $\pi_{*}\mP(D_{+}')$ has rank $7$. The map 
    \[\psi:(\mathcal{S}_0 \otimes H^0(\PP^1, \mO_{\PP^1}(3))) \oplus (\mathcal{S}_1  \otimes H^0(\PP^1, \mO_{\PP^1}(1)))\oplus \mathcal{S}_{2} \xr{\eta^1\cdot \beta^1 \oplus \eta^0 \cdot \beta^2 \oplus \eta^{2} \cdot \beta^{0}} \pi_{*}\mP(D_{j_{+}}'),\]
drops rank exactly over $Y$, since all sections in the image of $\psi$ over $Y$ vanish at $q$, while $\pi_{*}\mP(D_{j_+}')$ contains a section not vanishing at $q$.
    
\end{example}

\begin{thm}\label{thm:normal}
    Let $\tau \in \wt{\Sigma}_k$ and let $C$ be a general degree $k$, genus $g$ cover of $\PP^1$ totally ramified at two points $p$ and $q$ such that $k$ is the smallest positive integer satisfying $\mO_C(k(p-q)) \cong \mO_C$. Then, $W^{\tau}(C,p,q)$ is normal.
\end{thm}
\begin{proof}
    Since $W^{\tau}(C,p,q)$ is Cohen-Macaulay (\cref{thm:cm}) and is smooth away from all transmission loci of codimension $\geq 2$ (\cref{cor:smooth_open,thm:smooth}), it is normal by Serre's criterion ($R_1 + S_2$). 
\end{proof}

\begin{thm}\label{thm:irreducible}
    Let $\tau \in \wt{\Sigma}_k$ and let $C$ be a general degree $k$, genus $g$ cover of $\PP^1$ totally ramified at two points $p$ and $q$ such that $k$ is the smallest positive integer satisfying $\mO_C(k(p-q)) \cong \mO_C$. If $g > \inv_{k}(\tau)$, then $W^{\tau}(C,p,q)$ is irreducible.
\end{thm}
\begin{proof}
    Since $W^{\tau}(C,p,q)$ is connected (\cref{thm:connected}) and normal (\cref{thm:normal}), it is irreducible.
\end{proof}

\section{Monodromy}\label{sec:monodromy}

 In this section, assume the characteristic of the ground field does not divide $k$. We outline the argument for the following statement.

\begin{thm}\label{thm:monodromy}
  Let $\tau \in \wt{\Sigma}_{k}$. If   $g \geq \inv_{k}(\tau)$, then the universal $\mathcal{W}^{\tau}$ has a \textit{unique} component dominating a component of $\mathcal{H}_{k,g,2}$ that consists of twice-marked curves $(C,p,q)$ such that $k$ is the smallest positive integer satisfying $\mO_C(k(p-q)) \cong \mO_C$.
\end{thm}

We have already shown that if $g > \inv_{k}(\tau)$, then $W^{\tau}(C,p,q)$ is irreducible for a general point $(C,p,q)$ in our component $B$ of $\mathcal{H}_{k,g,2}$. Further, if $g=\inv_{k}(\tau)$ and $k=2$, then there is exactly one reduced word for $\tau$, and so $W^{\tau}(C,p,q)$ consists of one reduced point. Thus, it suffices to see that the statement holds when $g=\inv_{k}(\tau)$ and $k > 2$. We can apply the arguments of \cite[Section 10]{larson_global_2025} without much modification. We outline their proof technique here for the sake of completeness.

\begin{enumerate}
    \item \textbf{Changing the base:}
    Let $\mathcal{C}$ denote the universal curve over our component $B$ of $\mathcal{H}_{k,g,2}$. Then $W^{\tau}(\mathcal{C}/B) \to B$ is \'{e}tale near the generic fiber $\mathfrak{f}^{*}$ of our family $\mathcal{X}$ considered in \cref{sec:degeneration} \cite[Lemma 10.1]{larson_global_2025}. So, to show there exists a unique component of $W^{\tau}(\mathcal{C}/B)$ dominating $B$, it suffices to show the same statement for $W^{\tau}(\mathcal{C}/B)\times_{B} B' \to B'$ for any other irreducible reduced base $B'$ such that the image of $B' \to B$ meets $\mathfrak{f}^*$. We can therefore pass to studying the problem over an irreducible component of $\overline{\mathcal H}^{\mathrm{sm\text{-}ch}}_{k,g,2}$, which denotes the open substack of $\overline{\mathcal H}_{k,g,2}$ parameterizing chains of smooth curves with the two marked points on opposite ends of the chain. 

    \item \textbf{Reduction to the special fiber:}
    Now, by smoothness of $\overline{\mathcal H}^{\mathrm{sm\text{-}ch}}_{k,g,2}$ \cite[Lemma 10.3]{larson_global_2025}, the degeneration $f:X\to P$ and its smoothing $f^{*}:\mathcal{X}^{*} \to \mathcal{P}^{*}$ both correspond to points in the same unique component $B'$ of $\overline{\mathcal H}^{\mathrm{sm\text{-}ch}}_{k,g,2}$. 
Let $\mathcal{C}^{\mathrm{sm\text{-}ch}}$ denote the universal curve over $B'$, whose total space is also smooth \cite[Lemma 10.3]{larson_global_2025} ; because the fiber of $W^{\tau}(\mathcal{C}^{\mathrm{sm\text{-}ch}}/B') \to B'$ over $X$ is a reduced finite set, it is \'{e}tale near $X$. Thus, showing the uniqueness of a dominating component reduces to showing that any two points of the special fiber $W^{\tau}(X)$ lie on the same irreducible component of $W^{\tau}(\mathcal{C}^{\mathrm{sm\text{-}ch}}/B')$.

    \item \textbf{Reduction to pairs indexed by braid moves:}
    The points of $W^{\tau}(X)$ are  limit line bundles $L_T$ indexed by reduced words $T$ for $\tau+d-g$. Since any two such reduced words are connected by a sequence of the elementary braid moves \(F^i\) and \(S^i\) (see \cref{sub:connected}),  it is enough to prove that whenever these moves are defined, the two corresponding points
    \[
        L_T,\; L_{F^iT}
        \qquad\text{or}\qquad
        L_T,\; L_{S^iT}
    \]
    lie in the same irreducible component of $W^{\tau}(\mathcal C^{\mathrm{sm\text{-}ch}})$. 

    \item \textbf{Restriction to smoothing families } Let $T = \sigma^{k}_{m_{i_{\ell}}}\dots \sigma^{k}_{m_{i_{1}}}$.
    To show that $L_T$ and $L_{F^iT}$ lie on the same component, we can  smooth the node \(p^i\) on $X$, obtaining families parameterized by a component $\mathcal H^\circ_{k,2,2}$ (see \cref{fig:smoothing_node}) . Inside the restriction of \(W^{\tau}(\mathcal C^{\mathrm{sm\text{-}ch}})\) to these families, we can define a substack \(Y_F(i,T)\) whose fiber over the deformation shown in \cref{fig:smoothing_node} consists of limit line bundles $L$ whose aspect $L^{E^{j}} = L_{T}^{j}$ for $j\neq i,i+1$, and on $C$ is such that
    \begin{align*}
        L^{C} &\cong  \mO_{C}((m_{i}+i)p^{i-1} + (d-m_{i}-i-1)p^{i+1}+x)\\
        &\cong  \mO_{C}((m_{i+1}+i)p^{i-1} + (d-m_{i+1}-i-1)p^{i+1}+y),
    \end{align*}
    where $x,y$ are points on $C$ satisfying $\mO_{C}(x-y)\cong \mO_{C}((m_{i+1}-m_i)(p^{i-1}-p^{i}))$; this bundle is $\tau$-positive by the same argument as for \cite[Lemma 10.4]{larson_global_2025}.  Since the map $C\times C\to \Pic^{0}(C)$ is degree 2 away from the diagonal, this construction produces two $\tau$-positive limit line bundles over the general curve of these families, and this substack is denoted $Y_{F}(i,T)$. Moreover, these two bundles limit to $L_{T}$ and $L_{F^{i}T}$ over $X$. Similarly, for shuffles \(S^iT\), we can smooth two adjacent nodes and work over families parameterized by a component $\mathcal H^\circ_{k,3,2}$. Then we can define a substack \(Y_S(i,T)\), inside the restriction of \(W^{\tau}(\mathcal C^{\mathrm{sm\text{-}ch}})\) to these families, of two $\tau$-positive limit line bundles limiting to $L_{T}$ and $L_{S^{i}T}$ over $X$ with the same construction of \cite[Lemma 10.5]{larson_global_2025}. The result then follows since $Y_F(i,T)$ and $Y_S(i,T)$ are shown to be irreducible in \cite[\S 10.4]{larson_global_2025}.

\begin{figure}
    \centering
\begin{tikzpicture}[line cap=round,line join=round]

\newcommand{\samearc}[2]{%
  \draw[name path=#2]
    (#1,0) .. controls ({#1+1},-1) and ({#1+2},-1) .. ({#1+3},0);
}


\samearc{-2.4}{arcone}
\samearc{0}{arctwo}
\samearc{6.4}{arcfour}
\samearc{8.8}{arcfive}

\draw[name path=arcthree]
  (2.1,0) .. controls (3.8,-1.45) and (5.6,-1.45) .. (7.3,0);

\path[name intersections={of=arcone and arctwo, by=PmTwo}];
\path[name intersections={of=arctwo and arcthree, by=PmOne}];
\path[name intersections={of=arcthree and arcfour, by=PpOne}];
\path[name intersections={of=arcfour and arcfive, by=PpTwo}];

\fill (PmTwo) circle (1.7pt);
\fill (PmOne) circle (1.7pt);
\fill (PpOne) circle (1.7pt);
\fill (PpTwo) circle (1.7pt);

\fill (-3.55,-0.22) circle (1.2pt);
\fill (-3.25,-0.22) circle (1.2pt);
\fill (-2.95,-0.22) circle (1.2pt);

\fill (12.55,-0.22) circle (1.2pt);
\fill (12.85,-0.22) circle (1.2pt);
\fill (13.15,-0.22) circle (1.2pt);

\node at (-1.00,-1.08) {$E^{i-2}$};
\node at (1.65,-1.08) {$E^{i-1}$};
\node at (4.70,-1.4) {$C$};
\node at (8.15,-1.08) {$E^{i+2}$};
\node at (10.55,-1.08) {$E^{i+3}$};

\node at ($(PmTwo)+(0,-0.42)$) {$p^{i-2}$};

\node[align=center] at ($(PmOne)+(0,-0.42)$) {$p^{i-1}$};

\node[align=center] at ($(PpOne)+(0,-0.42)$) {$p^{i+1}$};

\node at ($(PpTwo)+(0,-0.42)$) {$p^{i+2}$};

\end{tikzpicture}
\caption{Deformation of $X$ obtained by smoothing $p^i$; $C$ is a genus $2$ curve.}
\label{fig:smoothing_node}
\end{figure}
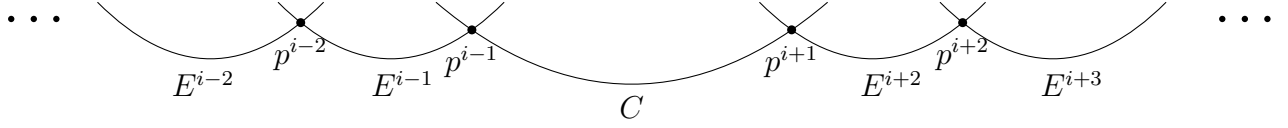
    
\end{enumerate}

\bibliographystyle{abbrv}

\appendix

\section{Degree $k$ covers of $\PP^1$ totally ramified at two points}
\label[app]{app:irreducibility}

In this appendix we show that the locus
\[
\Mexact
:=
\bigl\{(C,p,q) \in \mathcal{M}_{g,2} : \ord\bigl(\Oc(p-q)\bigr)=k\bigr\}
\]
is irreducible in characteristic 0 and for characteristics greater than $k$.  To show this, we work with the larger space $\mathcal{H}_{k,g,2}$, which has a stratification
\[\mathcal{H}_{k,g,2} = \bigcup_{d\mid k} \mathcal{M}_{g,2}(=d),\] and is typically not irreducible.
Nevertheless, in characteristic 0, we will show that assuming $g\geq 1$, the number of irreducible components of $\mathcal{H}_{k,g,2}$ is $\sigma_0(k)-1$, where $\sigma_{0}(k)$ is the number of positive divisors of $k$, and in particular that $\Mexact$ is irreducible in characteristic 0. The extension to characteristics greater than $k$ then follows by the standard argument of Fulton \cite[Theorem 3.3]{fulton1969hurwitz}, which is also reproduced in \cite[Prop.\ 2.2]{coskun2024stability}.

The relation $kp \sim kq$ exhibits $C$ as a degree-$k$ cover of $\mathbb P^1$ that is totally
ramified over $0$ and $\infty$.  A general such cover is simply ramified over $2g$ additional
branch points $x_1,\dots,x_{2g}$.  The corresponding monodromy data consist of two $k$-cycles
and $2g$ transpositions whose product is trivial.  Thus $\mathcal{H}_{k,g,2}$ contains a
dense open subset that is a finite cover of the (unordered) configuration space of $2g$ points in $\mathbb P^1\setminus\{0,\infty\}$, and its
irreducible components correspond to the orbits of the induced braid-group action on the
monodromy data. We now introduce the explicit action of braid moves on the monodromy data.

\subsection{The annular braid group and monodromy data}

Let $\mathcal H_{k,g,2}^{\circ} \subset \mathcal H_{k,g,2}$ denote the open substack of degree-$k$ covers $f \colon C \to \mathbb P^1$ that are totally ramified over $0$ and $\infty$ and all other branch points are simple, so there are $2g$ of them which we call $x_1,\dots,x_{2g}$. Thus there is a  map
\[
 \phi: \mathcal H_{k,g,2}^{\circ}
\longrightarrow
\operatorname{Conf}_{2g}\!\bigl(\mathbb P^1 \setminus \{0,\infty\}\bigr),
\qquad
[f] \longmapsto \{x_1,\dots,x_{2g}\}.
\]
 Since $\mathbb P^1 \setminus \{0,\infty\} \cong \mathbb C^\ast$, the target is the configuration space of $2g$ unordered points on the annulus $\mathbb C^\ast$.
Its fundamental group is the \defi{annular braid group} $CB_{2g}
:=
\pi_1\!\left(\operatorname{Conf}_{2g}(\mathbb C^\ast), *\right)$ \cite[\S 2]{bellingeri2016braid}. A presentation of $CB_{2g}$ due to Chow is  the following (see
\cite{chow1948braid} and \cite[\S 1]{gadbled2017categorical}).  It has generators
\[
\sigma_1,\dots,\sigma_{2g-1},\varepsilon
\]
and relations
\begin{align*}
\sigma_i \sigma_{i+1} \sigma_i &= \sigma_{i+1} \sigma_i \sigma_{i+1} \qquad (1\le i\le 2g-2),\\
\sigma_i \sigma_j &= \sigma_j \sigma_i \qquad (|i-j|>1),\\
\varepsilon \sigma_1 \varepsilon \sigma_1 &= \sigma_1 \varepsilon \sigma_1 \varepsilon,\\
\varepsilon \sigma_j &= \sigma_j \varepsilon \qquad (2\le j\le 2g-1).
\end{align*}
Geometrically, $\sigma_i$ exchanges the adjacent branch points $x_i$ and $x_{i+1}$, while
$\varepsilon$ winds $x_1$ once positively around the puncture $0$. This presentation is closely related to the usual extended affine type $A$ presentation, in which the generators are
\[
\sigma_1,\dots,\sigma_{2g},\rho
\]
with indices taken modulo $2g$, where $\rho$ cyclically permutes the points and thus satisfies (see \cite[\S 1]{gadbled2017categorical})
\[
\rho \sigma_i\rho^{-1}=\sigma_{i+1} \text{ and } \rho=\varepsilon \sigma_1\sigma_2\cdots \sigma_{2g-1}.
\]

The monodromy action of $CB_{2g}$ on the fiber of the map $\phi$ induces an action on the monodromy data of covers in $\mathcal{H}_{k,g,2}^{\circ}$. To describe this action, label the sheets of
a cover $f:C\to \PP^1$ in $\mathcal H_{k,g,2}^{\circ} $ by $\mathbb Z/k\mathbb Z$, so that the monodromy around $0$
is the $k$-cycle
\[
c=(0,1,\dots,k-1).
\]
Then the monodromy of the cover is recorded by a tuple of transpositions
\[
(\tau_1,\dots,\tau_{2g}),
\]
one for each simple branch point, together with the relation that the monodromy around $\infty$ is given by
\[
c_\infty=(c\tau_1\cdots\tau_{2g})^{-1}.
\]
The condition that the cover is totally ramified over $\infty$ is equivalent to
$c\tau_1\cdots\tau_{2g}$ being a $k$-cycle.   The generators $\sigma_i$ and $\varepsilon$ act on a monodromy tuple as follows: 
\begin{equation}\label{eq:flip}
    \sigma_i \cdot (c, \tau_1,\dots,\tau_r) = (c, \tau_1, \dots, \tau_{i+1}, \tau^{-1}_{i+1}\tau_i \tau^{i+1}, \dots, \tau_r),
\end{equation}
and 
\begin{equation}\label{eq:loop}
    \varepsilon \cdot (c, \tau_1,\dots,\tau_r) = ((c\tau_1)^{-1}c (c\tau_1), (c\tau_1)^{-1}\tau_1 (c\tau_1), \tau_2, \dots, \tau_r).  
\end{equation}

For $1\le i\le r$, define
\[
\varepsilon_1:=\varepsilon,
\qquad
\varepsilon_i
:=
\sigma_{i-1}\cdots \sigma_1\,\varepsilon\,\sigma_1\cdots \sigma_{i-1}.
\]
The element $\varepsilon_i \in CB_r$ loops $x_i$ once
around $0$. 
For $1\le i <r$, set
\[
\omega_i:=\varepsilon_i\varepsilon_{i-1}\cdots\varepsilon_1.
\]
Then $\omega_i$ cycles the points $x_1,\dots,x_i$
once around $0$ while preserving their cyclic order. Explicitly, 
\begin{equation}\label{eq:block_loop}
   \omega_i \cdot (c, \tau_1,\dots,\tau_r) = (u_i^{-1} c u_i, u_i^{-1} \tau_1 u_i, \dots, u_i^{-1} \tau_i u_i, \tau_{i+1}, \dots, \tau_{r}),
\end{equation}
where $u_i \coloneqq c\tau_1\dots \tau_i$.

Further, the monodromy representation is defined only up to relabeling of the $k$ sheets, that is, up to
simultaneous conjugation by an element of $S_k$.  Once we impose the normalization that the
monodromy around $0$ is exactly
\[
c=(0,1,\dots,k-1),
\]
the remaining allowed relabelings are the elements of the stabilizer of $c$ under the
conjugation action of $S_k$ on itself.  This stabilizer is its
centralizer $\langle c\rangle$.

Thus, we obtain the following equivalences of tuples $(\tau_1,\dots,\tau_{2g})$ that all lie in the same orbit under the action of $CB_r$ on the fiber of $\phi$.

\begin{enumerate}[label=\textup{(E\arabic*)},leftmargin=3.5em]
  \item \label{eq:E1}
  Simultaneous conjugation by $c$:
  \begin{equation*}
  (\tau_1,\dots,\tau_{2g})
  \sim
\bigl(c^{-1}\tau_1c,\dots,c^{-1}\tau_{2g}c\bigr).
  \label{eq:simultaneous-conjugation}
  \end{equation*}

  \item \label{eq:E2}
  The move swapping $x_i$ and $x_{i+1}$ for $1\le i<r$ as in \eqref{eq:flip}:
  \begin{equation*}
  (\tau_1,\dots,\tau_{i-1},\tau_i,\tau_{i+1},\tau_{i+2},\dots,\tau_r)
  \sim
  (\tau_1,\dots,\tau_{i-1},\tau_{i+1},\tau_{i+1}^{-1}\tau_i\tau_{i+1},\tau_{i+2},\dots,\tau_r).
  \label{eq:hurwitz-move}
  \end{equation*}
  We will say that this move \emph{moves $\tau_{i+1}$ backwards}; the inverse move \emph{moves
  $\tau_i$ forwards}.

  \item \label{eq:E3}
  For $1\le i <r$, cycling $x_1,\dots,x_i$ around $0$ as in \eqref{eq:block_loop} and then conjugating the whole tuple by $u_{i}$ gives us the equivalence:
  \begin{equation*}
  (\tau_1,\dots,\tau_i,\tau_{i+1},\dots,\tau_{2g})
  \sim
  \bigl(\tau_1,\dots,\tau_i,
  u_i\tau_{i+1}u_i^{-1},
  \dots,
u_i\tau_{2g} u_i^{-1}\bigr).
  \label{eq:cycling-move}
  \end{equation*}
\end{enumerate}

Given any $d \mid k$ and $d < k$, by Riemann's Existence theorem, there exists a degree $k$ cover $C_d$ of $\PP^1$ with the monodromy data $\tau_i = (0,d)$ for $1 \leq i \leq 2g$ along with $c,c_{\infty}$; further the degree $k$ map $C_d \to \PP^1$ factors as a degree $k/d$ map $C_d \to \PP^1$ followed by the composition $\PP^1 \to \PP^1: z \to z^{d}$, and so $C_d$ corresponds to a point of $M_{g,2}(=d)$. We first show such monodromy data are in distinct orbits under the action of $CB_{r}$, thus showing there are at least $\sigma_0(k)-1$ orbits, and therefore $\mathcal{H}_{k,g,2}$ has at least one component for each divisor $d \mid k$ for $d > 1$.

\begin{lem}
\label{lem:subgroup-invariant}
For a tuple $\underline{\tau}=(\tau_1,\dots,\tau_r)$, set
\[
G(\underline{\tau}) := \langle c,\tau_1,\dots,\tau_r\rangle \le S_k.
\]
Then:
\begin{enumerate}[label=\textup{(\alph*)},leftmargin=3em]
  \item The subgroup $G(\underline{\tau})$ is invariant under the action of $CB_r$.
  \item If $d' < d <k$ are positive divisors of $k$, then the constant tuples
  \[
  \bigl((0,d),\dots,(0,d)\bigr)
  \qquad\text{and}\qquad
  \bigl((0,d'),\dots,(0,d')\bigr)
  \]
  generate distinct subgroups of $S_k$. Therefore, the two constant tuples are in distinct orbits of the action of $CB_r$.
\end{enumerate}
\end{lem}

\begin{proof}
For \textup{(a)}, $\sigma_i$ replaces the pair
\[
(\tau_i,\tau_{i+1})
\]
by
\[
(\tau_{i+1},\tau_{i+1}^{-1}\tau_i\tau_{i+1}),
\]
so $G(\sigma_i\cdot \underline{\tau}) = G( \underline{\tau})$. Next, applying $\varepsilon$ replaces $(c,\tau_1)$ with $((c\tau_1)^{-1} c (c \tau_1), (c\tau_1)^{-1} \tau_1 (c \tau_1))$. Since
\[ \langle c, \tau_1 \rangle = \langle (c\tau_1)^{-1} c (c \tau_1), (c\tau_1)^{-1} \tau_1 (c \tau_1) \rangle,
\]
we also conclude that $G(\varepsilon \cdot\underline{\tau}) = G( \underline{\tau})$.
Thus, $G(\underline{\tau})$ is invariant under every annular braid generator.

For \textup{(b)}, let $G_d:=\langle c,(0,d)\rangle$.
For any divisor $e\mid k$, consider the partition of $\mathbb Z/k\mathbb Z$ into residue classes
modulo $e$:
\[\ZZ/k\ZZ = \bigsqcup_{a = 0}^{e-1} \{a, a+e\dots, a+(k/e-1)\cdot e\}.\]
The cycle $c$ always preserves this partition modulo $e$ since it simply sends $a \mapsto a+1$.  The transposition $(0,d)$
preserves this partition if and only if $0$ and $d$ lie in the same residue class modulo $e$, that
is, if and only if $e\mid d$.
Therefore, we see that $G_d$ preserves the partition modulo $e$ if and only if $e \mid d$. Now if $d' < d$, then $d\nmid d'$, and so $G_d$ preserves the partition modulo $d$, while
$G_{d'}$ does not. Therefore, $G_{d} \neq G_{d'}$
\end{proof}

In the remainder of this appendix, we prove there are at most $\sigma_0(k)-1$ orbits.

\begin{prop}
\label{prop:at-most-sigma0}
Assume $g > 0$. Consider tuples of transpositions $(\tau_1,\dots,\tau_r)$ such that $c\tau_1\tau_2\cdots\tau_r$
is a $k$-cycle.  Under the equivalences \textup{\ref{eq:E1}}--\textup{\ref{eq:E3}}, there are at
most $\sigma_0(k)-1$ orbits of such tuples.  Consequently, $\mathcal{H}_{k,g,2}$ has exactly $\sigma_0(k)-1$
irreducible components.
\end{prop}

The proof of \cref{prop:at-most-sigma0} proceeds by reducing every tuple to a standard form and then classifying the possible
standard forms when the total product is a $k$-cycle.

\subsection{Reduction to standard form}

\begin{dfn}
\label{def:standard-tuple}
A tuple of $2n+m$ transpositions of the form
\[
(\tau_1,\tau_2,\dots,\tau_{2n},(0,b_1),(0,b_2),\dots,(0,b_m)),\; \tau_{2i-1}=\tau_{2i} \text{ for }{1 \leq i \leq n}, \text{ and } \;
 b_1<\cdots<b_m,
\]
is called \emph{standard}.
\end{dfn}

Because the first $2n$ transpositions in a standard tuple come in equal pairs, they cancel in the
partial products that appear in \textup{\ref{eq:E3}}.  More precisely, for $j\ge 2n$ one has
\[
c\tau_1\tau_2\cdots\tau_{2n}(0,b_1)\cdots(0,b_{j-2n})
=
c(0,b_1)\cdots(0,b_{j-2n}).
\]
Thus the paired transpositions play no role in the cycling move once one passes the first $2n$
entries.

The first step is the following two-transposition reduction.

\begin{lem}
\label{lem:pair-reduction}
Let $u,v,b$ be integers with $0<u\le b<v\le k$, and set $d=\gcd(b,k,v-u)$.
Then the tuple $\bigl((0,b),(u,v)\bigr)$ is equivalent under
\textup{\ref{eq:E1}}--\textup{\ref{eq:E3}} to the standard tuple
\[
\bigl((0,d),(0,d)\bigr).
\]
\end{lem}

\begin{proof}
We argue by induction on $b$.  Observe that
\[
c(0,b)=(1,2,\dots,b)(0,b+1,\dots,k-1).
\]
Applying the move \textup{\ref{eq:E3}} with $j=1$ conjugates $(u,v)$ by $c(0,b)$.  This clearly
preserves $d$.  Moreover, under repeated conjugation by $c(0,b)$, the element $u$ moves in a cycle
of length $b$, while $v$ moves in a cycle of length $k-b$.  Hence, after repeatedly applying
\textup{\ref{eq:E3}}, the tuple is equivalent to one of the form
\[
\bigl((0,b),(0,w)\bigr)
\]
with $\gcd(b,k,w)=d$ and $w\le \gcd(b,k-b)=\gcd(b,k)$.

If $w=b$, then necessarily $d=b$, and there is nothing to prove.  Otherwise $w<b$.  Now move
$(0,w)$ backwards using \textup{\ref{eq:E2}}; this transforms the pair into
\[
\bigl((0,w),(w,b)\bigr).
\]
Since
\[
\gcd(w,k,b-w)=\gcd(b,k,w)=d,
\]
the induction hypothesis applies and yields the claim.
\end{proof}

\begin{lem}
\label{lem:standardization}
Every tuple of $r$ transpositions is equivalent under \textup{\ref{eq:E1}}--\textup{\ref{eq:E3}}
to a standard tuple.
\end{lem}

\begin{proof}
We argue by induction on $r$.  By the induction hypothesis, we may assume that the tuple has the
form
\[
(\tau_1,\tau_2,\dots,\tau_{2n},(0,b_1),(0,b_2),\dots,(0,b_m),(u,v)),
\qquad \tau_{2i-1}=\tau_{2i},
\]
for some $0\le u<v\le k-1$.

If $m=0$, then applying \textup{\ref{eq:E3}} with $j=2n$ simply conjugates $(u,v)$ by $c$.
Repeating this as necessary, we may assume that the final transposition is of the form $(0,b)$, and
hence the tuple is already standard.  We therefore assume from now on that $m>0$.

\smallskip
\noindent\textbf{Case 1:} Neither $u$ nor $v$ lies in the interval $[1,b_1]$.

In this case,
\[
c\tau_1\tau_2\cdots\tau_{2n}(0,b_1)=c(0,b_1)=(1,2,\dots,b_1)(0,b_1+1,\dots,k-1).
\]
Conjugation by $c(0,b_1)$ acts on any permutation supported on
$\{0,b_1+1,\dots,k-1\}$ in the same way as conjugation by the $(k-b_1)$-cycle
\[
c'=(0,b_1+1,\dots,k-1).
\]
Each of the transpositions
\[
(0,b_2),\dots,(0,b_m),(u,v)
\]
is supported on this set.  By induction on $r$, using the moves
\textup{\ref{eq:E1}}--\textup{\ref{eq:E3}} with $c$ replaced by $c'$, we may standardize the tail.
Thus our tuple is equivalent to one of the form
\[
(\tau_1,\dots,\tau_{2n},(0,b_1),\tau_1',\tau_2',\dots,\tau_{2n'}',(0,b_1'),\dots,(0,b_{m'}')),
\]
where $\tau_{2i-1}'=\tau_{2i}'$ and $b_1<b_1'<\cdots<b_{m'}'$.  Finally, repeatedly move $(0,b_1)$
forwards using \textup{\ref{eq:E2}}.  Because the transpositions $\tau_{2i-1}'$ and $\tau_{2i}'$
come in equal pairs, this preserves the standard form of the tail, and the result is standard.

\smallskip
\noindent\textbf{Case 2:} There is an interval $[b_i+1,b_{i+1}]$ or $[b_m+1,0]$ containing neither
$u$ nor $v$.

Repeatedly move $(0,b_m)$ backwards using \textup{\ref{eq:E2}}.  This yields the equivalent tuple
\[
(\tau_1,\tau_2,\dots,\tau_{2n},(b_m,0),(b_m,b_1),\dots,(b_m,b_{m-1}),(u,v)).
\]
Now apply \textup{\ref{eq:E3}} with $j=2n$ a total of $b_m$ times.  The result is
\[
(\tau_1,\tau_2,\dots,\tau_{2n},(0,-b_m),(0,b_1-b_m),\dots,(0,b_{m-1}-b_m),(u-b_m,v-b_m)).
\]
Thus the effect is to rotate the marked points $b_1,\dots,b_m,u,v$ around the cyclic order on
$\mathbb Z/k\mathbb Z$.  Repeating this rotation eventually reduces the situation to Case~1,
because one can arrange that neither $u$ nor $v$ lies in the interval $[1,b_1]$.

\smallskip
\noindent\textbf{Case 3:} Every interval contains either $u$ or $v$.

This forces $m=1$.  Since $u<v$, we have $u\in[1,b_1]$ and $v\in[b_1+1,k]$, so
Lemma~\ref{lem:pair-reduction} applies and completes the argument.
\end{proof}

\subsection{Classification of standard tuples}

We now show that, under the additional hypothesis that the total product is a $k$-cycle, the
standard form is determined by a divisor of $k$.

\begin{lem}
\label{lem:euclidean-algorithm}
A tuple of transpositions
\[
\bigl((a_1,b_1),(a_1,b_1),(a_2,b_2),(a_2,b_2),\dots,(a_n,b_n),(a_n,b_n)\bigr)
\]
is equivalent under \textup{\ref{eq:E1}}--\textup{\ref{eq:E3}} to the tuple consisting entirely of
transpositions $(0,d)$, where
\[
d=\gcd(k,b_1-a_1,b_2-a_2,\dots,b_n-a_n).
\]
\end{lem}

\begin{proof}
We argue by induction on $n$.

If $n=1$, then repeated application of \textup{\ref{eq:E1}} conjugates $(a_1,b_1)$ by $c$ a total
of $a_1$ times, so the pair is equivalent to
\[
\bigl((0,b_1-a_1),(0,b_1-a_1)\bigr).
\]
Applying Lemma~\ref{lem:pair-reduction} with $b=b_1-a_1$, $u=b_1-a_1$, and $v=k$ shows that this is
further equivalent to the constant pair with value $\gcd(k,b_1-a_1)$.

Assume now that $n\ge 2$.  By induction, we may suppose that
\[
\tau_1=\tau_2=\cdots=\tau_{2n-2}=(0,a),
\qquad
a=\gcd(k,b_1-a_1,b_2-a_2,\dots,b_{n-1}-a_{n-1}).
\]
Under repeated application of \textup{\ref{eq:E3}} with $j=2n-2$, the tuple is equivalent to
\[
\bigl((0,a),(0,a),\dots,(0,a),(0,r),(0,r)\bigr),
\qquad r=b_n-a_n.
\]
It therefore suffices to show that one may run the Euclidean algorithm on the pair $(a,r)$.

There are two elementary operations to justify.

\smallskip
\noindent\emph{Swapping $a$ and $r$.}
Repeatedly applying \textup{\ref{eq:E2}} moves the block $(0,r),(0,r)$ across a copy of $(0,a)$:
\[
(0,a),(0,r),(0,r)
\sim
(0,r),(a,r),(0,r)
\sim
(0,r),(0,r),(0,a).
\]
Thus one may interchange the roles of $a$ and $r$.

\smallskip
\noindent\emph{Replacing $r$ by $r-a$.}
By applying the inverse of the equivalence \textup{\ref{eq:E2}} repeatedly, we may replace a block of transpositions $(0,r)$ by a block
of equal size consisting of transpositions $(a,r)$:
\begin{align*}
(0,a),(0,a),(0,r),(0,r)
&\sim (0,a),(a,r),(0,a),(0,r)\\
&\sim (0,a),(a,r),(a,r),(0,a)\\
&\sim (0,a),(a,r),(0,r),(a,r)\\
&\sim (0,a),(0,a),(a,r),(a,r).
\end{align*}
Now apply the inverse of \textup{\ref{eq:E3}}: since
\[
(c(0,a)\cdots(0,a))^{-1}=c^{-1},
\]
conjugating $(a,r)$ by this product a total of $a$ times yields
$(0,r-a)$.  Hence we may replace $r$ by $r-a$.

These two operations realize the Euclidean algorithm, so after finitely many steps we arrive at the
constant tuple whose entries are all $(0,\gcd(a,r))$.  Since
\[
\gcd(a,r)=\gcd(k,b_1-a_1,b_2-a_2,\dots,b_n-a_n),
\]
this proves the claim.
\end{proof}

\enlargethispage{2\baselineskip}
We can now finish the proof of the proposition.

\begin{proof}[Proof of Proposition~\ref{prop:at-most-sigma0}]
By Lemma~\ref{lem:standardization}, every tuple of transpositions is equivalent to a standard tuple
\[
(\tau_1,\tau_2,\dots,\tau_{2n},(0,b_1),\dots,(0,b_m)).
\]
Suppose in addition that
\[
c\tau_1\tau_2\cdots\tau_{2n}(0,b_1)\cdots(0,b_m)
\]
is a $k$-cycle.  Since the first $2n$ transpositions occur in equal pairs, their product is the
identity, so this simplifies to
\[
c(0,b_1)\cdots(0,b_m) =(1\,2\,\dots\,b_1)\,
(b_1+1\,\dots\,b_2)\cdots
(b_{m-1}+1\,\dots\,b_m)\,
(0\,\,b_m+1\,\dots\,k-1).
\]
Thus, if $m>0$, this permutation is a product of at least two disjoint cycles, and hence cannot be a $k$-cycle. Therefore necessarily $m=0$. So, every tuple is equivalent to one of the form
\[
\bigl((a_1,b_1),(a_1,b_1),\dots,(a_n,b_n),(a_n,b_n)\bigr),
\]
and by Lemma~\ref{lem:euclidean-algorithm} this is equivalent to the constant tuple whose entries are
all $(0,d)$, where
\[
d=\gcd(k,b_1-a_1,\dots,b_n-a_n).
\]
Therefore each orbit is represented by one of the constant tuples
\[
\bigl((0,d),(0,d),\dots,(0,d),(0,d)\bigr)
\]
for some $d\mid k, d< k$. Thus, there are at most $\sigma_0(k)-1$ such possibilities, thus completing the proof.
\end{proof}
\end{document}